\documentclass{article}

\usepackage{amsmath, amssymb, amsfonts}
\usepackage{amsthm, thmtools, thm-restate}
\usepackage{mathtools, bm}

\usepackage{algorithm, algorithmicx, algpseudocode}

\usepackage{adjustbox, epsfig, graphicx, subfigure}
\usepackage{nicefrac}
\usepackage{bbold}
\usepackage[justification=centering]{caption}

\usepackage{times}

\usepackage{color}

\usepackage{upgreek}

\newtheorem{theorem}{Theorem}[section]
\newtheorem{proposition}[theorem]{Proposition}
\newtheorem{definition}[theorem]{Definition}
\newtheorem{lemma}[theorem]{Lemma}
\newtheorem{corollary}[theorem]{Corollary}

\declaretheoremstyle[
notefont=\bfseries, notebraces={}{},
bodyfont=\normalfont\itshape,
headformat=\NAME \NOTE
]{nopar}

\definecolor{puorange}{rgb}{0.80,0.20,0}
\definecolor{bluegray}{rgb}{0.04,0,0.7}
\definecolor{greengray}{rgb}{0.05,0.50,0.15}
\definecolor{darkbrown}{rgb}{0.40,0.2,0.05}
\definecolor{darkcyan}{rgb}{0,0.4,1}
\definecolor{black}{rgb}{0,0,0}
\definecolor{grey}{rgb}{0.93,0.93,0.93}

\usepackage[colorlinks,citecolor=bluegray,linkcolor=darkbrown,urlcolor=blue,breaklinks]{hyperref}

\newcommand{\reals}{{\mathbb R}}

\newcommand{\norm}[1] {\left \| #1 \right \|}

\newcommand{\idm}{\operatorname{I}}
\newcommand{\id}{\operatorname{I}}

\newcommand{\Tr}{\operatorname{\bf Tr}}

\newcommand{\Expect}{\operatorname{\mathbb E}}

\DeclareMathOperator*{\argmin}{arg\,min}

\newcommand{\bigO}{\mathcal{O}}

\newcommand{\red}[1]{\textcolor{puorange}{\bf #1}}

\def\shortdisplay{\setlength{\abovedisplayskip}{5pt}%
	\setlength{\belowdisplayskip}{5pt}%
	\setlength{\abovedisplayshortskip}{2pt}%
	\setlength{\belowdisplayshortskip}{2pt}}

\let\oldselectfont\selectfont
\def\selectfont{\oldselectfont\shortdisplay}

\newcommand{\horizon}{\tau}
\newcommand{\dyn}{\phi}

\newcommand{\Dyn}{\Phi}
\newcommand{\Ddyn}{\upphi}

\newcommand{\dimcmd}{\horizon p}

\newcommand{\stepsize}{\gamma}
\newcommand{\accstepsize}{\delta}

\newcommand{\valuefunc}{c}

\newcommand{\cst}{\mu}

\newcommand{\unknown}{\bar u}
\newcommand{\optimobj}{f}
\newcommand{\ctrlobj}{h}
\newcommand{\targetfunc}{\tilde x}
\newcommand{\regfunc}{g}

\usepackage{fullpage}
\usepackage{natbib}
\title{Iterative Linearized Control: \\ Stable Algorithms and Complexity Guarantees}
\date{June 10th, 2019}
\author{
	Vincent Roulet$^1$ \qquad
	Siddhartha Srinivasa$^2$ \qquad
	Dmitriy Drusvyatskiy$^3$ \qquad
	Zaid Harchaoui$^1$
	\\ \\
	$^1$ Department of Statistics, University of Washington \\
	$^2$ Paul G. Allen School of Computer Science and Engineering, University of Washington \\
	$^3$ Department of Mathematics, University of Washington \\
	\small{\texttt{ \{vroulet, zaid, ddrusv\}@uw.edu, \{siddh\}@cs.uw.edu}}
}
\begin{document}
	\maketitle
\begin{abstract}
We examine popular gradient-based algorithms for nonlinear control in the light of the modern complexity analysis of first-order optimization algorithms. 
The examination reveals that the complexity bounds can be clearly stated in terms of calls to a computational oracle related to dynamic programming 
and implementable by gradient back-propagation using machine learning software libraries such as PyTorch or TensorFlow. Finally, we propose a regularized Gauss-Newton algorithm enjoying worst-case complexity bounds and improved convergence behavior in practice. The software library based on PyTorch is publicly available. 
\end{abstract}

\section*{Introduction}
Finite horizon discrete time nonlinear control has been studied for decades, with applications ranging from spacecraft dynamics to robot learning~\citep{Bell67,whittle1982optimization,Bert05}. 
Popular nonlinear control algorithms, such as differential dynamic programming or iterative linear quadratic Gaussian algorithms, 
are commonly derived using a linearization argument relating the nonlinear control problem to a linear control problem~\citep{Todo03, Li07}. 

We examine nonlinear control algorithms based on iterative linearization techniques through the lens of the modern complexity analysis of first-order optimization algorithms. We first reformulate the problem as the minimization of an objective that is written as a composition of functions. Owing to this reformulation, we can frame several popular nonlinear control algorithms as first-order optimization algorithms applied to this objective. 

We highlight the equivalence of dynamic programming and gradient back-propagation in this framework and underline the central role of the corresponding automatic differentiation oracle in the complexity analysis in terms of convergence to a stationary point of the objective. We show that the number of calls to this automatic differentiation oracle is the relevant complexity measure given the outreach of machine learning software libraries such as PyTorch or TensorFlow~\citep{Pyto17, Tens15}.

Along the way we propose several improvements to the iterative linear quadratic regulator (ILQR) algorithm, resulting in an accelerated regularized Gauss-Newton algorithm enjoying a complexity bound in terms of convergence to a stationary point and displaying stable convergence behavior in practice. Regularized Gauss-Newton algorithms give a template for the design of algorithms based on partial linearization with guaranteed convergence~\citep{Bjor96,Burk85, nesterov2007modified, MR3511391, Drus16}. The proposed accelerated regularized Gauss-Newton algorithm is based on a Gauss-Newton linearization step stabilized by a proximal regularization and boosted by a Catalyst extrapolation scheme, potentially accelerating convergence while preserving the worst-case guarantee. 

\paragraph{Related work.}
Differential dynamic programming (DDP) and iterative linearization algorithms are popular algorithms for finite horizon discrete time nonlinear control~\citep{Tass14}.
DDP is based on approximating the Bellman equation at the current trajectory in order to use standard dynamic programming. Up to our knowledge, the complexity analysis of DDP has been limited; see~\citep{Mayn66, Jaco70, Todo03} for classical analyses of DDP. 

Iterative linearization algorithms such as the iterative linear quadratic regulator (ILQR) or the iterative linearized Gaussian algorithm (ILQG) linearize the trajectory in order to use standard dynamic programming~\citep{Li04, Todo05, Li07}. Again, the complexity analysis of ILQR for instance has been limited. In this paper, we refer to the definitions of ILQR and ILQG as given in the original papers~\citep{Li04, Todo05, Li07}, the same names have been then used for variants of those algorithms that use a roll-out phase on the true trajectory as in, e.g., \citep{tassa2012synthesis} where line-searches were proposed. Line-searches akin to the Levenberg-Marquardt method were proposed but without convergence rates~\citep{Todo05}.
It is worthwhile to mention related approaches in the nonlinear model predictive control area~\citep{grune2017nonlinear, richter2012computational, dontchev2018inexact}.

We adopt the point of view of the complexity theory of first-order optimization algorithms. The computation of a Gauss-Newton step (or a Newton step) through dynamic programming for nonlinear control problems is classical; see~\citep{whittle1982optimization, Dunn89, Side05}. However, while the importance of the addition of a proximal term in Gauss-Newton algorithms is now well-understood~\citep{nesterov2007modified}, several popular nonlinear control algorithms involving such steps, such as ILQR, have not been revisited yet~\citep{Li04}. Our work shows how to make these improvements. 

We also show how gradient back-propagation, \textit{i.e.}, automatic differentiation~\citep{Grie08}, a popular technique usually derived using either a chain rule argument or a Lagrangian framework~\citep{Bert05,lecun1988theoretical}, allows one to solve the dynamic programming problems arising in linear quadratic control. Consequently, the subproblems that arise when using iterative linearization for nonlinear control can be solved with calls to an automatic differentiation oracle implementable in PyTorch or TensorFlow~\citep{Tens15, Pyto17, Kaka18}.

The regularized Gauss-Newton method was extensively studied to minimize the nonlinear least squares objectives arising in inverse problems~\citep{Bjor96, Noce06, Kalt08, Hans13}. The complexity-based viewpoint used in~\citep{nesterov2007modified, cartis2011evaluation, Drus16} informs our analysis and offers generalizations to locally Lipschitz objectives. We build upon these results in particular when equipping the proposed regularized Gauss-Newton algorithm with an extrapolation scheme in the spirit of~\citep{Paqu17}.

All notations are presented in Appendix~\ref{app:notations}. The code for this project is available at \url{https://github.com/vroulet/ilqc}.

\section{Discrete time control}\label{sec:opt_ctrl_setting}
We first present the framework of finite horizon discrete time nonlinear control.  

\paragraph{Exact dynamics.}
Given state variables $x\in \reals^d$ and control variables $u\in \reals^p$, we consider the control of finite trajectories $\bar x = (x_1;\ldots;x_\horizon) \in \reals^{\horizon d}$ of horizon $\horizon$ whose dynamics are controlled by a command $\bar u = (u_0;\ldots;u_{\horizon-1}) \in \reals^{\horizon p}$, through 
\begin{equation}\label{eq:dynamic}
x_{t+1} = \dyn_t(x_t,u_t), \qquad \mbox{for $t=0,\ldots,\horizon-1$,}
\end{equation}
starting from a given $\hat x_0 \in \reals^d$, where the functions $\dyn_t: \reals^d \times \reals^p \rightarrow \reals^d$ are assumed to be differentiable.

Optimality is measured through convex costs $h_t$, $g_t$, on the state and control variables $x_t$, $u_t$ respectively, defining the discrete time nonlinear control problem
\begin{align}
\label{eq:opt_control_pb}
\min_{\substack{x_0,\ldots,x_\horizon \in \reals^d\\ u_0,\ldots,u_{\horizon-1} \in \reals^{p}}} \quad & \sum_{t=1}^{\horizon} h_t(x_t) + \sum_{t=0}^{\horizon-1}g_t(u_t) \\
\mbox{subject to} \quad & x_{t+1} = \dyn_t(x_t,u_t), \qquad x_0 = \hat x_0, \nonumber
\end{align}
where, here and thereafter, the dynamics must be satisfied for $t=0,\ldots, \horizon-1$.
\paragraph{Noisy dynamics.}
The discrepancy between the model and the dynamics can be taken into account by considering noisy dynamics as
\begin{equation}\label{eq:noisydynamic}
x_{t+1} = \dyn_t(x_t, u_t, w_t),
\end{equation}
where $w_t \sim \mathcal{N}(0, \idm_q)$ for $t=0,\ldots,\horizon-1$.
The resulting discrete time control problem consists of optimizing the average cost under the noise  $\bar w = (w_0; \ldots; w_{\horizon-1})$ as
\begin{align}
\label{eq:noisy_opt_control_pb}
\begin{split}
\min_{\substack{x_0,\ldots,x_\horizon \in \reals^d\\ u_0,\ldots,u_{\horizon-1} \in \reals^{p}}} \quad & \Expect_{\bar w}\left[\sum_{t=1}^{\horizon} h_t(x_t)\right] + \sum_{t=0}^{\horizon-1}g_t(u_t) \\
\mbox{subject to} \quad & x_{t+1} = \dyn_t(x_t,u_t, w_t), \qquad  x_0 = \hat x_0.
\end{split}
\end{align}

\paragraph{Costs and penalties.}
The costs on the trajectory can be used to force the states to follow a given orbit $\hat x_1, \ldots, \hat x_\horizon$ as 
\begin{equation}\label{eq:quad_state_costs}
h_t(x_t) = \frac{1}{2} (x_t - \hat x_t)^\top Q_t (x_t - \hat x_t), \quad \mbox{with} \quad Q_t \succeq 0,
\end{equation}
which gives a quadratic tracking problem, while the regularization
penalties on the control variables are typically quadratic functions
\begin{equation}\label{eq:quad_control_costs}
g_t(u_t) = \frac{1}{2} u_t^\top R_t u_t, \quad \mbox{with} \quad R_t \succ 0.
\end{equation}
The regularization penalties can also encode constraints on the control variable such as the indicator function of a box
\begin{equation}\label{eq:constraint_costs}
g_t(u_t)  = \iota_{\{u : c_t^-  \leq u \leq  c_t^+ \}}(u_t), \quad \mbox{with} \quad c_t^-, c_t^+ \in \reals^p,
\end{equation}
where $\iota_S$ denotes the indicator function of a set $S$. 

\paragraph{Iterative Linear Control algorithms.}
We are interested in the complexity analysis of algorithms such as the iterative linear quadratic regulator (ILQR) algorithm as defined in~\citep{Li04,Todo05, Li07}, used for exact dynamics, which iteratively computes the solution of 
\begin{align}
\min_{\substack{y_0, \ldots y_\horizon \in \reals^d \\ v_0,\ldots, v_{\horizon-1} \in \reals^p} } &
\sum_{t=1}^{\horizon} q_{h_t}(x_t^{(k)} + y_t) + \sum_{t=0}^{\horizon-1} q_{g_t}(u_t^{(k)}+v_t) \label{eq:LQR_gauss_newton} \\
\mbox{\textup{subject to}} \quad & y_{t+1} = \ell_{\phi_t}(y_t, v_t), \qquad  y_0 = 0, \nonumber
\end{align}
where $\bar u^{(k)}$ is the current command, $\bar x^{(k)}$ is the corresponding trajectory given by~\eqref{eq:dynamic}, $q_{h_t}, q_{g_t}$ are quadratic approximations of the costs $h_t, g_t$ around respectively $x_t^{(k)}, u_t^{(k)}$ and  $\ell_{\phi_t}$ is the linearization of $\dyn_t$ around $(x_t^{(k)}, u_t^{(k)})$. The next iterate is then given by $\bar u^{ (k+1)} = \bar u^{(k)} + \alpha \bar v^*$ where $\bar v^*$ is the solution of~\eqref{eq:LQR_gauss_newton} and $\alpha$ is a step-size given by a line-search method. To understand this approach, we frame the problem as the minimization of a composition of functions.

Note that the term ILQR or ILQG has then been used to refer to a variant of the above algorithm that uses the feedback gains computed in the resolution of the linear control problem to control to move along the true trajectory, see~\citep{tassa2012synthesis}.

\paragraph{Formulation as a composite optimization problem.}
We call an optimization problem a \emph{composite optimization problem} if it consists in the minimization of a composition of functions.
For a fixed command $\bar u  \in\reals^{\horizon p}$, denote by $\tilde x(\bar u) = (\tilde x_1(\bar u); \ldots; \tilde x_\horizon(\bar u)) \in \reals^{\horizon d}$ the trajectory given by the exact dynamics, which reads 
\begin{align}\label{eq:traj_func}
\tilde x_1(\bar u) & = \dyn_0(\hat x_0, u_0), & 
\tilde x_{t+1}(\bar u) & = \dyn_t(\tilde x_t(\bar u),u_t).
\end{align}
Similarly denote by $\tilde x(\bar u, \bar w) \in \reals^{ \horizon d}$ the trajectory in the noisy case.
Denoting the total cost by $h(\bar x) = \sum_{t=1}^{\horizon} h_t(x_t)$, the total penalty by $g(\bar u) = \sum_{t=0}^{\horizon-1}g_t(u_t)$, the  control problem~\eqref{eq:opt_control_pb} with exact dynamics reads
\begin{align}\label{eq:composite_pb}
\min_{\bar u \in \reals^{\horizon p}} \quad & f(\bar u) \triangleq h(\tilde x(\bar u)) + g(\bar u),
\end{align}
and with noisy dynamics, 
\begin{align}\label{eq:averaged_composite_pb}
\min_{\bar u \in \reals^{\horizon p}} \quad & f (\bar u) \triangleq \Expect_{\bar w}\left[h(\tilde x(\bar u, \bar w ))\right] + g(\bar u),
\end{align}
\textit{i.e.}, we obtain a composite optimization problem whose structure can be exploited to derive oracles on the objective.

\section{Oracles in discrete time control}\label{sec:oracles_ctrl}
We adopt here the viewpoint of the complexity theory of first-order optimization. 
Given the composite problem~\eqref{eq:composite_pb}, what are the relevant oracles and what are the complexities of calls to these oracles? 
We first consider exact dynamics $\dyn_t$ of the form~\eqref{eq:dynamic} and unconstrained cost penalties such as~\eqref{eq:quad_control_costs}.

\subsection{Exact and unconstrained setting}
\paragraph{Model minimization.}
Each step of the optimization algorithm is defined by the minimization of a regularized model of the objective. 
For example, a \emph{gradient step} on a point $\bar u $ with step-size $\stepsize$ corresponds to linearizing both $h$ and $\tilde x$ and defining the linear model
\begin{equation*}
\ell_f(\bar u  + \bar v; \bar u) = \ell_h\left(\tilde x (\bar u) + \nabla \tilde x(\bar u)^\top\bar v; \tilde x(\bar u)\right) + \ell_g(\bar u + \bar v ; \bar u)
\end{equation*}
of the objective $f$, where $\ell_h(\bar x + \bar y; \bar x) = h(\bar x) + \nabla h(\bar x)^\top\bar y$ and $\ell_g(\bar u + \bar v; \bar u)$ is defined similarly.  
Then, this model with a proximal regularization is minimized in order to get the next iterate
\begin{align}\label{eq:grad_step}
\bar u^+ & = \bar u + \argmin_{\bar v \in \reals^{\horizon p}}\left\{\ell_f(\bar u + \bar v ; \bar u) + \frac{1}{2\gamma}\|\bar v\|_2^2\right\}.
\end{align}

Different models can be defined to better approximate the objective. For example, if only the mapping $\tilde x$ is linearized, this corresponds to defining the convex model at a point $\bar u$
\begin{equation}\label{eq:cvx_model}
c_f(\bar u + \bar v;\bar u) =  h\left(\tilde x(\bar u) + \nabla \tilde x(\bar u)^\top\bar v \right) +  g(\bar u + \bar v).
\end{equation}
We get then a \emph{regularized Gauss-Newton step} on a point $\bar u \in \reals^{\dimcmd}$ with step size $\stepsize > 0$ as
\begin{align}\label{eq:prox_linear_step}
\bar u^+ = \bar u + \argmin_{\bar v \in \reals^{\horizon p}} \left\{c_f(\bar u + \bar v;\bar u) +\frac{1}{2\gamma} \|\bar v\|_2^2\right\}.
\end{align} 

Although this model better approximates the objective, its minimization may be computationally expensive for general functions $h$ and $g$. 
We can use a quadratic approximation of $h$ around the current mapping $\tilde x (\bar u)$ and linearize the trajectory around $\bar u$ which defines the quadratic model 
\begin{equation}\label{eq:quad_model}
q_f(\bar u + \bar v ; \bar u) = q_h\left(\tilde x (\bar u) + \nabla \tilde x(\bar u)^\top \bar v; \tilde x(\bar u)\right) + q_g(\bar u + \bar v; \bar u),
\end{equation}
where $q_h(\bar x + \bar y; \bar x) \triangleq h(\bar x) + \nabla h(\bar x)^\top \bar y + \bar y^\top \nabla^2 h(\bar x) \bar y/2$ and  $q_g(\bar u + \bar v; \bar u)$ is defined similarly. 
A \emph{Levenberg-Marquardt step} with step-size $\gamma$ consists in minimizing the model~\eqref{eq:quad_model} with a proximal regularization
\begin{align}\label{eq:LM_step}
\bar u^+ = \bar u + \argmin_{\bar v \in \reals^{\horizon p}} \left\{ q_f(\bar u + \bar v;\bar u) +\frac{1}{2\gamma} \|\bar v\|_2^2\right\}.
\end{align} 

\paragraph{Model-minimization steps by linear optimal control.}
Though the chain rule gives an analytic form of the gradient, we can use the definition of a gradient step as an optimization sub-problem to understand its implementation.
Formally, the above steps~\eqref{eq:grad_step}, \eqref{eq:prox_linear_step}, \eqref{eq:LM_step}, define a model $m_f$ of the objective $f$  in~\eqref{eq:composite_pb} on a point $\bar u$, as
\[
f(\bar u + \bar v) \approx m_f(\bar u + \bar v; \bar u)=m_h\left(\tilde x (\bar u){+}\nabla \tilde x(\bar u)^\top \bar v; \tilde x(\bar u)\right) + m_g(\bar u + \bar v; \bar u),
\]
where $m_h = \sum_{t=1}^{\horizon} m_{h_t}$, $m_g = \sum_{t=0}^{\horizon-1} m_{g_t}$ are models of $h$ and $g$ respectively, composed of models on the individual variables. The model-minimization step with step-size $\stepsize$, 
\begin{equation}\label{eq:model_min}
\bar u^+ = \bar u + \argmin_{\bar v \in \reals^{\horizon p}} \left\{m_f(\bar u+ \bar v; \bar u)  + \frac{1}{2\gamma}\|\bar v\|_2^2\right\},
\end{equation} 
amounts then to a linear control problem as shown in the following proposition.

\begin{proposition}\label{prop:LQR_step}
	The model-minimization step~\eqref{eq:model_min} for control problem~\eqref{eq:opt_control_pb} written as~\eqref{eq:composite_pb} is given by $\bar u^+ = \bar u + \bar v^*$ where $\bar v^* = (v_0^*;\ldots, v_{\horizon-1}^*)$ is the solution of
	\begin{align}
	\min_{\substack{y_0, \ldots y_\horizon \in \reals^d \\ v_0,\ldots, v_{\horizon-1} \in \reals^p} } &
	\sum_{t=1}^{\horizon} m_{h_t}(x_t{ + }y_t; x_t)  + \sum_{t=0}^{\horizon-1}  m_{g_t}(u_t{+}v_t; u_t) + \frac{1}{2\stepsize}\|v_t\|_2^2 \nonumber
	\\
	\mbox{\textup{subject to}} \quad & y_{t+1} = \Dyn_{t,x}^\top y_t + \Dyn_{t,u}^\top v_t, \label{eq:lin_opt_ctrl} \qquad y_0 = 0, 
	\end{align}
	where $\Dyn_{t,x}{=}\nabla_x \dyn_t(x_t, u_t)$, $\Dyn_{t,u}{ =} \nabla_u \dyn_t(x_t, u_t)$ and $x_t = \tilde x_t(\bar u)$.
\end{proposition}
\begin{proof}
	Recall that the trajectory defined by $\bar u$ reads
	\begin{align*}
	\tilde x_1(\bar u) & = \dyn_0(\hat x_0, F_0^\top \bar u), &
	\tilde x_{t+1}(\bar u) & = \dyn_t(\tilde x_t(\bar u), F_t^\top \bar u),
	\end{align*}
	where $F_t = e_{t+1} \otimes \idm_p \in \reals^{\horizon p\times p} $, $e_t\in \reals^\horizon$ is the $t$\textsuperscript{th} canonical vector in $\reals^\horizon$, such that $F_t^\top \bar u = u_t$. The gradient reads $	\nabla \tilde x_1 (\bar u)  = F_0\nabla_u\dyn_0(x_0, u_0) $ followed by
	\begin{align*}
	\nabla \tilde x_{t+1}(\bar u) & = \nabla \tilde x_t(\bar u) \nabla_x \dyn_t(x_t, u_t) + F_t \nabla_u \dyn_t (x_t, u_t),
	\end{align*}
	where $x_t = \tilde x_t(\bar u)$ and $x_0= \hat x_0$. For a given $\bar v = (v_0; \ldots; v_{\tau-1})$, the product $\bar y = (y_1; \ldots; y_{\horizon}) = \nabla \tilde x (\bar u)^\top \bar v$ reads $y_1  = \nabla_u \dyn_0(x_0, u_0)^\top v_0$ followed by 
	\begin{align*}
	y_{t+1} & = \nabla_x \dyn_t(x_t, u_t)^\top y_t +\nabla_u \dyn_t (x_t, u_t)^\top v_t,
	\end{align*}
	where we used that $y_t =  \nabla \tilde x_t(\bar u) ^\top \bar v$. Plugging this into~\eqref{eq:model_min} gives the result.
\end{proof}

\paragraph{Dynamic programming.}
If the models used in~\eqref{eq:model_min} are linear or quadratic, the resulting linear control problems~\eqref{eq:lin_opt_ctrl} can be solved efficiently using dynamic programming, \textit{i.e.}, with a linear cost in $\horizon$, as presented in the following proposition. The cost is $\bigO(\horizon p^3 d^3)$. Details on the implementation for quadratic costs are provided in Appendix~\ref{app:dyn_prog}. 

Since the leading dimension of the discrete time control problem is the length of the trajectory $\horizon$, all of the above optimization steps have roughly the same cost. This means that, in discrete time control problems,~\emph{second order steps such as~\eqref{eq:LM_step} are roughly as expensive as gradient steps}.

\begin{proposition}
	Model-minimization steps of the form~\eqref{eq:model_min} for discrete time control problem~\eqref{eq:opt_control_pb} written as~\eqref{eq:composite_pb} with linear or quadratic convex models $m_h$  and  $m_g$ can be solved in linear time with respect to the length of the trajectory $\horizon$ by dynamic programming.
\end{proposition}
The proof of the proposition relies on the dynamic programming approach explained below.
The linear optimal control problem~\eqref{eq:lin_opt_ctrl} can be divided into smaller subproblems and then solved recursively. 
Consider the linear optimal control problem~\eqref{eq:lin_opt_ctrl} as
\begin{align}\label{eq:opt_ctrl_gen}
\min_{\substack{y_1, \ldots, y_{\horizon} \quad \\ v_0, \ldots, v_{\horizon-1}}} &
\sum_{t=1}^\horizon q_{h_t} (y_t) + \sum_{t=0}^{\horizon-1} q_{g_t} (v_t)\\
\mbox{subject to} \quad & y_{t +1} = \ell_{t}(y_{t}, v_{t}), \qquad  y_{0} = 0,\nonumber
\end{align}
where $\ell_t$ is a linear dynamic in state and control variables, $q_{g_t}$ are strongly convex quadratics and $q_{h_t}$ are convex quadratic or linear functions. 
For $0\leq t\leq \horizon$, given  $\hat y_t$, define the cost-to-go from $\hat y_t$, as the solution of 
\begin{align}\label{eq:cost-to-go}
\valuefunc_{t}(\hat y_t) = \underset{\substack{y_t,\ldots,y_\horizon\\v_t,\ldots,v_{\horizon-1}}}{\mbox{min}}\quad  &
  \sum_{t'=t}^\horizon q_{h_{t'}} (y_{t'})  + \sum_{t'=t}^{\horizon-1} q_{g_{t'}} (v_{t'})\\
\mbox{subject to} \quad &   y_{t'+1} = \ell_{t'}(y_{t'},v_{t'}), \quad \mbox{for $t'=t \ldots, \horizon-1$} \nonumber\\
&  y_t = \hat y_t.\nonumber
\end{align}
The cost-to-go functions can be computed recursively by the Bellman equation for $t\in \{\horizon-1,\ldots,0\}$,
\begin{equation}
\valuefunc_t(\hat y_t)=q_{h_{t}}(\hat y_t) +  \min_{v_t} \left\{ q_{g_{t}}(v_t) + \valuefunc_{t+1}(\ell_t(\hat y_t,v_t))\right\}  \label{eq:Bellman}
\end{equation}
solved for $ v^*_t(\hat y_t){ = }\argmin_{v_t} \left\{ q_{g_{t}}(v_t){ +} \valuefunc_{t+1}(\ell_t(\hat y_t,v_t))\right\}$. The final cost initializing the recursion is defined as $\valuefunc_\horizon(\hat y_\horizon) = q_{h_{\horizon}}(\hat y_\horizon)$. For quadratic costs and linear dynamics, the problems defined in~\eqref{eq:Bellman} are themselves quadratic problems that can be solved analytically to get an expression for $\valuefunc_{t}$.

The solution of~\eqref{eq:opt_ctrl_gen} is given by computing $\valuefunc_0(0)$, which amounts to iteratively solving the Bellman equations starting from $\hat y_0 = 0$. Formally, starting form $t=0$ and $\hat y_0 = 0$, it iteratively gets the optimal control $v_t^*$  at time $t$ defined by the analytic form of the cost-to-go function  and moves along the dynamics to get the corresponding optimal next state,
\begin{align}\label{eq:roll_out}
v^*_t &= v^*_t(y_t),  &
y_{t+1} &= \ell_t(y_t,v^*_t).
\end{align}
The cost of the overall dynamic procedure that involves a \emph{backward} pass to compute the cost-to-go functions and a \emph{roll-out} pass to compute the optimal controls is therefore linear in the length of the trajectory $\horizon$. The main costs lie in solving quadratic problems in the Bellman equation~\eqref{eq:Bellman} which only depend on the state and control dimensions $d$ and $p$.

\paragraph{Gradient back-propagation as dynamic programming.}
We illustrate the derivations for a gradient step in the following proposition that shows a cost of $\bigO(\horizon (p d + d^2))$. We recover the well-known gradient back-propagation algorithm used to compute the gradient of the objective. The dynamic programming viewpoint provides here a natural derivation.
\begin{proposition}\label{prop:grad}
	A gradient step~\eqref{eq:grad_step} for discrete time control problem~\eqref{eq:opt_control_pb} written as~\eqref{eq:composite_pb} and solved by dynamic programming amounts to
	\begin{enumerate}
		\item a \emph{forward} pass that computes the derivatives $\nabla_x \dyn_t(x_t, u_t)$, $\nabla_u \dyn_t(x_t, u_t)$, $\nabla h_t(x_t)$, $ \nabla g_t(u_t)$ for $t=0,\ldots, \horizon$ along the trajectory  given by $x_{t+1} = \dyn_t(x_t,u_t)$ for $t=0,,\ldots, \horizon-1$, 
		\item a \emph{backward} pass that computes linear cost-to-go functions as $\valuefunc_t(y_t) = \lambda_t^\top y_t + \mu_t$ where $\lambda_\horizon = \nabla h_\horizon (x_\horizon)$,
		{$\lambda_t = \nabla h_t(x_t) +  \nabla_x \dyn_t(x_t, u_t) \lambda_{t+1}$}, for $t=\horizon-1, \ldots 0$,
		\item a \emph{roll-out} pass that outputs 
		$ v^*_t =  - \gamma(\nabla_u \dyn_t(x_t, u_t) \lambda_{t+1} + \nabla g_t(u_t))$, for $t= 0,\ldots \horizon-1$.
	\end{enumerate}
\end{proposition}
\begin{proof}
	Recall that a gradient step is given as $ \bar u^+ = \bar u + \bar v^*$ where $\bar v^*$ is the solution of 
	\[
	\min_{\bar v \in \reals^{\horizon p}} \ell_h\left(\tilde x (\bar u) + \nabla \tilde x(\bar u)^\top\bar v; \tilde x(\bar u)\right) + \ell_g(\bar u + \bar v ; \bar u) + \frac{1}{2\stepsize}\|\bar v\|_2^2.
	\]
	where $\ell_h(\bar x + \bar y; \bar x) = h(\bar x) + \nabla h(\bar x)^\top\bar y$ and $\ell_g(\bar u + \bar v; \bar u)$ is defined similarly.  
	From Prop.~\ref{prop:LQR_step}, we get that it amounts to a linear optimal control problem of the form
	\begin{equation}\label{eq:linear_ctrl}
	\begin{split}
	\min_{\substack{y_{0}, \ldots, y_{\horizon} \\v_0, \ldots, v_{\horizon-1}}}\quad &
	\sum_{t=1}^{\horizon} y_t^\top a_t + \sum_{t=0}^{\horizon-1}v_t^\top b_t +  \frac{1}{2\stepsize}\sum_{t=0}^{\horizon-1} \|v_t\|^2\\
	\mbox{subject to} \quad & y_{t+1} = \Dyn_{t,x}^\top y_t + \Dyn_{t,u}^\top v_t,\qquad  y_0 = 0,
	\end{split}
	\end{equation}
	where  
	$a_t = \nabla h_t(x_t)$,
	$b_t = \nabla g_t(u_t)$,
	$\Dyn_{t,x} = \nabla_x \dyn_t(x_t , u_t)$, 
	$\Dyn_{t,u} = \nabla_u \dyn_t(x_t , u_t)$ and $x_t = \tilde x_t(\bar u)$. The definition of the linear problem is the \emph{forward} pass. 
	
	When solving~\eqref{eq:linear_ctrl} with dynamic programming, cost-to-go functions are linear, $\valuefunc_t (y) = \lambda_t^\top y + \cst_t$.
	Recursion starts with $\lambda_\horizon= a_\horizon$, $\cst_\horizon = 0$. Then, assuming $\valuefunc_{t+1} (y) =\lambda_{t+1}^\top y + \cst_{t+1}$ for $t \in \{\horizon-1, \ldots, 0\}$, we get 
	\begin{align}
	\valuefunc_t(y) & = a_t^\top y + \min_{v\in \reals^p} \left\{ b_t^\top v + \lambda_{t+1}^\top (\Dyn_{t,x}^\top y + \Dyn_{t,u}^\top v) + \frac{1}{2\stepsize}\|v\|_2^2 \right\} + \cst_{t+1} \label{eq:bellman_linear} \\
	& = (a_t + \Dyn_{t,x} \lambda_{t+1})^\top y + \cst_{t+1} + \frac{\gamma}{2}\|b_t+\Dyn_{t,u} \lambda_{t+1}\|_2^2, \nonumber
	\end{align}
	and so we identify $\lambda_t = a_t + \Dyn_{t,x} \lambda_{t+1}$ and $\cst_t =\cst_{t+1} + \frac{\gamma}{2}\|b_t+ \Dyn_{t,u}\lambda_{t+1}\|_2^2$ that define the cost-to-go function at time $t$. This defines the \emph{backward} pass.
	
	The optimal control variable at time $t$ is then independent of the starting state and reads from~\eqref{eq:bellman_linear}, 
	\[
	v^*_t = - \stepsize (\Dyn_{t,u} \lambda_{t+1} + b_t).
	\] 
	This defines the \emph{roll-out} pass.
\end{proof}

\subsection{Noisy or constrained settings}

\paragraph{Noisy dynamics.}
For inexact dynamics defining the problem~\eqref{eq:averaged_composite_pb}, we consider a Gaussian approximation of the linearized trajectory around the exact current trajectory. Formally, the Gaussian approximation of the random linearized trajectory
$
\ell_{\tilde x} (\bar u + \bar v; \bar u, \bar w) = \tilde x(\bar u, \bar w) + \nabla_{\bar u}\tilde x(\bar u, \bar w)^\top \bar v
$
around the exact linearized trajectory given for $\bar w= 0$ reads
\begin{align*}
\hat \ell_{\tilde x} (\bar u + \bar v; \bar u, \bar w) 
= & \tilde x(\bar u, 0) + \nabla_{\bar u}\tilde x(\bar u, 0)^\top \bar v + \nabla_{\bar w} \tilde x(\bar u, 0)^\top \bar w  + \nabla_{\bar u \bar w}^2 \tilde x(\bar u, 0)[\bar v, \bar w, \cdot],
\end{align*}
which satisfies 
$
\Expect_{\bar w}[\hat \ell_{\tilde x} (\bar u+\bar v; \bar u, \bar w)]  =\tilde x(\bar u,0)+\nabla_{\bar u}\tilde x(\bar u,0)^\top \bar v
$,   see Appendix~\ref{app:notations} for gradient and tensor notations.

The model we consider for the state cost is then of the form
\begin{align}\label{eq:approx_gaussian}
m_f(\bar u + \bar v; \bar u) = \Expect_{\bar w}\left[m_h\left(\hat \ell_{\tilde x}(\bar u + \bar v; \bar u, \bar w);  \tilde x(\bar u, 0)\right)\right] + m_g(\bar u  + \bar v; \bar u).
\end{align}
For simple dynamics $\phi_t$, their minimization with an additional proximal term amounts to a linear quadratic Gaussian control problem as stated in the following proposition.
\begin{proposition}
	Assume $\nabla^2_{xx} \phi_t$, $\nabla _{xw}^2 \phi_t$ and $\nabla_{ux}^2 \phi_t$ to be zero. The model minimization step~\eqref{eq:model_min} for model~\eqref{eq:approx_gaussian} is given by $\bar u^+ = \bar u+ \bar v^*$ where $\bar v^*$ is the solution of 
\begin{align}
	&\min_{\substack{\bar y, \bar v} } \quad  
	\sum_{t=1}^{\horizon} \Expect_{\bar w}\left[m_{h_t}(x_t + y_t; x_t)\right]  + \sum_{t=0}^{\horizon-1}  m_{g_t}(u_t+v_t; u_t) + \frac{1}{2\stepsize}\|v_t\|_2^2 \nonumber\\
	&\mbox{\textup{s.t.}} \quad  y_{t+1} = \Dyn_{t,x}^\top y_t + \Dyn_{t,u}^\top v_t + \Dyn_{t, w}^\top w_t  + \Ddyn_{t,u, w}[ v_t, w_t, \cdot], \nonumber\\
	&\phantom{\mbox{\textup{s.t}} \quad \quad \: } y_0 = 0,  \label{eq:LQG}
\end{align}
where $\Dyn_{t,x} = \nabla_x \dyn_t(x_t, u_t, 0)$, $\Dyn_{t,u} = \nabla_u \dyn_t(x_t, u_t, 0)$,  $\Dyn_{t,w} = \nabla_w \dyn_t(x_t, u_t, 0)$,  $\Ddyn_{t, u, w} =\nabla_{u w}^2 \dyn_t(x_t, u_t, 0) $, $x_t =  \tilde x_t(\bar u, 0)$.
\end{proposition}
\begin{proof}
	The Gaussian approximation $\hat \ell_{\tilde x} (\bar u + \bar v; \bar u, \bar w)$ can be decomposed as in Prop.~\ref{prop:LQR_step}. Recall that the trajectory reads
	\begin{align*}
	\tilde x_1(\bar u, \bar w) = \phi_0(\hat x_0, F_0\bar u, G_0\bar w), \qquad \tilde x_{t+1} = \phi_t(\tilde x_t(\bar u, \bar w), F_t\bar u, G_t\bar w),
	\end{align*}
	where $F_t = e_{t+1} \otimes \idm_{p}$, $G_t = e_{t+1} \otimes \idm_q$ and $e_t$ is the $t$\textsuperscript{th} canonical vector in $\reals^\horizon$, such that $F_t^\top \bar u = u_t$ and $G_t^\top \bar w  = w_t$.
	We have then 
	\begin{align}
	\nabla_{\bar u} \tilde x_1(\bar u, \bar w) = & F_0 \nabla_u\phi_0(\hat x_0, F_0^\top\bar u, G_0^\top\bar w) \nonumber \\
	\nabla_{\bar u} \tilde x_{t+1}(\bar u, \bar w) = & \nabla_{\bar u}\tilde x_t(\bar u, \bar w) \nabla_x\phi_t(\tilde x_t(\bar u, \bar w), F_t^\top\bar u, G_t^\top\bar w) + F_t \nabla_u\phi_t(\tilde x_t(\bar u, \bar w), F_t^\top\bar u, G_t^\top\bar w) \nonumber\\
	\nabla_{\bar w} \tilde x_1(\bar u, \bar w) = & G_0 \nabla_w\phi_0(\hat x_0, F_0^\top\bar u, G_0^\top\bar w) \\
	\nabla_{\bar w} \tilde x_{t+1}(\bar u, \bar w) = & \nabla_{\bar w}\tilde x_t(\bar u, \bar w) \nabla_x\phi_t(\tilde x_t(\bar u, \bar w), F_t^\top\bar u, G_t^\top\bar w) + G_t \nabla_w\phi_t(\tilde x_t(\bar u, \bar w), F_t^\top\bar u, G_t^\top\bar w) \nonumber
	\end{align}
	Finally denoting for clarity $\tilde x = \tilde x (\bar u, \bar w)$ and $\phi_t = \phi_t(\tilde x_t, u_t, w_t)$,
	\begin{align*}
	\nabla_{\bar u \bar w}^2 \tilde x_1 = &  \nabla_{uw}^2\phi_0[F_0^\top, G_0^\top, \cdot] \\
	\nabla_{\bar u \bar w}^2 \tilde x_{t+1} = &  \nabla_{uw}^2 \tilde x_{t}[ \cdot, \cdot, \nabla_x \phi_t] + \nabla^2_{xx} \phi_t[\nabla_{\bar u} \tilde x_t^\top, \nabla_{\bar w}\tilde x_t^\top, \cdot]  + \nabla _{xw}^2\phi_t[\nabla_{\bar u} \tilde x_t^\top, G_t^\top, \cdot] \\
	& + \nabla_{ux}^2\phi_t[F_t^\top, \nabla_{\bar w}\tilde x_t^\top, \cdot]  + \nabla_{ u  w}^2\phi_t[F_t^\top, G_t^\top, \cdot]
	\end{align*}
	Denote $a = \nabla_{\bar u}\tilde x(\bar u, 0)^\top \bar v$, $b = \nabla_{\bar w} \tilde x(\bar u, 0)^\top \bar w$ and $ c= \nabla_{\bar u \bar w}^2 \tilde x(\bar u, 0)[\bar v, \bar w, \cdot]$, with $a, b, c \in \reals^{\horizon d}$. Those can be decomposed as, e.g., $a = (a_1; \ldots; a_\horizon)$ with $a_t = \nabla_{\bar u}\tilde x_t(\bar u, 0)^\top \bar v$ and we denote similarly $b_t = \nabla_{\bar w} \tilde x_t(\bar u, 0)^\top \bar w$, $c_t = \nabla_{\bar u \bar w}^2 \tilde x_t(\bar u, 0)[\bar v, \bar w, \cdot]$ the decomposition of $b$ and $c$ in $\horizon$ slices.
	Assuming $\nabla^2_{xx} \phi_t$ $\nabla _{xw}^2 \phi_t$ and $\nabla_{ux}^2 \phi_t$ to be zero,	we get as in Prop.~\ref{prop:LQR_step},
	\begin{equation*}
	\begin{array}{lll}
	a_{1} \: \: \:\,  = \Dyn_{0,u}^\top v_0 & 	b_{1}\: \: \:\,  = \Dyn_{0, w}^\top w_0 &	c_1\: \: \:\,  = \Ddyn_{0, u, w}[v_0, w_0, \cdot] \\ 
	a_{t+1}  = \Dyn_{t, x}^\top a_t + \Dyn_{t, u}^\top v_t & 	b_{t+1}  = \Dyn_{t, x}^\top b_t + \Dyn_{t, w}^\top w_t & 	c_{t+1}  =  \Dyn_{t, x}^\top c_t + \Ddyn_{t, u, w}[v_t, w_t, \cdot]
	\end{array}
	\end{equation*}
	where $\Dyn_{t, x} = \nabla_x \dyn_t(x_t, u_t, 0)$, $\Dyn_{t, u} = \nabla_u \dyn_t(x_t, u_t, 0)$, $\Dyn_{t, w} = \nabla_w \dyn_t(x_t, u_t, 0)$, $\Ddyn_{t, u,w} = \nabla_{uw}^2 \dyn_t(x_t, u_t, 0)$ and $x_t = \tilde x_t(\bar u, 0)$.
	Therefore the variable $y = a + b + c = \nabla_{\bar u}\tilde x(\bar u, 0)^\top \bar v + \nabla_{\bar w} \tilde x(\bar u, 0)^\top \bar w + \nabla_{\bar u \bar w}^2 \tilde x(\bar u, 0)[\bar v, \bar w, \cdot]$ decomposed as $y = (y_1; \ldots; y_\horizon)$ satisfies
	\begin{align*}
	y_1 &  = \Dyn_{0, u}^\top v_0 +  \Dyn_{0, w}^\top w_0 +  \Ddyn_{0, u, w}[v_0, w_0, \cdot] \\
	y_{t+1}& = \Dyn_{t, x}^\top y_t + \Dyn_{t, u}^\top v_t  + \Dyn_{t, w}^\top w_t + \Ddyn_{t, u, w}[v_t, w_t, \cdot].
	\end{align*}
	Plugging this in the model-minimization step gives the result.
\end{proof}
The linear control problem~\eqref{eq:LQG} can again be solved by dynamic programming  by modifying the Bellman equation~\eqref{eq:Bellman} in the backward pass, i.e., by solving analytically for white noise $w_t$,
\begin{align*}
\valuefunc_t(\hat y_t) = m_{h_t}(\hat y_t) + \min_{v_t}\left\{ m_{g_t}(v_t) + \Expect_{w_t}\left[ \valuefunc_{t+1}(\Dyn_{t,x}^\top \hat y_t + \Dyn_{t,u}^\top v_t + \Dyn_{t, w}^\top w_t  + \Ddyn_{t,u, w}[ v_t, w_t, \cdot])\right]\right\}.
\end{align*}
The complete resolution for quadratics is provided in Appendix~\ref{app:compa_ILQR}.
\paragraph{Dealing with constraints.}
For constrained control problems with exact dynamics, the model-minimization steps will amount to linear control problems under constraints, which cannot be solved directly by dynamic programming. However their resolution by an interior point method boils down to solving linear quadratic control problems each of which has a low computational cost as shown before. 

Formally, the resulting subproblems we are interested in are linear quadratic control problems under constraints of the form 
\begin{align}
\label{eq:opt_control_pb_constrained}
\min_{\substack{y_0,\ldots,y_\horizon \in \reals^d\\ v_0,\ldots,v_{\horizon-1} \in \reals^{p}}} \quad & \sum_{t=1}^{\horizon} q_{h_t}(y_t) + \sum_{t=0}^{\horizon-1} q_{g_t}(v_t) \\
\mbox{subject to} \quad & y_{t+1} = \ell_t(y_t,v_t), \quad y_0 = 0, \quad v_t \in \mathcal{U}_t \nonumber,
\end{align}
where $\mathcal{U}_t= \{u: C_t u\leq d_t\}$, $q_{h_t}$ are convex quadratics, $q_{g_t}$ are strongly convex quadratics and $\ell_t$ are linear dynamics. Interior point methods introduce a log-barrier function  $\mathcal{B}_{t}(u)  = \log(d_t - C_tu)$  and minimize 
\begin{align*}
& \min_{\substack{y_0,\ldots,y_\horizon \in \reals^d\\ v_0,\ldots,v_{\horizon-1} \in \reals^{p}}} \quad  \sum_{t=1}^{\horizon} q_{h_t}(y_t)  + \sum_{t=0}^{\horizon-1} q_{g_t}(v_t) + \mu_k \mathcal{B}_{t}(v_t) \\
& \mbox{subject to} \quad  y_{t+1} = \ell_t(y_t,v_t), \qquad  y_0 = 0,
\end{align*}
where $\mu_k$ increases along the iterates $k$ of the interior point method. 
We leave the exploration of constrained problems to future work.

\section{Automatic-differentiation oracle}\label{sec:auto_diff}
The iterative composition structure we studied so far appears not only in control but more generally in optimization problems that involve successive transformations of a given input as for example in
\begin{equation}\label{eq:compo_illus}
\min_{u_0,\ldots, u_{\horizon-1}}  h(\phi_{\horizon-1}(\phi_{\horizon-2}(\ldots\phi_0(\hat x_0, u_0) \ldots,u_{\horizon-2}),u_{\horizon-1})).
\end{equation}
The identification of such structures led to the development of efficient \emph{automatic-differentiation} software libraries able to compute gradients in any graph of computations both in CPUs and GPUs. We present then implementations and complexities of the optimization methods presented before where automatic-differentiation is the computational bottleneck.

\paragraph{Functions and problem definition.}
We first recall the definition of decomposable functions along the trajectories.
\begin{definition}
	A function $f:\reals^{\horizon d}\rightarrow \reals^{\horizon d'}$ is a \emph{multivariate $\tau$-decomposable function} if it is composed of $\tau$  functions $f_t : \reals^d \rightarrow \reals^{d'}$ such that for $\bar x = (x_1;\ldots; x_{\horizon}) \in \reals^{\horizon d}$, we have $f(\bar x) = (f_1(x_1); \ldots; f_\horizon(x_\horizon)) \in \reals^{\horizon d'}$. 
	
	A function $f:\reals^{\horizon d}\rightarrow \reals$ is a \emph{real $\tau$-decomposable function} if it is composed of $\tau$  functions $f_t : \reals^d \rightarrow \reals$ such that for $\bar x = (x_1;\ldots; x_{\horizon}) \in \reals^{\horizon d}$, we have $f(\bar x) = \sum_{t=1}^\horizon f_t(x_t)$. 
	
	We denote by $\mathcal{D}^\horizon(\reals^{\horizon d}, \reals^{\horizon d'})$ and $\mathcal{D}^\horizon(\reals^{\horizon d})$ the sets of multivariate and real, respectively, $\tau$-decomposable functions whose components $f_t$ are differentiable.
\end{definition}
For a given decomposable function $f \in \mathcal{D}^\horizon(\reals^{\horizon d},\reals^{\horizon d'})$ and  a point $\bar z \in \reals^{\horizon d'}$, the gradient-vector product reads $\nabla f(\bar x) \bar z = (\nabla f_1(x_1) z_1; \ldots; \nabla f_{\horizon} (x_\horizon) z_\horizon ) \in \reals^{\horizon d}$, i.e., it can be computed directly from the components defining $f$. Similarly, the convex conjugate of a real decomposable function $f \in \mathcal{D}^\horizon(\reals^{\horizon d})$ is directly given by the convex conjugate of its components.

We formalize now the class of trajectory functions.
\begin{definition}[Trajectory function]
	A function $\tilde x: \reals^{\horizon p} \rightarrow \reals^{\horizon d}$ is a \emph{trajectory function of horizon $\horizon$} if it is defined by an input $\hat x_0 \in \reals^d$ and $\horizon$ compositions of functions $\dyn_t : \reals^d \times \reals^p \rightarrow \reals^d$ such that for $\bar u = (u_0; \ldots; u_{\horizon-1}) \in \reals^{\horizon p}$, we have $\tilde x(\bar u) = (\tilde x_1(\bar u); \ldots; \tilde x_\horizon(\bar u))$ defined by
	\begin{align*}
	\tilde x_1(\bar u) = \dyn_0(\hat x_0, u_0), \quad 
	\tilde x_{t+1}(\bar u) = \dyn_t(\tilde x_t(\bar u),u_t), \qquad \mbox{for $t=1,\ldots, \horizon-1$.}
	\end{align*}
	
	We denote by $\mathcal{T}^\horizon(\reals^{\horizon p}, \reals^{\horizon d} )$ the set of trajectory functions of horizon $\horizon$ whose dynamics $\phi_t$ are differentiable.
\end{definition}
As presented in Section~\ref{sec:oracles_ctrl}, the gradient back-propagation is divided in two main phases: (i) the forward pass that computes and store the gradients of the dynamics along the trajectory given by a command, (ii) the backward and roll-out passes that compute the gradient of the objective given the gradients of the costs and penalties along the trajectory. We can decouple the two phases by computing and storing once and for all the gradients of the dynamics along the trajectory, then making calls to the backward and roll-out passes for any dual inputs, i.e., not restricting ourselves to the gradients of the costs and penalties along the trajectories.

Formally, given $\tilde x \in \mathcal{T}^\horizon(\reals^{\horizon p}, \reals^{\horizon d})$ and $\bar u \in \reals^{\horizon p}$, we use that, once $\tilde x(\bar u)$ is computed and the successive gradients are stored, any gradient vector product of the form  $\nabla \tilde x(\bar u) \bar z$ for $\bar z \in \reals^{\horizon d}$  can be computed in linear time with respect to $\horizon$ by a dynamic programming procedure (specifically an automatic-differentiation software) that solves $\min_{\bar v \in \reals^{\horizon p}} -\bar z^\top \nabla \tilde x(\bar u) ^\top \bar v + \frac{1}{2}\|\bar v\|_2^2$. The main difference with classical optimization oracles is that \emph{we do not compute or store the gradient $\nabla \tilde x(\bar u) \in \reals^{\horizon p \times \horizon d}$ but yet have access to gradient-vector products $\bar z \rightarrow \nabla \tilde x(\bar u) \bar z$}. This lead us to define oracles for trajectory functions as calls to an automatic-differentiation procedure as follows. 
\begin{definition}[Automatic-differentiation oracle]
	An \emph{automatic-differentiation oracle} is any procedure that, given $\tilde x \in \mathcal{T}^\horizon(\reals^{\horizon p}, \reals^{\horizon d})$ and  $\bar u \in \reals^{\horizon p}$, computes 	
	\[
	\bar z \rightarrow \nabla \tilde x(\bar u)\bar z \quad \mbox{for any $\bar z \in \reals^{\horizon d}$}.
	\]
\end{definition}
Derivatives of the gradient vector product can then be computed themselves by back-propagation as recalled in the following lemma. 
\begin{lemma}\label{lem:jac_vec_product}
	Given a trajectory function $\tilde x \in \mathcal{T}^\horizon(\reals^{\horizon p}, \reals^{\horizon d})$, a command  $\bar u  \in \reals^{\horizon p}$ and a real decomposable function $f \in \mathcal{D}^\horizon(\reals^{\horizon p})$, the derivative of $\bar z \rightarrow f(\nabla \tilde x(\bar u)\bar z)$ requires two calls to an automatic-differentiation procedure.
\end{lemma}
\begin{proof}
	We describe the backward pass of Prop.~\ref{prop:grad} as a function of $\bar z$, the computations are the same except that $\bar a$ is replaced by $-\bar z$.
	Given $\bar z = (z_1; \ldots; z_\horizon) \in \reals^{\horizon d}$, the backward pass that computes $\nabla \tilde x(\bar u) \bar z$ defines a linear trajectory function $\tilde \lambda: \bar z \rightarrow (\tilde \lambda_1(\bar z); \ldots; \tilde \lambda_\horizon(\bar z)) \in \reals^{\horizon d}$ and a linear decomposable function $\tilde \theta :\bar \lambda \rightarrow \tilde \theta(\bar \lambda) = (\tilde \theta_0(\lambda_1); \ldots \tilde \theta_{\horizon-1} (\lambda_{\horizon})) \in \reals^{\horizon p}$ as
	\begin{gather*}
	\tilde \lambda_\horizon(\bar z) = -z_\horizon, \quad \tilde \lambda_{t}(\bar z) = \Phi_{t,x}  \tilde \lambda_{t+1}(\bar z) - z_t \quad \mbox{for $t = \horizon -1,\ldots,  1$},\\\tilde \theta_t( \lambda_{t+1}) = -\Phi_{t,u} \lambda_{t+1} \quad \mbox{for $t = 0, \ldots, \horizon-1$},
	\end{gather*}
	where $\Dyn_{t,x}{=}\nabla_x \dyn_t(x_t, u_t)$, $\Dyn_{t,u}{ =} \nabla_u \dyn_t(x_t, u_t)$ and $x_t = \tilde x_t(\bar u)$.
	The function we are interested in reads then $f(\nabla \tilde x(\bar u) \bar z) = f(\tilde \theta( \tilde \lambda( \bar z)))$. Its derivative amounts then to compute the linear trajectory function $\tilde \lambda(\bar z)$ by one call to an automatic differentiation procedure, then to back-propagate through this linear trajectory function by another call to an automatic-differentiation procedure. The derivatives of the decomposable functions can be directly computed from their individual components.
	\end{proof}
We focus on problems that involve only a final state cost as in~\eqref{eq:compo_illus} or in the experiments presented in Section~\ref{sec:exp}. 
Formally those problems read  
\begin{equation}\label{eq:last_state_ctrl}
\min_{\bar u \in \reals^{\horizon p}} \quad  h(\tilde x_\horizon(\bar u)) + g(\bar u),
\end{equation}
where $\tilde x$ is trajectory function of horizon $\horizon$, $h$ is a cost function and $g$ is a real $\horizon$-decomposable penalty. Denote by $P(\tilde x, h, g)$, the problem~\eqref{eq:last_state_ctrl} for a given choice of $\tilde x, h, g$.
We present complexities of the oracles defined before for classes of problems $\mathcal{P}(\mathcal{T}, \mathcal{H}, \mathcal{G}) = \{ P(\tilde x, h, g): \tilde x \in \mathcal{T}, h \in \mathcal{H},  g \in \mathcal{G}\}$ defined by a class of trajectory functions $\mathcal{T} \subset \mathcal{T}^\horizon(\reals^{\horizon p}, \reals^{\horizon d})$, a class of state cost $\mathcal{H} \subset \mathcal{F}(\reals^d) = \{f:\reals^d \rightarrow \reals, \: f \:\mbox{differentiable}\}$ and a class of decomposable penalty function $\mathcal{G} \subset \mathcal{D}^\horizon(\reals^{\horizon d})$. Inclusions of classes of functions define inclusions of the problems.
The class of problems for which we can provide iteration complexity is defined by 
\begin{itemize}
	\item $\mathcal{T} = \mathcal{T}^\horizon_\alpha(\reals^{\horizon p}, \reals^{\horizon d})$ the class of trajectory functions of horizon $\horizon$ with $\alpha$-continuously differentiable dynamics,
	\item $\mathcal{H} = \mathcal{Q}_{L}(\reals^d)$ the class of quadratic convex functions $L$-smooth,
	\item $\mathcal{G} = \mathcal{Q}_{L}^\horizon(\reals^{\horizon p}) = \mathcal{Q}_{L}(\reals^{\horizon p}) \cap \mathcal{D}^\horizon(\reals^{\horizon p})$ the class of quadratic $\horizon$-decomposable functions $L$-smooth.
\end{itemize}
Note that any $\tilde x \in \mathcal{T}_{\alpha}^\horizon(\reals^{\horizon p}, \reals^{\horizon d})$ is $\alpha$-continuously differentiable. In the rest of this section we provide the oracle complexity of oracles for this problem general classes of problems detailed each time.

\paragraph{Model-minimization steps with automatic-differentiation oracles.}
Now we precise the feasibility and the complexity of the inner-steps of the steps defined in Section~\ref{sec:oracles_ctrl} in terms of the class of problems and the automatic-differentiation oracle defined above. The total complexity of the algorithms, when available, are presented in Section~\ref{sec:prox_lin}.
\paragraph{Gradient step.}
For any problem belonging to $\mathcal{P}(\mathcal{T}^\horizon(\reals^{\horizon p}, \reals^{\horizon d}), \mathcal{F}(\reals^d), \mathcal{D}^\horizon(\reals^{\horizon p}))$, a gradient step  amounts to compute $ \nabla \tilde x_\horizon(\bar u)\nabla h(\tilde x_\horizon(\bar u))$ and $\nabla g(\bar u) $ by a single call to an automatic-differentiation oracle.

\paragraph{Regularized Gauss-Newton step.}
In the setting~\eqref{eq:last_state_ctrl}, the regularized Gauss-Newton step~\eqref{eq:prox_linear_step} amounts to solve
\begin{align}\label{eq:prox_lin_last_state}
\min_{\bar v \in \reals^{\horizon p}} \: & h\left(\tilde x_\horizon(\bar u) + \nabla \tilde x_\horizon(\bar u)^\top\bar v\right) + g(\bar u + \bar v) + \frac{1}{2\gamma}\|\bar v\|_2^2.
\end{align}
For smooth objectives $h$ and $g$, this is a smooth strongly convex problem that can be solved approximately by a linearly convergent first order method, leading to the inexact regularized Gauss-Newton procedures described in~\citep{Drus16}.
The overall cost of an approximated regularized Gauss-Newton step is then given by the following proposition. We define (i) $\mathcal{T}^\horizon_{\alpha, L_0,\ldots, L_\alpha}(\reals^{\horizon p}, \reals^{\horizon d})$ the class of trajectory functions of horizon $\horizon$ whose dynamics $\phi_t$ are $L_\beta$-Lipschitz-continuous for all $\beta \leq \alpha$, (ii) $\mathcal{C}_{\alpha, \beta, L_\beta}(\reals^d)$ the class of convex functions $\alpha$-differentiable whose $\beta$-derivative, for $\beta\leq \alpha$, is $L_\beta$-Lipschitz continuous and (iii)  $\mathcal{C}_{\alpha, \beta, L_\beta}^\horizon(\reals^{\horizon p})= \mathcal{C}_{\alpha, \beta, L_\beta}(\reals^{\horizon p}) \cap \mathcal{D}^\horizon(\reals^{\horizon p})$ the class of $\horizon$-decomposable convex functions with corresponding smoothness properties. Note that any $\tilde x \in \mathcal{T}_{\alpha, L_0,\ldots, L_\alpha}^\horizon(\reals^{\horizon p}, \reals^{\horizon d})$ has a Lipschitz continuous $\beta$-derivative for $\beta\leq \alpha$.
\begin{proposition}\label{prop:autodiff_cplxity}
	For problems belonging to  $\mathcal{P}(\mathcal{T}_{1, L_0,L_1}(\reals^{\horizon p}, \reals^{\horizon d}), \mathcal{C}_{1,1,L_1^h}(\reals^d), \mathcal{C}_{1,1,L_1^g}^\horizon(\reals^{\horizon p}))$ defined in~\eqref{eq:last_state_ctrl}, an approximate regularized Gauss-Newton step given by~\eqref{eq:prox_lin_last_state} is solved up to $\varepsilon$ accuracy by a fast gradient method with at most
	\[
	\bigO\left(\sqrt{L_1^hM_0^2\stepsize + L_1^g\stepsize + 1} \log(\varepsilon)\right),
	\]
	calls to an automatic differentiation oracle, where $M_0$ is the Lipschitz-continuity of $\tilde x_\horizon$.
\end{proposition}
\begin{proof}
	For problems $\mathcal{P}(\mathcal{T}_{1, L_0,L_1}(\reals^{\horizon p}, \reals^{\horizon d}), \mathcal{C}_{1,1,L_1^h}(\reals^d), \mathcal{C}_{1,1,L_1^g}^\horizon(\reals^{\horizon p}))$, the regularized Gauss-Newton subproblem~\eqref{eq:last_state_ctrl} is $L_1^hM_0^2 + L_1^g + \stepsize^{-1}$ smooth and $\stepsize^{-1}$ strongly convex.
	Therefore to achieve $\varepsilon$ accuracy, a fast gradient method requires at most $\bigO(\sqrt{(L_1^hM_0^2 + L_1^g + \stepsize^{-1})/\stepsize^{-1}} \log(\varepsilon))$ calls to first order oracles of $\bar v \rightarrow h(\tilde x_\horizon(\bar u) + \nabla \tilde x_\horizon(\bar u)^\top\bar v) + g(\bar u + \bar v) + \frac{1}{2\gamma}\|\bar v\|_2^2$. 
	Each call requires to compute $\nabla \tilde x_\horizon(\bar u)^\top\bar v$, that is the derivative of $z \rightarrow v^\top \nabla \tilde x_\horizon(\bar u)z$ for $z\in \reals^d$, which costs two calls to an automatic differentiation procedure according to Lem.~\ref{lem:jac_vec_product}. An additional call to an automatic differentiation oracle  is then needed to compute $\nabla \tilde x_\horizon(\bar u) \nabla h(\tilde x_\horizon(\bar u) + \nabla \tilde x_\horizon(\bar u)^\top\bar v)$.
\end{proof}

\paragraph{Levenberg-Marquardt step.}
In the setting~\eqref{eq:last_state_ctrl}, the Levenberg-Marquardt step~\eqref{eq:LM_step} amounts to solve
\begin{align}\label{eq:LM_step_last_state}
\min_{\bar v \in \reals^{\horizon p}} \: & q_h\left(\tilde x_\horizon(\bar u)+\nabla \tilde x_\horizon(\bar u)^\top\bar v; \tilde x_\horizon(\bar u) \right)+q_g(\bar u+\bar v; \bar u) + \frac{1}{2\stepsize}\|\bar v\|_2^2,
\end{align}
where $q_h$ and $q_g$ are quadratic approximations of  $h$ and $g$ respectively, both being assumed to be twice differentiable. Here, duality offers a fast resolution of the step as shown in the following proposition. It shows that its cost is only $2d +1$ times more than one of a gradient step. Recall also that for $h$, $g$ quadratics the Levenberg-Marquardt step amounts to a regularized Gauss-Newton step. We define (i) $\mathcal{C}_{\alpha}(\reals^d)$ the class of convex functions $\alpha$-continuously differentiable and (ii) $\mathcal{C}_{\alpha}^\horizon(\reals^{\horizon p}) = \mathcal{C}_{\alpha}(\reals^{\horizon p}) \cap \mathcal{D}^\horizon(\reals^{\horizon p})$ the class of $\horizon$-decomposable convex functions with corresponding differentiation properties.
\begin{proposition}\label{prop:autodiff_cplxity_quad}
	For problems belonging to $\mathcal{P}(\mathcal{T}^\horizon(\reals^{\horizon p}, \reals^{\horizon d}), \mathcal{C}_{2}(\reals^d), \mathcal{C}_{2}^\horizon(\reals^{\horizon p}))$ defined in~\eqref{eq:last_state_ctrl}, a Levenberg-Marquardt step~\eqref{eq:LM_step_last_state} is solved exactly with at most $2d +1$ calls to an automatic differentiation oracle.
\end{proposition}
\begin{proof}
	The dual problem of the Levenberg-Marquardt step~\eqref{eq:LM_step_last_state} reads
	\begin{equation}\label{eq:dual_pb}
	\min_{z \in \reals^d} r^*(z)  + s^*(-\nabla \tilde x_\horizon(\bar u) z),
	\end{equation}
	where $r(x) = q_h\left(\tilde x_\horizon(\bar u) + x; \tilde x_\horizon(\bar u) \right) $, $s(\bar v)  = q_g(\bar u + \bar v; \bar u) + \frac{1}{2\stepsize}\|\bar v\|_2^2$ and $r^*, s^*$ are their respective conjugate functions that can be computed in closed form. Note that, as $s$ is $\horizon$ decomposable, so is $s^*$. The dual problem can then be solved in $d$ iterations of a conjugate gradient method, each iteration requires to compute the gradient of  $s^*(-\nabla \tilde x_\horizon(\bar u) z)$. According to Lem.~\ref{lem:jac_vec_product} this amounts to two calls to an automatic differentiation oracle. A primal solution is then given by $\bar v^* = \nabla s^*(-\nabla \tilde x_\horizon(\bar u) z^*)$ which is given by an additional call to an automatic differentiation oracle.
\end{proof}

\section{Composite optimization}\label{sec:prox_lin}
Before analyzing the methods of choice for composite optimization, we review classical algorithms for nonlinear control and highlight improvements for better convergence behavior. All algorithms are completely detailed in Appendix~\ref{app:compa_ILQR}.

\subsection{Optimal control methods}
\paragraph{Differential Dynamic Programming.}
Differential Dynamic Programming (DDP) is presented  as a dynamic programming procedure applied to a second order approximation of the Bellman equation~\citep{Jaco70}. Formally at a given command $\bar u$ with associated trajectory $\bar x = \tilde x (\bar u)$, it consists in approximating the cost-to-go functions as 
\begin{align*}
\valuefunc_{t}(y) & =  q_{h_t}(x_t+y; x_t) + \min_{v} \{q_{g_t}(u_t+v; u_t) + q_{\valuefunc_{t+1} \circ \dyn_t}(x_t+y, u_t+v; x_t, u_t)\}
\end{align*}
where for a function $f(y)$, $q_f(y; x)$ denotes its second order approximation around $x$. The roll-out pass is then performed on the true trajectory as normally done in a dynamic programming procedure.
We present an interpretation of DDP as an optimization on the state variables in Appendix~\ref{app:approx_dyn_prog}.

\paragraph{ILQR, ILQG \citep{Li04, Todo05, Li07}.}
DDP was superseded by the Iterative Linearized Quadratic Regulator (ILQR) method, presented in Section~\ref{sec:opt_ctrl_setting}~\citep{Li04}.
In the case of noisy dynamics, the Linear Quadratic Regulator problem~\eqref{eq:LQR_gauss_newton} was replaced by a Linear Quadratic Gaussian problem where the objectives are averaged with respect to the noise, the iterative procedure was then called ILQG as presented in~\citep{Todo05, Li07}. 

Prop.~\ref{prop:LQR_step} clarifies that these procedures, as defined in~\citep{Li04, Todo05, Li07}, amount to compute 
\begin{equation}\label{eq:gauss_newton}
\bar v^* = \argmin_{\bar v\in \reals^{\horizon p}} q_f(\bar u + \bar v;\bar u) 
\end{equation}
to perform a line-search along its direction such that 
$f(\bar u + \alpha \bar v^*) \leq f(\bar u)$. For ILQR the model $q_f$ is defined as in \eqref{eq:quad_model}, while for ILQG this corresponds to the model defined in~\eqref{eq:approx_gaussian} with quadratic models $q_f$ and $q_g$. Compared to a Levenberg-Marquardt step~\eqref{eq:LM_step}, that reads
\begin{equation}\label{eq:Levenberg_Marquardt}
\bar u^+ = \bar u + \argmin_{\bar v \in \reals^{\horizon p}}\left\{q_f(\bar u + \bar v;\bar u) +\frac{1}{2\stepsize} \|\bar v\|_2^2 \right\},
\end{equation}
we see that those procedures do not take into account the inaccuracy of the model far from the current point. Although a line-search can help ensuring convergence, no rate of convergence is known.
For quadratics $h_t, g_t$, the Levenberg-Marquardt steps become regularized Gauss-Newton steps whose analysis shows the benefits of the regularization term in~\eqref{eq:Levenberg_Marquardt} to ensure convergence to a stationary point.

\paragraph{ILQG \citep{tassa2012synthesis}.}
The term ILQG has often been used to refer to an algorithm combining ideas from DDP and ILQR resp.~\citep{tassa2012synthesis}. The general structure proposed then is akin to DDP in the sense that it uses a dynamic programming approach where the cost-to-go functions are approximated. However, as in ILQR, only the first order derivatives of the dynamics are taken into account to approximate the cost-to-go functions. Formally, at a given command $\bar u$ with associated trajectory $\bar x = \tilde x (\bar u)$, ILQG consists in approximating the cost-to-go functions as 
\begin{align*}
\valuefunc_{t}(y) & =  q_{h_t}(x_t+y; x_t) + \min_{v} \{q_{g_t}(u_t+v; u_t) + q_{\valuefunc_{t+1}}(\dyn_t(x_t) + \nabla \dyn_t(x_t, u_t)^\top v; x_t)\}
\end{align*}
While the cost-to-go functions are the same as in~\citep{Li04}, the roll-out pass is then performed on the true trajectory and not the linearized one. The analysis is therefore similar to the one of DDP. We leave it for future work and focus on the original definition of ILQR given in~\citep{Li04}.

\subsection{Regularized ILQR via regularized Gauss-Newton}
We present convergence guarantees of the regularized Gauss-Newton method for composite optimization problems of the form
\begin{equation}\label{eq:composite_analysis}
\min_{{\bar u}\in \reals^{\horizon d}} \: \optimobj({\bar u}) = \ctrlobj(\targetfunc({\bar u})) + \regfunc({\bar u}),
\end{equation}
where $\ctrlobj:\reals^{\horizon d} \rightarrow \reals$ and $\regfunc:\reals^{\horizon p}\rightarrow \reals$ are convex quadratic, and $\targetfunc:\reals^{\horizon p}\rightarrow \reals^{\horizon d}$ is differentiable with continuous gradients. The regularized Gauss-Newton method then naturally leads to a regularized ILQR. 
In the following, we denote by $L_\ctrlobj$ and $L_\regfunc$ the smoothness constants of respectively $\ctrlobj$ and $\regfunc$ and by $\ell_{\targetfunc,S}$ the Lipschitz constant of $\targetfunc$ on the initial sub-level set $S = \{ {\bar u}: \optimobj({\bar u}) \leq \optimobj({\bar u}_0)\}$. 

The regularized Gauss-Newton method consists in iterating, starting from a given ${\bar u}_0$,
\begin{equation}\label{eq:prox_lin_algo}
{\bar u}_{k+1} = \bar u_k + \argmin_{\bar v \in \reals^{\horizon p}}\left\{c_f(\bar u_k{ + }\bar v;\bar u_k)+ \frac{1}{2\stepsize_k} \|\bar v\|_2^2 \right\},
\end{equation}
We use $\bar u_{k+1} = \text{GN}(u_k; \gamma_k)$ to denote~\eqref{eq:prox_lin_algo} hereafter.
The convergence is stated in terms of the difference of iterates
that,
in this case, can directly be linked to the norm of the gradient, denoting $H = \nabla^2 h(\bar x)$ and $G = \nabla^2 g({\bar u})$,
\begin{equation}\label{eq:explicit_GN}
{\bar u}_{k+1} = {\bar u}_k - (\nabla \targetfunc({\bar u}_k) H \nabla\targetfunc({\bar u}_k)^\top + G + \stepsize_k^{-1}\idm_{\horizon p})^{-1} \nabla \optimobj ({\bar u}_k).
\end{equation}

The convergence to a stationary point is guaranteed as long as we are able to get a sufficient decrease condition when minimizing this model as stated in the following proposition.
\begin{restatable}{proposition}{proxlinconv}\label{prop:prox_lin_conv}
	Consider a composite objective $f$ as in~\eqref{eq:composite_analysis} with convex models $c_f(\cdot; \bar u)$ defined in~\eqref{eq:cvx_model}.
	Assume that the step sizes $\stepsize_k$ of the regularized Gauss-Newton method~\eqref{eq:prox_lin_algo} are chosen such that 
	\begin{equation}\label{eq:suff_decrease}
	\optimobj({\bar u}_{k+1}) \leq c_f({\bar u}_{k+1}; {\bar u}_k)+ \frac{1}{2\stepsize_k} \|{\bar u}_{k+1}-{\bar u}_k\|_2^2
	\end{equation}
	and $\stepsize_{\min} \leq \gamma_k \leq \stepsize_{\max}$.
	
	Then the objective value decreases over the iterations and the sequence of iterates satisfies
	\[
	\min_{k=0,\ldots,N}\norm{\nabla \optimobj ({\bar u}_{k})}^2 \leq \frac{2 L (\optimobj({\bar u}_0) - \optimobj^*) }{N+1} ,
	\]
	where $L = \max_{\stepsize \in[\stepsize_{\min}, \stepsize_{\max}]} \stepsize (\ell_{\targetfunc,S}^2L_h  + L_g + \stepsize^{-1})^2$ and $\optimobj^* = \lim_{k \rightarrow +\infty}\optimobj({\bar u}_k)$.
\end{restatable}

To ensure the sufficient decrease condition, one needs the model to approximate the objective up to a quadratic error which is ensured on any compact set as stated in the following proposition.
\begin{restatable}{lemma}{suffdecrease}
\label{lem:suff_decrease}
	Consider a composite objective $f$ as in~\eqref{eq:composite_analysis} with convex models $c_f(\cdot; \bar u)$ defined in~\eqref{eq:cvx_model}.
	For any compact set $C\subset \reals^{\horizon p}$ there exists $M_C >0$ such that for any ${\bar u}, \bar v \in C$,
	\begin{equation}\label{eq:quad_error}
	|\optimobj(\bar v) - c_f(\bar v; {\bar u}) | \leq \frac{M_C}{2} \|\bar v-  {\bar u}\|_2^2.
	\end{equation}
\end{restatable}
Finally one needs to ensure that the iterates stay in a bounded set which is the case for sufficiently small step-sizes such that the sufficient decrease condition is satisfied along the sequence of iterates generated by the algorithm.
\begin{restatable}{lemma}{minstep}\label{lem:min_step}
	Consider a composite objective $f$ as in~\eqref{eq:composite_analysis}.
	For any $ k$ such that ${\bar u}_k \in S$, where $S = \{ {\bar u}: \optimobj({\bar u}) \leq \optimobj({\bar u}_0)\}$ is the initial sub-level set, any step-size 
	\begin{equation}\label{eq:min_step_size}
	\gamma_k \leq \hat \gamma = \min\{\ell_{\optimobj, S}^{-1}, M_{C}^{-1}\}
	\end{equation} 
	ensures that the sufficient decrease condition~\eqref{eq:suff_decrease} is satisfied, 
	where $\ell_{\optimobj, S}$ is the Lipschitz constant of $\optimobj$ on $S$, $C = S + B_1$ with $B_1$ the unit Euclidean ball centered at 0 and $M_C$ ensures~\eqref{eq:quad_error}.
\end{restatable}
Combining Prop.~\ref{prop:prox_lin_conv} and Lem.~\ref{lem:suff_decrease}, we can guarantee that the iterates stay in the initial sub-level set and satisfy the sufficient decrease condition for sufficiently small step-sizes $\stepsize_k$. At each iteration the step-size can be found by a line-search guaranteeing sufficient decrease; see Appendix~\ref{app:prox_lin} for details. The final complexity of the algorithm with line-search then follows.
\begin{corollary}\label{cor:line_search}
	For a composite objective $f$ as in~\eqref{eq:composite_analysis}, the regularized Gauss-Newton method~\eqref{eq:prox_lin_algo} with a decreasing line-search starting from $\stepsize_0 \geq \hat \stepsize$ with decreasing factor $\rho$ finds an $\varepsilon$-stationary point after at most 
	\[
	\frac{2L(f(\bar u_0) - f^*)}{\varepsilon^2} + \log(\stepsize_0/\hat \stepsize)/\log(\rho^{-1})
	\]
	calls to the regularized Gauss-Newton oracle, with $\hat \stepsize $ defined in \eqref{eq:min_step_size}, $\optimobj^* = \lim_{k \rightarrow +\infty}\optimobj({\bar u}_k)$ and
	\[
	L = \max_{\stepsize \in[\hat \stepsize, \stepsize_{0}]} \stepsize (\ell_{\targetfunc,S}^2L_h  + L_g + \stepsize^{-1})^2.
	\]
\end{corollary}

\begin{table}[t]
	\begin{center}
		\begin{tabular}{|c|c|c|c|}
			\hline 
			& GD & ILQR & RegILQR \\ 
			\hline 
			Global convergence guarantee & Yes & No & Yes \\ 
			\hline 
			Number of calls to auto-differentiation oracle & $\tau (pd + d^2)$ & $\tau p ^3d^3$ & $\tau p ^3d^3$ \\ 
			\hline 
			Cost per call to auto-differentiation oracle & $1$ & $2d+1$ & $2d+1$ \\ 
			\hline 
		\end{tabular}
		\caption{Convergence properties and oracle costs of Gradient Descent (GD), ILQR, and regularized ILQR (RegILQR) for problem~\eqref{eq:opt_control_pb} with quadratic $h$, $g$ . The automatic-differentiation oracle cost is stated for problems of the form~\eqref{eq:last_state_ctrl} \label{tab:compa}.}
	\end{center} 
\end{table}

\subsection{Accelerated ILQR via accelerated Gauss-Newton}
In Algo.~\ref{algo:acc_prox_lin} we present an accelerated variant of the regularized Gauss-Newton algorithm that blends a regularized Gauss-Newton step and an extrapolated step to potentially capture convexity in the objective. See Appendix~\ref{app:acc_proxlin} for the proof.

\begin{restatable}{proposition}{acc}\label{prop:acc_prox_lin}
	Consider Algo.~\ref{algo:acc_prox_lin}  applied to a composite objective $f$ as in~\eqref{eq:composite_analysis} with decreasing step-sizes $(\stepsize_{k})_{k\geq 0}$ and $(\accstepsize_{k})_{k\geq 0}$. Then Algo.~\ref{algo:acc_prox_lin} satisfies the convergence of the regularized Gauss-Newton method~\eqref{eq:prox_lin_algo} with line-search as presented in Cor.~\ref{cor:line_search}. Moreover, if the convex models $c_f(\bar v; \bar u)$ defined in~\eqref{eq:cvx_model} lower bound the objective as 
	\begin{equation}\label{eq:lower_bound}
	c_f(\bar v; \bar u) \leq f(\bar v)
	\end{equation}
	for any $\bar u, \bar v \in \reals^{\horizon p}$, then after $N$ iterations of Algo.~\ref{algo:acc_prox_lin}, 
	\begin{flalign*}
	f(\bar u_N) - f^* & \leq \frac{4\accstepsize^{-1} \norm{\bar u^* - \bar u_0}^2}{(N+1)^2},
	\end{flalign*}
	where $\accstepsize = \min_{k \in \{1,\ldots N\}} \accstepsize_k$, $f^* = \min_{\bar u} f(\bar u)$ and $\bar u^* \in\argmin_{\bar u} f(\bar u)$.
\end{restatable}

\begin{algorithm}[t]
	\caption{Accelerated Regularized Gauss-Newton}
	\label{algo:acc_prox_lin}
	\begin{algorithmic}[1]
		\Statex{\textbf{Input:}} Composite objective $f$ in~\eqref{eq:composite_pb} with convex models $c_f$ as  in~\eqref{eq:cvx_model}. Initial $\bar u_0 \in \reals^{\horizon p}$, desired accuracy $\varepsilon$.
		\Statex{{\bf Initialize:} $\alpha_1 := 1$, $\bar z_0 := \bar u_0$}
		\Statex{{\bf Repeat:} for $k =1, 2, \dots$} 
		\State{\underline{Compute regularized step}}
		
		\Statex{Get $\bar v_k = \text{GN}(\bar u_{k-1}; \gamma_k)$ by line-search on $\stepsize_k$ s.t.
			\vspace{-2ex}
			\[
			f(\bar v_k)  \leq c_{f}(\bar v_k;\bar u_{k-1}) + \frac{1}{2\gamma_k}\|\bar v_k - \bar u_{k-1}\|_2^2.
			\]
			\vspace{-4ex}
		}
		\State{\underline{Compute extrapolated step}}
		\Statex{- Set$\label{eq:y_def} 
			\: \:\bar y_k = \alpha_k \bar z_{k-1} + (1-\alpha_k) \bar u_{k-1}.	$	
		}
		\Statex{- Get $\bar w_k = \text{GN}(\bar y_k; \accstepsize_{k})$ by line-search on $\accstepsize_k$ s.t.
			\vspace{-2ex}
			\begin{equation}\label{eq:suff_decrease_acc}
			f(\bar w_k)  \leq c_{f}(\bar w_k;\bar y_{k}) + \frac{1}{2\delta_k}\|\bar w_k - \bar y_{k}\|_2^2.
			\end{equation}
			\vspace{-4ex}
		}
		\Statex{- Set 
			$\label{eq:z_eq}
			\: \bar z_k = \bar u_{k-1} + (\bar w_k-\bar u_{k-1})/\alpha_k.
			$}	
		\Statex{-  Pick $\alpha_{k+1} \in (0,1)$ s.t.
			$\label{eq:def_alpha}
			(1-\alpha_{k+1})/\alpha^2_{k+1} = 1/\alpha_{k}^2.
			$ }
		\State{\underline{Pick best of two steps}}
		\Statex{
			Choose $\bar u_k$ such that
			\vspace{-2ex}
			\begin{equation}\label{eq:best_of}
			f(\bar u_k) \leq \min \, \{ f(\bar v_k), f(\bar w_k) \}
			\end{equation}
			\vspace{-4ex}
		}
		\Statex{\textbf{until} $\varepsilon$-near stationarity $\norm{\nabla f(\bar u_k)} < \varepsilon$}
	\end{algorithmic}
\end{algorithm}
\subsection{Total complexity with automatic-differentiation oracles}
Previous results allow us to state the total complexity of the regularized ILQR algorithm in terms of calls to automatic differentiation oracles  as done in the following corollary that combines Cor.~\ref{cor:line_search} and Prop.~\ref{prop:acc_prox_lin} with Prop.~\ref{prop:autodiff_cplxity_quad}. A similar result can be obtained for the accelerated variant. Table~\ref{tab:compa} summarizes then convergence properties and computational costs of classical methods for discrete time non-linear control.

\begin{corollary}
	Consider problems $\mathcal{P}(\mathcal{T}_{1}(\reals^{\horizon p}, \reals^{\horizon d}), \mathcal{Q}_{L_h}(\reals^d), \mathcal{Q}_{L_g}^\horizon(\reals^{\horizon p}))$ defined in~\eqref{eq:last_state_ctrl}. The regularized Gauss-Newton method~\eqref{eq:prox_lin_algo} with a decreasing line-search starting from $\stepsize_0 \geq \hat \stepsize$ with decreasing factor $\rho$ finds an $\varepsilon$-stationary point after at most 
	\[
	(2d+1)\left(\frac{2L(f(\bar u_0) - f^*)}{\varepsilon^2} + \log(\stepsize_0/\hat \stepsize)/\log(\rho^{-1})\right)
	\]
	calls to an automatic differentiation oracle, with $\hat \stepsize $ defined in \eqref{eq:min_step_size},
	$
	L = \max_{\stepsize \in[\hat \stepsize, \stepsize_{0}]} \stepsize (\ell_{\targetfunc,S}^2L_h  + L_g + \stepsize^{-1})^2,$ $\ell_{\targetfunc,S}$ is the Lipschitz constant of $\targetfunc$ on the initial sub-level set $S = \{ {\bar u}: \optimobj({\bar u}) \leq \optimobj({\bar u}_0)\}$
	and $\optimobj^* = \lim_{k \rightarrow +\infty}\optimobj(\unknown_k)$
\end{corollary}

\section{Experiments}\label{sec:exp}
We illustrate the performance of the algorithms considered in Sec.~\ref{sec:prox_lin} including the proposed accelerated regularized Gauss-Newton algorithm on two classical problems drawn from~\citep{Li04}: swing-up a pendulum, and move a two-link robot arm. 

\subsection{Control settings}
The physical systems we consider below are described by continuous dynamics of the form
\begin{align*}
\ddot z(t) = f(z(t), \dot z (t), u(t))
\end{align*}
where $z(t), \dot z(t), \ddot z(t)$ denote respectively the position, the speed and the acceleration of the system and $u(t)$ is a force applied on the system. 
The state $x(t) = (x_1(t), x_2(t))$ of the system is defined by the position $x_1(t) = z(t)$ and the speed $x_2(t) = \dot z(t)$ and the continuous cost is defined as
\begin{align*}
J(x, u) = \int_0^T h(x(t))dt + \int_0^T g(u(t))dt \quad \mbox{or} \quad J(x, u) = h(x(T)) + \int_0^T g(u(t))dt,
\end{align*}
where $T$ is the time of the movement and $h, g$ are given convex costs. The discretization of the dynamics with a time step $\delta$ starting from a given state $\hat x_0 = (z_0, 0)$ reads then
\begin{equation}\label{eq:discrete_dyn}
\begin{split}
x_{1,t+1} & = x_{1,t} + \delta x_{2,t} \\
x_{2,t+1} & = x_{2,t} + \delta f(x_{1,t}, x_{2,t}, u_t) 
\end{split}
\quad \mbox{for $t = 0,\ldots, \horizon-1$}
\end{equation}
where $\horizon = \lceil T/\delta\rceil$ and the discretized cost reads
\begin{align*}
J(\bar x, \bar u ) = \sum_{t=1}^{\horizon}h(x_t) + \sum_{t=0}^{\horizon-1} g(u_t) \quad \mbox{or} \quad J(\bar x, \bar u) = h(x_\horizon) + \sum_{t=0}^{\horizon-1} g(u_t).
\end{align*}

\begin{figure}
	\begin{center}
		\includegraphics[width=0.3\textwidth]{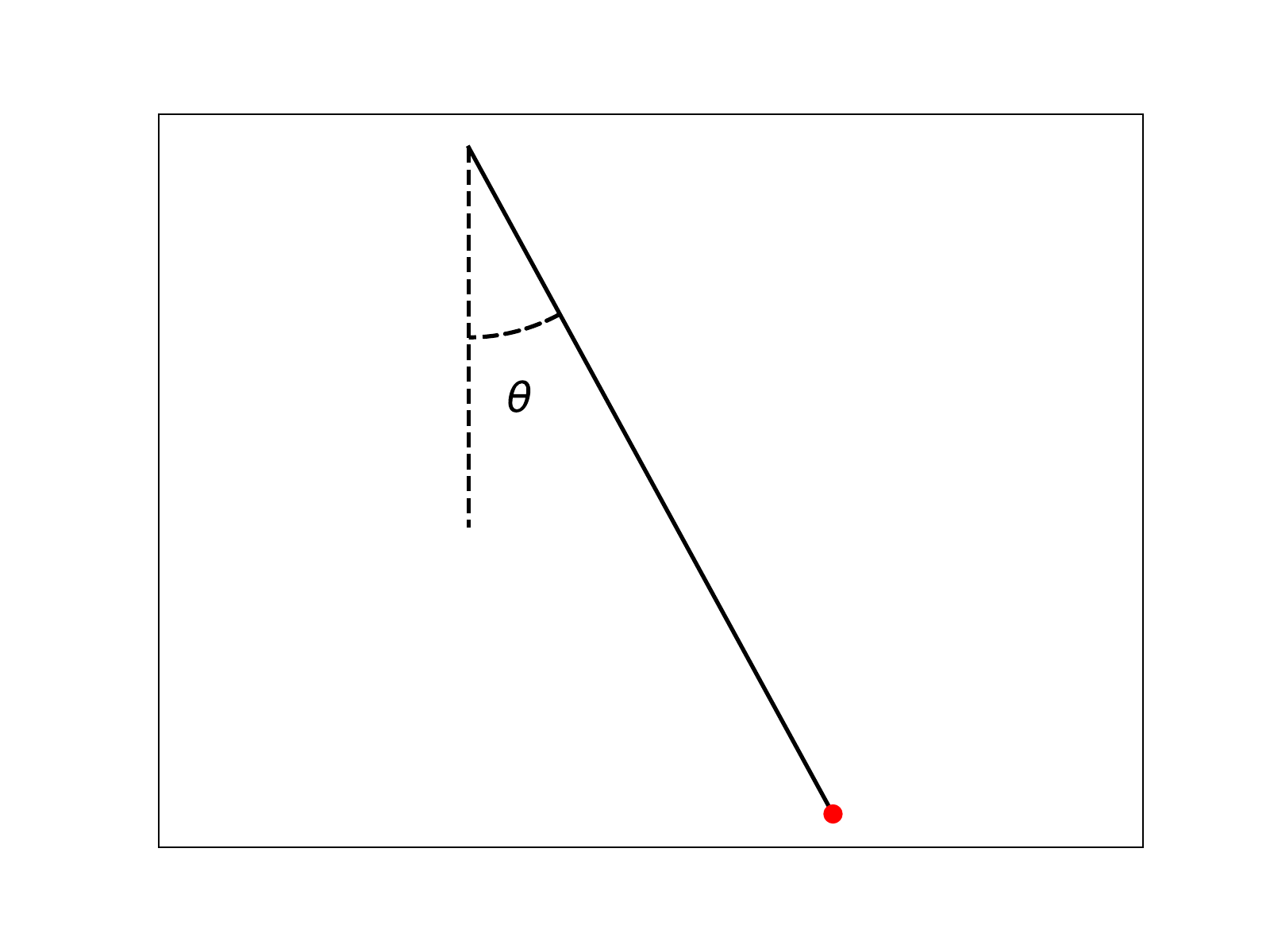}
		\includegraphics[width=0.3\textwidth]{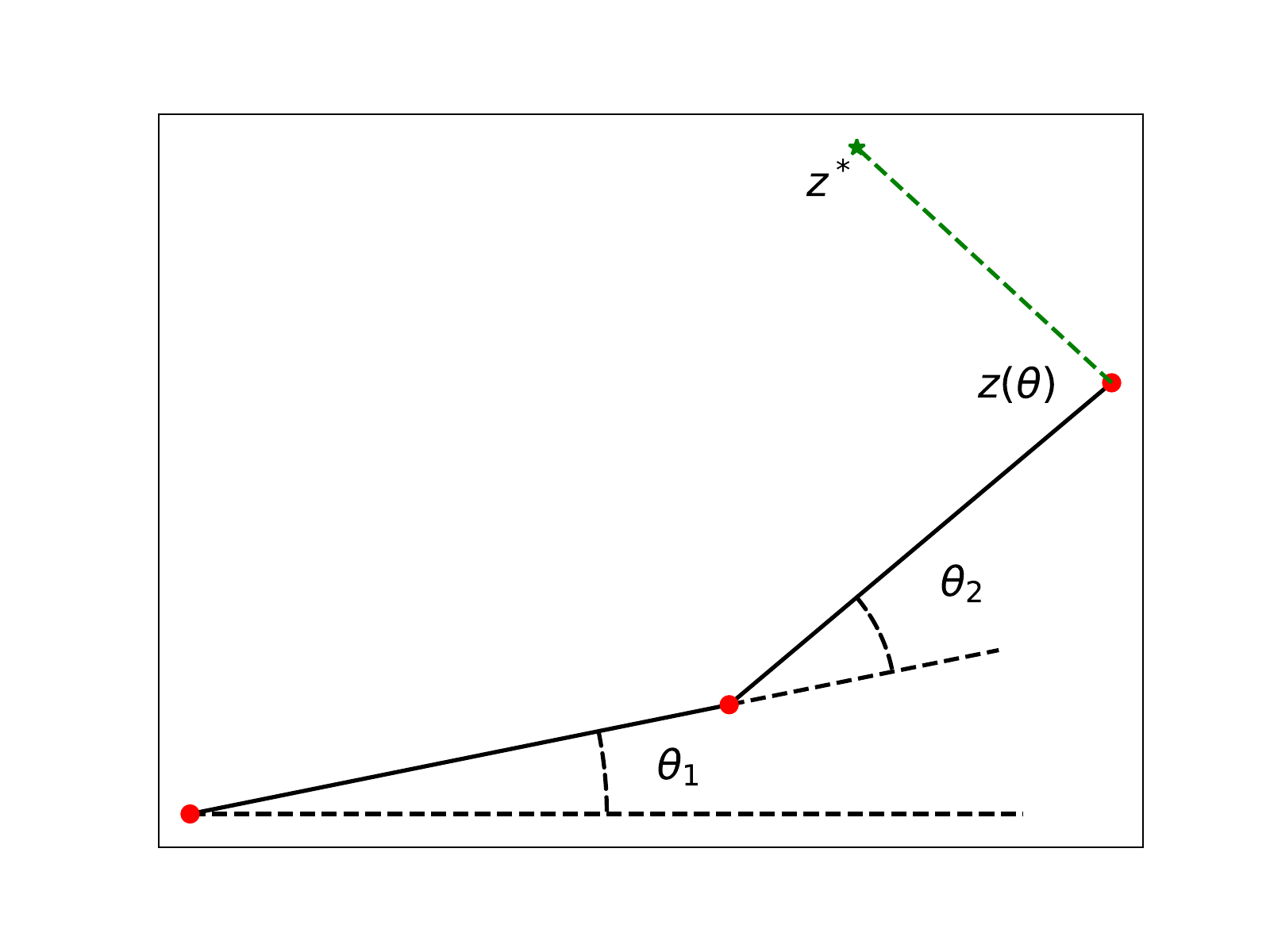}
	\end{center}
	\caption{Control settings considered. From left to right: pendulum, two-link arm robot. \label{fig:ctrl_settings}}
\end{figure}
\paragraph{Pendulum.}
We consider a simple pendulum illustrated in Fig.~\ref{fig:ctrl_settings}, where $m=1$ denotes the mass of the bob, $l=1$ denotes the length of the rod, $\theta$ describes the angle subtended by the vertical axis and the rod, and $\mu=0.01$ is the friction coefficient. The dynamics are described by
\[
\ddot \theta(t) = -\frac{g}{l}\sin \theta(t) - \frac{\mu}{ml^2} \dot\theta(t) + \frac{1}{ml^2}u(t)
\]
The goal is to make the pendulum swing up (i.e. make an angle of $\pi$ radians) and stop at a given time $T$. The cost writes as
\begin{equation}\label{eq:inverse_pendulum_dynamics}
J(x,u) = 
(\pi - x_1(T))^2 + \lambda_1 x_2(T)^2 + \lambda_2 \int_0^T u^2(t)dt,
\end{equation}
where $x(t) = (\theta(t), \dot \theta(t))$, $\lambda_1 = 0.1, \lambda_2 = 0.01$, $T = 5$. 

\paragraph{Two-link arm.}
We consider the arm model with two joints
(shoulder and elbow), moving in the horizontal plane presented in \citep{Li04} and illustrated in~\ref{fig:ctrl_settings}. The dynamics are described by
\begin{equation}\label{eq:dynamics_2_link_arm_model}
M(\theta(t)) \ddot\theta(t) + C(\theta(t),\dot \theta(t)) + B \dot \theta(t) = u(t),
\end{equation}
where $\theta = (\theta_1, \theta_2)$ is the joint angle vector, $M(\theta) \in \reals^{2\times 2}$ is a positive definite symmetric inertia matrix, $C(\theta, \dot \theta)\in \reals^2$ is a vector centripetal and Coriolis forces, $B \in \reals^{2 \times 2}$ is the joint
friction matrix, and $u(t) \in \reals^2$ is the joint torque that we control. 
We drop the dependence on $t$ for readability. The dynamics are then
\begin{equation}\label{sec:2links_arm_dynamcis}
\ddot\theta  = M(\theta)^{-1} (u - C(\theta,\dot \theta) - B \dot \theta).
\end{equation}
The expressions of the different variables and parameters are given by
\begin{align}
M(\theta) & = 
\left( 
\begin{matrix}
a_1 + 2a_2\cos \theta_2 & a_3 + a_2 \cos \theta_2 \\
a_3 + a_2\cos \theta_2 & a_3
\end{matrix}
\right) & 
C(\theta, \dot \theta) & = 
\left( 
\begin{matrix}
- \dot \theta_2(2 \dot \theta_1 + \dot \theta_2) \\
\dot \theta_1^2
\end{matrix}
\right)a_2\sin\theta_2  \label{eq:detailed_dynamics_1}\\
B & = 
\left( 
\begin{matrix}
b_{11} & b_{12} \\
b_{21} & b_{22}
\end{matrix}
\right) & &
\begin{array}{ll}
a_1 & = k_1 + k_2 + m_2 l_1^2 \\
a_2 & = m_2l_1d_2\\
a_3 & = k_2,
\end{array} \label{eq:detailed_dynamics_2}
\end{align}
where $b_{11} = b_{22} = 0.05$, $b_{12} = b_{21} = 0.025$, $l_i$ and $k_i$ are respectively the length (30cm, 33cm) and the moment of inertia (0.025kgm\textsuperscript{2} , 0.045kgm\textsuperscript{2}) of link $i$ , $m_2$ and $d_2$ are respectively the mass (1kg) and the distance (16cm) from the joint center to the center of the mass for the second link.

The goal is to make the arm reach a feasible target $\theta^*$ and stop at that point. The objective reads
\begin{equation}\label{eq:2arm_objective}
J(x, u) = \|\theta(T) -\theta^*\|^2 + \lambda_1 \|\dot \theta(T) \|^2 + \lambda_2 \int_0^T \|u(t)\|^2dt,
\end{equation} 
where $x(t) = (\theta(t), \dot \theta(t))$, $\lambda_1 = 0.1, \lambda_2 = 0.01$, $T = 5$. 

\subsection{Results}
We use the automatic differentiation capabilities of PyTorch~\citep{Pyto17} to implement the automatic differentiation oracles introduced in Sec.~\ref{sec:auto_diff}. The Gauss-Newton-type steps in Algo. 1 are computed by solving the dual problem associated as presented in Sec.~\ref{sec:auto_diff}.

In Figure~\ref{fig:exp}, we compare the convergence, in terms of function value and gradient norm, of ILQR (based on Gauss-Newton), regularized ILQR (based on regularized Gauss-Newton), and accelerated regularized ILQR (based on accelerated regularized Gauss-Newton). These algorithms were presented in Sec.~\ref{sec:prox_lin}.

For ILQR, we use an Armijo line-search to compute the next step. For both the regularized ILQR and the accelerated regularized ILQR, we use a constant step-size sequence tuned after a burn-in phase of 5 iterations.
We leave the exploration of more sophisticated line-search strategies for future work. 

The plots show stable convergence of the regularized ILQR on these problems. The proposed accelerated regularized Gauss-Newton algorithm displays stable and fast convergence. Applications of accelerated regularized Gauss-Newton algorithms to reinforcement learning problems would be interesting to explore~\citep{Rech18, Faze18, Dean18}. 

\begin{figure}
	\begin{center}
		\includegraphics[width = 0.3\linewidth]{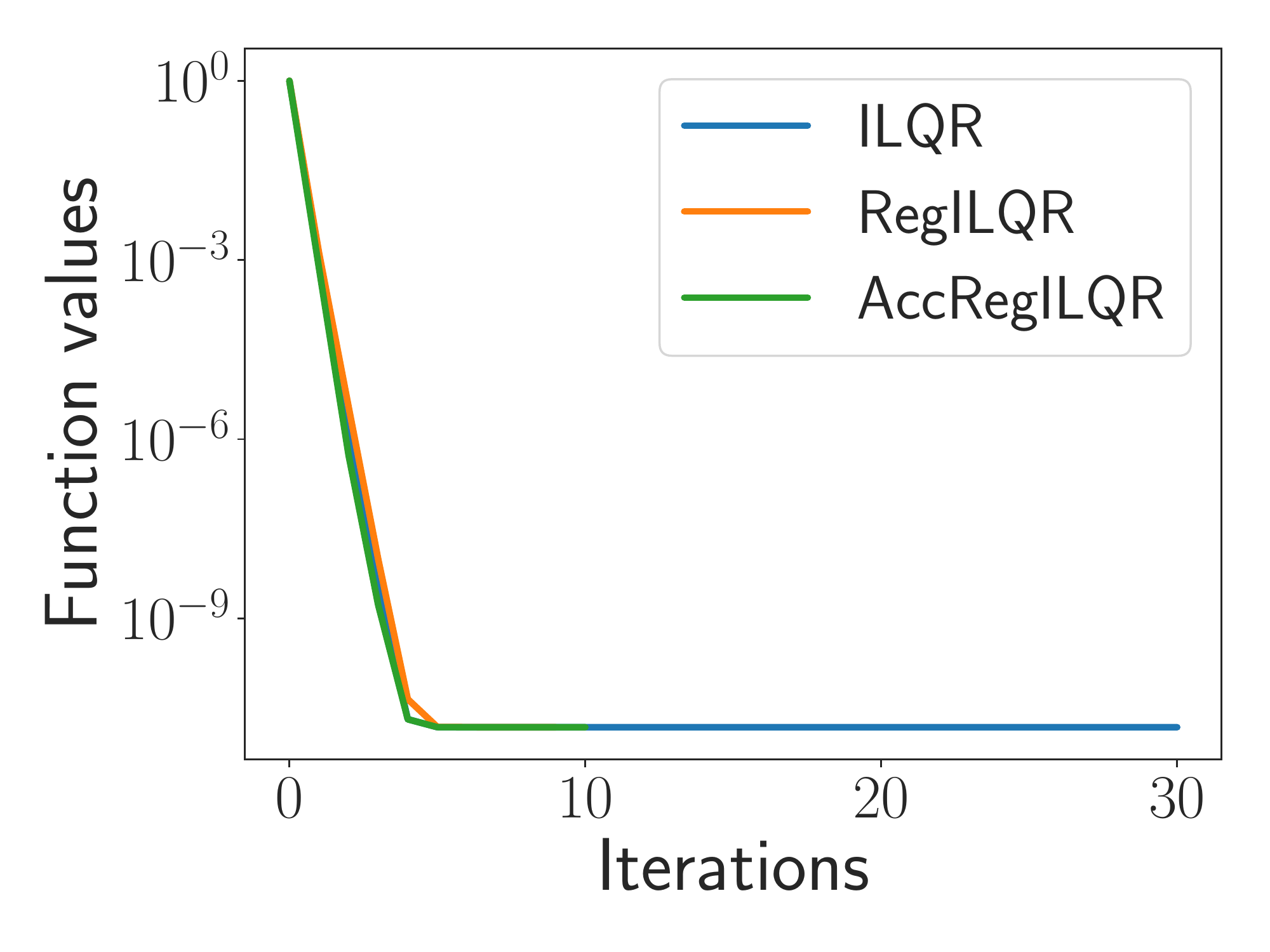}
		\includegraphics[width = 0.3\linewidth]{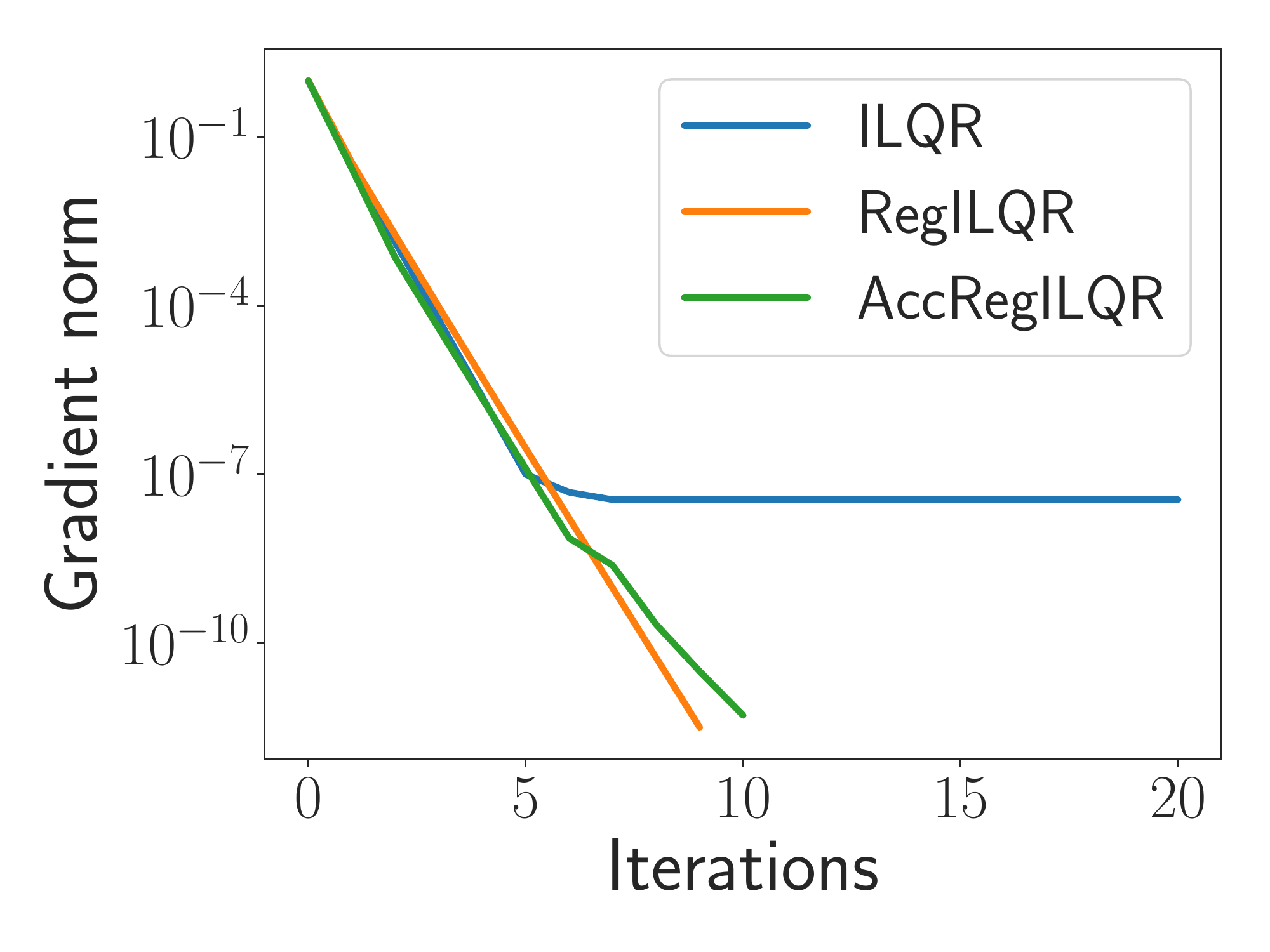}
		
		\includegraphics[width = 0.3\linewidth]{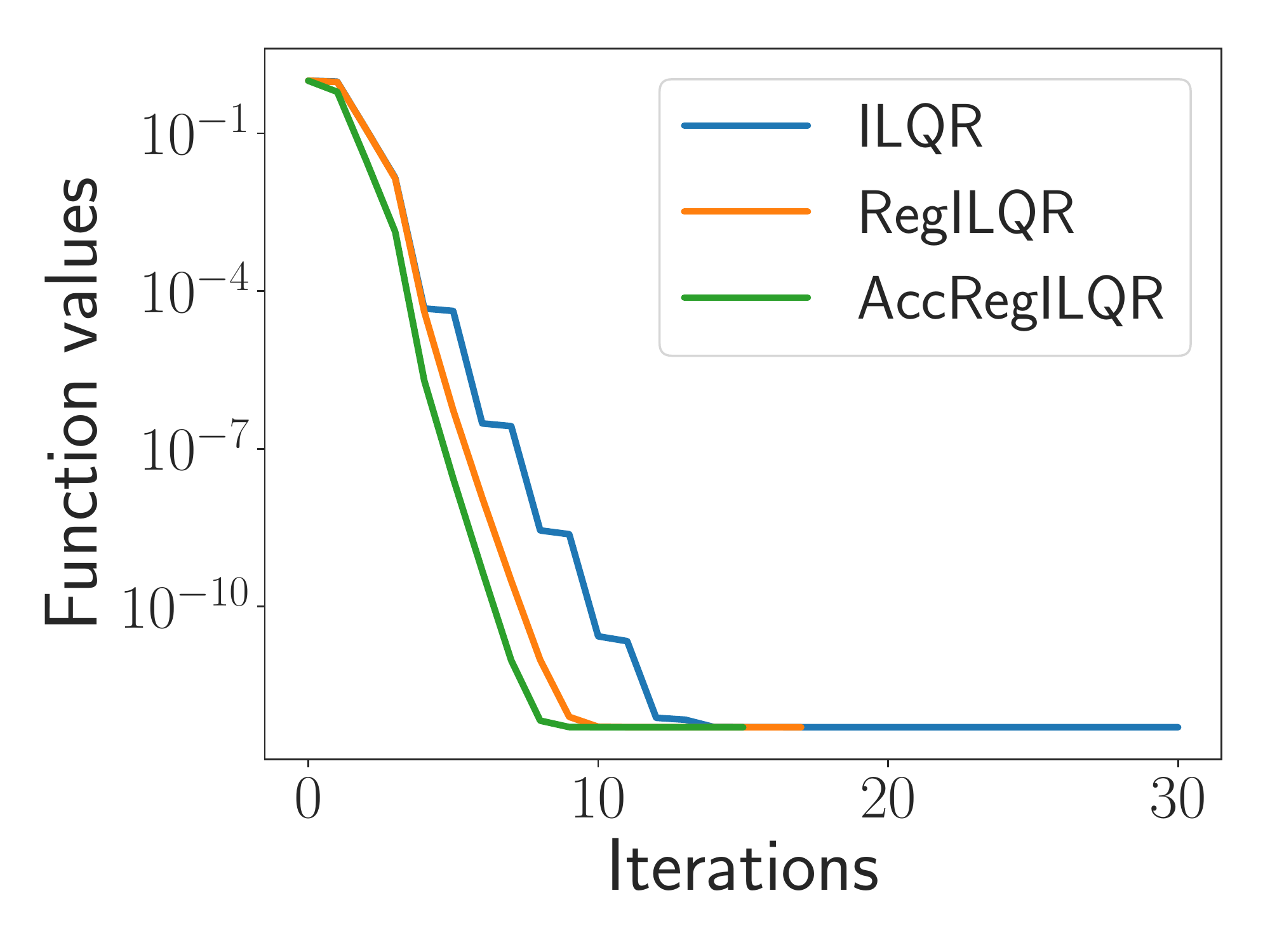}
		\includegraphics[width = 0.3\linewidth]{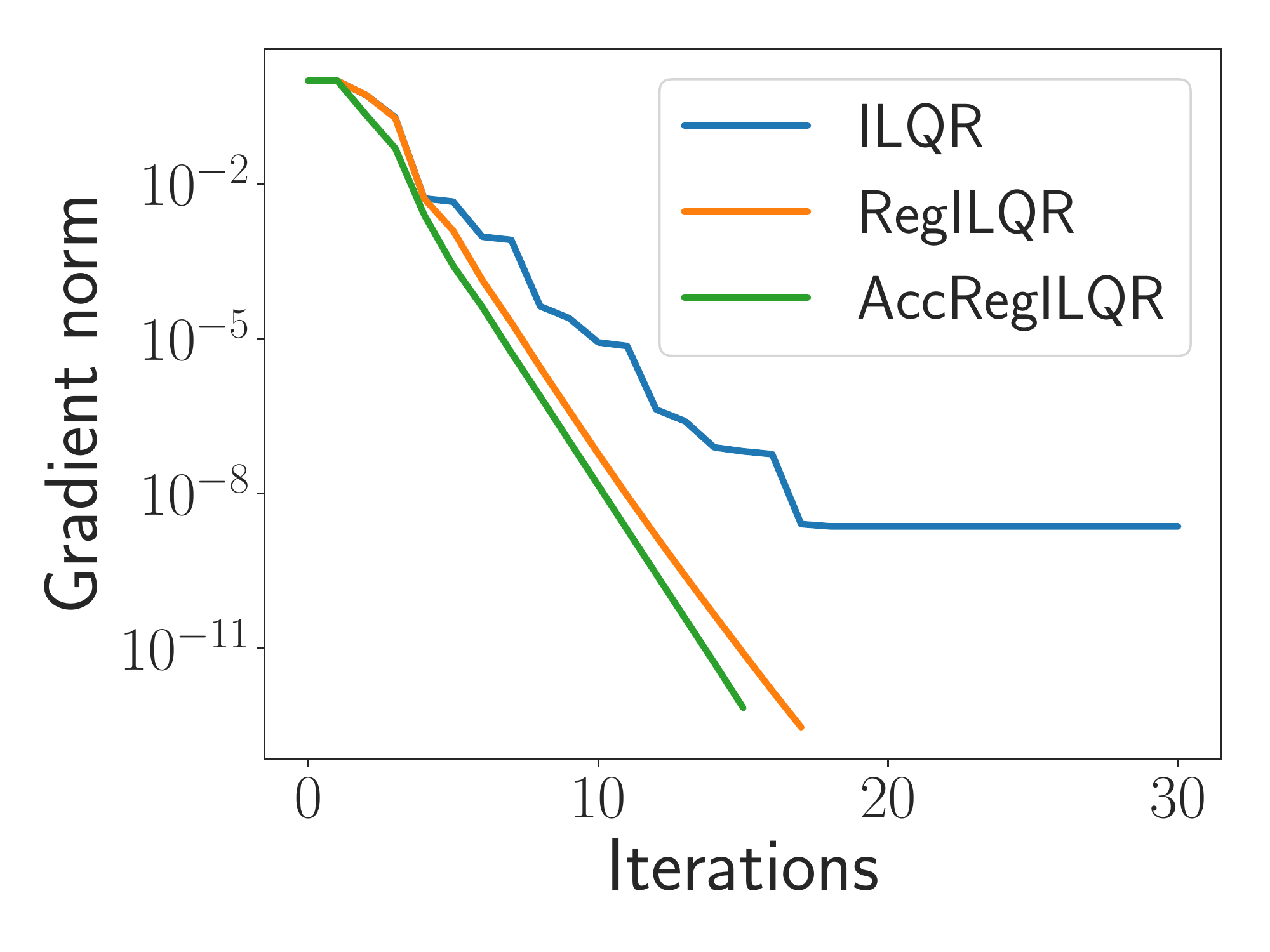}
		\caption{Convergence of ILQR, regularized ILQR and accelerated regularized ILQR on the inverted pendulum (top) and two-link arm (bottom) control problems for an horizon $\tau=100$.\label{fig:exp}}
	\end{center}
\end{figure}

\clearpage
\section*{Acknowledgements}
We would like to thank Aravind Rajeswaran for pointing out additional references.
This work was funded by NIH R01 (\#R01EB019335), NSF CCF (\#1740551), CPS (\#1544797), DMS (\#1651851), DMS (\#1839371), NRI (\#1637748), ONR, RCTA, Amazon, Google, and Honda. 

\bibliography{../bib_ctrl_dl}

\begin{thebibliography}{37}
\providecommand{\natexlab}[1]{#1}
\providecommand{\url}[1]{\texttt{#1}}
\expandafter\ifx\csname urlstyle\endcsname\relax
  \providecommand{\doi}[1]{doi: #1}\else
  \providecommand{\doi}{doi: \begingroup \urlstyle{rm}\Url}\fi

\bibitem[Abadi et~al.(2015)Abadi, Agarwal, Barham, Brevdo, Chen, Citro,
  Corrado, Davis, Dean, Devin, Ghemawat, Goodfellow, Harp, Irving, Isard, Jia,
  Jozefowicz, Kaiser, Kudlur, Levenberg, Man\'{e}, Monga, Moore, Murray, Olah,
  Schuster, Shlens, Steiner, Sutskever, Talwar, Tucker, Vanhoucke, Vasudevan,
  Vi\'{e}gas, Vinyals, Warden, Wattenberg, Wicke, Yu, and Zheng]{Tens15}
Abadi, M., Agarwal, A., Barham, P., Brevdo, E., Chen, Z., Citro, C., Corrado,
  G.~S., Davis, A., Dean, J., Devin, M., Ghemawat, S., Goodfellow, I., Harp,
  A., Irving, G., Isard, M., Jia, Y., Jozefowicz, R., Kaiser, L., Kudlur, M.,
  Levenberg, J., Man\'{e}, D., Monga, R., Moore, S., Murray, D., Olah, C.,
  Schuster, M., Shlens, J., Steiner, B., Sutskever, I., Talwar, K., Tucker, P.,
  Vanhoucke, V., Vasudevan, V., Vi\'{e}gas, F., Vinyals, O., Warden, P.,
  Wattenberg, M., Wicke, M., Yu, Y., and Zheng, X.
\newblock {TensorFlow}: Large-scale machine learning on heterogeneous systems,
  2015.
\newblock URL \url{https://www.tensorflow.org/}.

\bibitem[Bellman(1971)]{Bell67}
Bellman, R.
\newblock \emph{Introduction to the mathematical theory of control processes},
  volume~2.
\newblock Academic press, 1971.

\bibitem[Bertsekas(2005)]{Bert05}
Bertsekas, D.~P.
\newblock \emph{Dynamic programming and optimal control}.
\newblock Athena Scientific, 3rd edition, 2005.

\bibitem[Bjorck(1996)]{Bjor96}
Bjorck, A.
\newblock \emph{Numerical methods for least squares problems}.
\newblock SIAM, 1996.

\bibitem[Burke(1985)]{Burk85}
Burke, J.~V.
\newblock Descent methods for composite nondifferentiable optimization
  problems.
\newblock \emph{Mathematical Programming}, 33\penalty0 (3):\penalty0 260--279,
  1985.

\bibitem[Cartis et~al.(2011)Cartis, Gould, and Toint]{cartis2011evaluation}
Cartis, C., Gould, N. I.~M., and Toint, P.~L.
\newblock On the evaluation complexity of composite function minimization with
  applications to nonconvex nonlinear programming.
\newblock \emph{SIAM Journal on Optimization}, 21\penalty0 (4):\penalty0
  1721--1739, 2011.

\bibitem[De~O.~Pantoja(1988)]{Panto88}
De~O.~Pantoja, J.
\newblock Differential dynamic programming and {N}ewton's method.
\newblock \emph{International Journal of Control}, 47\penalty0 (5):\penalty0
  1539--1553, 1988.

\bibitem[Dean et~al.(2018)Dean, Mania, Matni, Recht, and Tu]{Dean18}
Dean, S., Mania, H., Matni, N., Recht, B., and Tu, S.
\newblock Regret bounds for robust adaptive control of the linear quadratic
  regulator.
\newblock In \emph{Advances in Neural Information Processing Systems}, pp.\
  4188--4197, 2018.

\bibitem[Dontchev et~al.(2018)Dontchev, Huang, Kolmanovsky, and
  Nicotra]{dontchev2018inexact}
Dontchev, A.~L., Huang, M., Kolmanovsky, I.~V., and Nicotra, M.~M.
\newblock Inexact {N}ewton-{K}antorovich methods for constrained nonlinear
  model predictive control.
\newblock \emph{IEEE Transactions on Automatic Control}, 2018.

\bibitem[Drusvyatskiy \& Paquette(2018)Drusvyatskiy and Paquette]{Drus16}
Drusvyatskiy, D. and Paquette, C.
\newblock Efficiency of minimizing compositions of convex functions and smooth
  maps.
\newblock \emph{Mathematical Programming}, pp.\  1--56, 2018.

\bibitem[Dunn \& Bertsekas(1989)Dunn and Bertsekas]{Dunn89}
Dunn, J.~C. and Bertsekas, D.~P.
\newblock Efficient dynamic programming implementations of {N}ewton's method
  for unconstrained optimal control problems.
\newblock \emph{Journal of Optimization Theory and Applications}, 63\penalty0
  (1):\penalty0 23--38, 1989.

\bibitem[Fazel et~al.(2018)Fazel, Ge, Kakade, and Mesbahi]{Faze18}
Fazel, M., Ge, R., Kakade, S., and Mesbahi, M.
\newblock Global convergence of policy gradient methods for the linear
  quadratic regulator.
\newblock In \emph{Proceedings of the 35th International Conference on Machine
  Learning}, volume~80, 2018.

\bibitem[Griewank \& Walther(2008)Griewank and Walther]{Grie08}
Griewank, A. and Walther, A.
\newblock \emph{Evaluating derivatives: principles and techniques of
  algorithmic differentiation}.
\newblock SIAM, 2008.

\bibitem[Gr{\"u}ne \& Pannek(2017)Gr{\"u}ne and Pannek]{grune2017nonlinear}
Gr{\"u}ne, L. and Pannek, J.
\newblock \emph{Nonlinear model predictive control}.
\newblock Springer, 2017.

\bibitem[Hansen et~al.(2013)Hansen, Pereyra, and Scherer]{Hans13}
Hansen, P.~C., Pereyra, V., and Scherer, G.
\newblock \emph{Least squares data fitting with applications}.
\newblock JHU Press, 2013.

\bibitem[Jacobson \& Mayne(1970)Jacobson and Mayne]{Jaco70}
Jacobson, D.~H. and Mayne, D.~Q.
\newblock \emph{Differential Dynamic Programming}.
\newblock Elsevier, 1970.

\bibitem[Kakade \& Lee(2018)Kakade and Lee]{Kaka18}
Kakade, S.~M. and Lee, J.~D.
\newblock Provably correct automatic sub-differentiation for qualified
  programs.
\newblock In \emph{Advances in Neural Information Processing Systems}, pp.\
  7125--7135, 2018.

\bibitem[Kaltenbacher et~al.(2008)Kaltenbacher, Neubauer, and Scherzer]{Kalt08}
Kaltenbacher, B., Neubauer, A., and Scherzer, O.
\newblock \emph{Iterative regularization methods for nonlinear ill-posed
  problems}, volume~6.
\newblock Walter de Gruyter, 2008.

\bibitem[LeCun et~al.(1988)LeCun, Touresky, Hinton, and
  Sejnowski]{lecun1988theoretical}
LeCun, Y., Touresky, D., Hinton, G., and Sejnowski, T.
\newblock A theoretical framework for back-propagation.
\newblock In \emph{Proceedings of the 1988 connectionist models summer school},
  volume~1, pp.\  21--28, 1988.

\bibitem[Lewis \& Wright(2016)Lewis and Wright]{MR3511391}
Lewis, A.~S. and Wright, S.~J.
\newblock A proximal method for composite minimization.
\newblock \emph{Mathematical Programming}, 158:\penalty0 501--546, 2016.

\bibitem[Li \& Todorov(2004)Li and Todorov]{Li04}
Li, W. and Todorov, E.
\newblock Iterative linear quadratic regulator design for nonlinear biological
  movement systems.
\newblock In \emph{1st International Conference on Informatics in Control,
  Automation and Robotics}, volume~1, pp.\  222--229, 2004.

\bibitem[Li \& Todorov(2007)Li and Todorov]{Li07}
Li, W. and Todorov, E.
\newblock Iterative linearization methods for approximately optimal control and
  estimation of non-linear stochastic system.
\newblock \emph{International Journal of Control}, 80\penalty0 (9):\penalty0
  1439--1453, 2007.

\bibitem[Liao \& Shoemaker(1991)Liao and Shoemaker]{Liao91}
Liao, L.-Z. and Shoemaker, C.~A.
\newblock Convergence in unconstrained discrete-time differential dynamic
  programming.
\newblock \emph{IEEE Transactions on Automatic Control}, 36\penalty0
  (6):\penalty0 692--706, 1991.

\bibitem[Liao \& Shoemaker(1992)Liao and Shoemaker]{Liao92}
Liao, L.-Z. and Shoemaker, C.~A.
\newblock Advantages of differential dynamic programming over {N}ewton's method
  for discrete-time optimal control problems.
\newblock Technical report, Cornell University, 1992.

\bibitem[Mayne(1966)]{Mayn66}
Mayne, D.
\newblock A second-order gradient method for determining optimal trajectories
  of non-linear discrete-time systems.
\newblock \emph{International Journal of Control}, 3\penalty0 (1):\penalty0
  85--95, 1966.

\bibitem[Nesterov(2007)]{nesterov2007modified}
Nesterov, Y.
\newblock Modified {G}auss-{N}ewton scheme with worst case guarantees for
  global performance.
\newblock \emph{Optimization Methods \& Software}, 22\penalty0 (3):\penalty0
  469--483, 2007.

\bibitem[Nocedal \& Wright(2006)Nocedal and Wright]{Noce06}
Nocedal, J. and Wright, S.~J.
\newblock \emph{Numerical Optimization}.
\newblock Springer, 2nd edition, 2006.

\bibitem[Paquette et~al.(2018)Paquette, Lin, Drusvyatskiy, Mairal, and
  Harchaoui]{Paqu17}
Paquette, C., Lin, H., Drusvyatskiy, D., Mairal, J., and Harchaoui, Z.
\newblock Catalyst for gradient-based nonconvex optimization.
\newblock In \emph{21st International Conference on Artificial Intelligence and
  Statistics}, pp.\  1--10, 2018.

\bibitem[Paszke et~al.(2017)Paszke, Gross, Chintala, Chanan, Yang, DeVito, Lin,
  Desmaison, Antiga, and Lerer]{Pyto17}
Paszke, A., Gross, S., Chintala, S., Chanan, G., Yang, E., DeVito, Z., Lin, Z.,
  Desmaison, A., Antiga, L., and Lerer, A.
\newblock Automatic differentiation in {P}y{T}orch, 2017.
\newblock URL \url{https://pytorch.org/}.

\bibitem[Recht(2018)]{Rech18}
Recht, B.
\newblock A tour of reinforcement learning: The view from continuous control.
\newblock \emph{Annual Review of Control, Robotics, and Autonomous Systems},
  2018.

\bibitem[Richter et~al.(2012)Richter, Jones, and
  Morari]{richter2012computational}
Richter, S., Jones, C.~N., and Morari, M.
\newblock Computational complexity certification for real-time {MPC} with input
  constraints based on the fast gradient method.
\newblock \emph{IEEE Transactions on Automatic Control}, 57\penalty0
  (6):\penalty0 1391--1403, 2012.

\bibitem[Sideris \& Bobrow(2005)Sideris and Bobrow]{Side05}
Sideris, A. and Bobrow, J.~E.
\newblock An efficient sequential linear quadratic algorithm for solving
  nonlinear optimal control problems.
\newblock In \emph{Proceedings of the American Control Conference}, pp.\
  2275--2280, 2005.

\bibitem[Tassa et~al.(2012)Tassa, Erez, and Todorov]{tassa2012synthesis}
Tassa, Y., Erez, T., and Todorov, E.
\newblock Synthesis and stabilization of complex behaviors through online
  trajectory optimization.
\newblock In \emph{2012 IEEE/RSJ International Conference on Intelligent Robots
  and Systems}, pp.\  4906--4913. IEEE, 2012.

\bibitem[Tassa et~al.(2014)Tassa, Mansard, and Todorov]{Tass14}
Tassa, Y., Mansard, N., and Todorov, E.
\newblock Control-limited differential dynamic programming.
\newblock In \emph{IEEE International Conference on Robotics and Automation},
  pp.\  1168--1175, 2014.

\bibitem[Todorov \& Li(2003)Todorov and Li]{Todo03}
Todorov, E. and Li, W.
\newblock Optimal control methods suitable for biomechanical systems.
\newblock In \emph{Proceedings of the 25th Annual International Conference of
  the IEEE}, volume~2, pp.\  1758--1761, 2003.

\bibitem[Todorov \& Li(2005)Todorov and Li]{Todo05}
Todorov, E. and Li, W.
\newblock A generalized iterative lqg method for locally-optimal feedback
  control of constrained nonlinear stochastic systems.
\newblock In \emph{Proceedings of the American Control Conference}, pp.\
  300--306, 2005.

\bibitem[Whittle(1982)]{whittle1982optimization}
Whittle, P.
\newblock \emph{Optimization over Time}.
\newblock John Wiley \& Sons, Inc., New York, NY, USA, 1982.

\end{thebibliography}
\bibliographystyle{icml2019}

\clearpage

\appendix
\section{Notations}\label{app:notations}
We use semicolons to denote concatenation of vectors, namely for $\horizon$ $d$-dimensional vectors $a_1, \ldots, a_\horizon \in \reals^d$, we have $(a_1;\ldots; a_\horizon) \in \reals^{\horizon d}$.
The Kronecker product is denoted $\otimes$.

\subsection{Tensors}
A tensor $\mathcal{A} = (a_{i,j,k})_{i\in\{1,\ldots, d\}, j\in \{1,\ldots,n\}, k\in \{1,\ldots, p\}} \in \reals^{d \times n \times p}$ is represented as list of matrices $\mathcal{A} = (A_1,\ldots, A_p)$ where $A_k = (a_{i,j,k})_{i\in\{1,\ldots, d\}, j\in \{1,\ldots,n\}} \in \reals^{d\times n}$ for $ k\in \{1,\ldots p\}$. Given matrices $P \in \reals^{d \times d'}, Q \in \reals^{n \times n'}, R \in \reals^{p \times p'}$, we denote
\[
\mathcal{A}[P, Q, R] = \left(\sum_{k=1}^{p} R_{k,1}P^\top A_k Q,  \ldots,  \sum_{k=1}^{p} R_{k,p'}P^\top A_k Q \right) \in \reals^{d'\times n'\times p'}
\]
If $P, Q$ or $ R$ are identity matrices, we use the symbol "$\: \cdot\: $" in place of the identity matrix. For example, we denote $\mathcal{A}[P, Q, \idm_p] = \mathcal{A}[P,Q, \cdot] = \left(P^\top A_1 Q,  \ldots,  P^\top A_p Q \right)$.
If $P, Q$ or $R$ are vectors we consider the flatten object. In particular, for $x\in \reals^d, y\in \reals^n$, we denote
\[
\mathcal{A}[x, y, \cdot] =  \left(\begin{matrix}
x^\top A_1 y\\ \vdots \\ x^\top A_py
\end{matrix}
\right) \in \reals^p
\]
rather than having $\mathcal{A}[x, y, \cdot] \in \reals^{1 \times 1\times p}$.
Similarly, for $z\in \reals^p$, we have
\[
\mathcal{A}[\cdot, \cdot, z] = \sum_{k=1}^p z_kA_k \in \reals^{d\times n}.
\]
Finally note that we have for $x\in \reals^d, y\in \reals^n$ and $R \in \reals^{p \times p'}$,
\[
\mathcal{A}(x, y, R) = \left(\sum_{k=1}^{p}x^\top A_k y R_{k,1}, \ldots, \sum_{k=1}^{p}x^\top A_k y R_{k,p'}\right)^\top = R^\top \mathcal{A}[x, y, \cdot] \in \reals^{p'}.
\]
For a tensor $\mathcal{A}= (a_{i,j,k})_{i\in\{1,\ldots, d\}, j\in \{1,\ldots,n\}, k\in \{1,\ldots, p\}} \in \reals^{d \times n \times p}$, we denote 
\[
\mathcal{A}^\pi = (a_{k,i,j})_{k\in \{1,\ldots, p\}, i\in\{1,\ldots, d\}, j\in \{1,\ldots,n\}} \in \reals^{p \times  d \times n}.
\]
the tensor whose indexes have been shifted onward. We then have for matrices $P \in \reals^{d \times d'}, Q \in \reals^{n \times n'}, R \in \reals^{p \times p'}$,
\[
(\mathcal{A}[P, Q, R])^\pi = \mathcal{A}^\pi[R, P, Q].
\]
For matrices, we have $A^\pi = A^\top$ and for one dimensional vectors $x^\pi = x$.

\subsection{Gradients}
Given a state space of dimension $d$, and control space of dimension $p$, for a real function of state and control $f: \reals^{d +p} \mapsto \reals$, whose value is denoted $f(x,u)$, we decompose its gradient $\nabla f( x,  u) \in \reals^{d +p}$ on $( x,  u)$ as a part depending on the state variables  and a part depending on the control variables  as follows 
\[
\nabla f( x,  u) =\left(
\begin{matrix}
\nabla_x f( x,  u) \\
\nabla_u f( x,  u)
\end{matrix}\right) \qquad \mbox{with} \qquad \nabla_x f(x,u) \in \reals^d, \quad \nabla_u f(x,u) \in \reals^p.
\]
Similarly we decompose its Hessian $\nabla f( x,  u) \in \reals^{(d +p)\times(d +p) }$ on blocks  that correspond to the state and control variables as follows
\begin{gather*}
\nabla^2 f( x,  u) =\left(
\begin{matrix}
\nabla_{xx} f( x,  u) & \nabla_{xu} f(x,u)\\
\nabla_{ux} f( x,  u) & \nabla_{uu} f( x,  u) 
\end{matrix}\right) 
\\
\qquad \mbox{with} \qquad \nabla_{xx} f(x,u) \in \reals^{d\times d}, \quad \nabla_{uu} f(x,u) \in \reals^{p\times p}, \nabla_{xu} f(x,u) = \nabla_{ux} f( x,  u)^\top \in \reals^{d\times p}.
\end{gather*}

For a multivariate function $f :\reals^d \mapsto \reals^n$, composed of $f^{(j)}$ real functions with $j\in \{1, \ldots n\}$, we denote $\nabla f( x) = (\nabla f^{(1)}( x), \ldots, \nabla f^{(n)}( x))\in \reals^{d \times n}$, that is  the transpose of its Jacobian on $ x$, $\nabla f( x) = (\frac{\partial f^{(j)} }{\partial x_i}( x))_{\substack{1\leq i\leq d, 1\leq j\leq n}} \in \reals^{d \times n}$. 
We represent its second order information by a tensor $\nabla^2 f( x) = (\nabla^2 f^{(1)}( x), \ldots, \nabla^2 f^{(n)}( x))\in \reals^{d \times d \times n}$

We combine previous definitions to describe the dynamic functions. Given a state space of dimension $d$, a control space of dimension $p$ and an output space of dimension $d_+$, for a dynamic function $\phi: \reals^{d + p} \mapsto \reals^{d_+}$ and a pair of control and state variable $( x,  u)$, we denote $\nabla_x \phi( x,  u) = (\nabla_x \phi^{(1)}( x,  u), \ldots, \nabla_x \phi^{(d_+)}( x,  u))$ and  we define similarly $\nabla_u \phi( x,  u)$. 
For its second order information we define  $\nabla_{xx} \phi( x,  u) = (\nabla_{xx} \phi^{(1)}( x,  u), \ldots, \nabla_{xx} \phi^{(d_+)}( x,  u))$, similarly for $\nabla_{xx} \phi( x,  u)$.
Dimension of these definitions are
\begin{gather*}
\nabla_x \phi( x,  u) \in \reals^{d \times d_+}, \quad \nabla_u \phi( x,  u) \in \reals^{p \times d_{+}}, 
\nabla_{xx} \phi( x,  u) \in \reals^{d \times d \times d_+}, \quad \nabla_{uu} \phi( x,  u) \in \reals^{p \times p \times d_{+}}\\
\nabla_{xu} \phi( x,  u) \in \reals^{d \times p \times d_+}, \quad \nabla_{ux} \phi( x,  u) \in \reals^{p \times d \times d_{+}}.
\end{gather*}

\section{Dynamic programming for linear quadratic optimal control problems}\label{app:dyn_prog}
We present the dynamic programming resolution of the Linear Quadratic control problem
\begin{align}\label{eq:LQ}
\begin{split}
\underset{\substack{x_0,\ldots,x_\horizon \in \reals^d\\u_0,\ldots,u_{\horizon-1} \in \reals^p}}{\mbox{min}}  \quad &  \sum_{t=1}^{\horizon} h_t(x_t) + \sum_{t=0}^{\horizon -1} g_t(u_t) \\
\mbox{subject to} \quad & x_{t+1} =  \dyn_t(x_t, u_t) \qquad \mbox{for $t=0,\ldots,\horizon-1$}\\
& x_0 \quad =\hat x_0,
\end{split}
\end{align}
with 
\begin{align*}
 \dyn(x, u) = \Dyn_{t,x}^\top x + \Dyn_{t,u}^\top u, \qquad
h_t(x)  = h_{t,x}^\top x + \frac{1}{2} x^\top H_{t,xx},  \quad \mbox{and} \quad
g_t(u)  = g_{t,u}^\top u  + \frac{1}{2}u^\top G_{t,uu} u,
\end{align*}
where $\Dyn_{t,x} \in \reals^{d \times d},  \Dyn_{t,u} \in \reals^{p \times d}$, $h_{t,x} \in \reals^{d},  g_{t,u} \in \reals^{p }, H_{t,xx} \in \reals^{d \times d }, G_{t, uu} \in \reals^{p \times p }$ with $H_{t,xx} \succeq 0$ and $G_{t,uu}\succ 0$. 

Dynamic programming applied to this problem is presented in Algo.~\ref{algo:LQ} it is based on the following proposition that computes cost-to-go functions as quadratics.

\begin{proposition}\label{prop:LQR_comput}
	Cost-to-go functions~\eqref{eq:cost-to-go} to minimize are quadratics of the form, for $t\in {0,\ldots,\horizon}$,
	\begin{equation}\label{eq:cost-to-go_quad}\tag{$\mathcal{H}_t$}
	c_t (x) = \frac{1}{2}x^\top C_{t,xx} x + c_{t,x}^\top x, \quad \mbox{with $C_{t,xx} \succeq 0$}.
	\end{equation}
	The optimal control at time $t$ from a state $x_t$ reads for $t\in \{0,\ldots,\horizon-1\}$,
	\[
	u^*_t(x_t) =  K_t x_t + k_t
	\]
	where $C_{t,xx}, c_{t,x}, K_t, k_t$ are defined recursively in lines~\ref{line:dyn_prog0},~\ref{line:dyn_prog1},~\ref{line:dyn_prog2},~\ref{line:dyn_prog3} of Algo.~\ref{algo:LQ}.
\end{proposition}
\begin{proof}
We prove recursively~\eqref{eq:cost-to-go_quad}. By definition of the cost-to-function, we have
\begin{align*}
c_\horizon (x) & =  h_\horizon(x), \quad \mbox{so} \quad  C_{\horizon,xx} = H_{\horizon,xx}, \quad c_{\horizon,x} = h_{\horizon,x},
\end{align*}
and by assumption on the costs $C_{\horizon,xx} = H_{\horizon,xx} \succeq 0$, which ensures $(\mathcal{H}_\horizon)$.

Then for $0\leq t\leq \horizon-1$, assume $(\mathcal{H}_{t+1})$, we search to compute
\begin{align*}
c_t (x)  = h_t(x) + \min_u \{ g_t(u) +  c_{t+1}(\dyn_t(x, u))\}
\end{align*}
To follow the computations, we drop the dependency on time and denote by superscript $'$ the quantities at time $t+1$, e.g. $ c' =  c_{t+1}$.
We therefore search an expression for 
\begin{equation}\label{eq:comput_W}
c(x)  = h(x) + \min_u \{ g(u) +  c'( \dyn(x, u)) \} = \min_u W(x,u).
\end{equation}
The function $W$ is a quadratic in $(x,u)$ of the form
\[
W(x,u) =  w_x^\top x +  w_u^\top u +  \frac{1}{2} x^\top W_{xx} x + \frac{1}{2}u^\top W_{uu} u + u^\top  W_{ux} x.
\]
Developing the terms in \eqref{eq:comput_W} we have
\begin{align*}
W(x,u) & =   h_{x}^\top x + g_{u}^\top u +  \frac{1}{2} x^\top H_{xx} x + \frac{1}{2}u^\top G_{uu} u  \\
& + \frac{1}{2}(  \Dyn_{x}^\top x + \Dyn_{u}^\top u)^\top C'_{xx}( \Dyn_{x}^\top x + \Dyn_{u}^\top u) + ( \Dyn_{x}^\top x + \Dyn_{u}^\top u)^\top c'_{x}.
\end{align*}
By identification  and using that $C'_{xx}$ is symmetric, we get 
\begin{align*}
w_x & = g_x + \Dyn_x c'_x  & W_{xx} & = H_{xx} + \Dyn_x C'_{xx}\Dyn_x^\top \\
w_u & = g_u + \Dyn_u c'_x  & W_{uu} & = G_{uu} + \Dyn_u C'_{xx}\Dyn_u^\top  \\
& & W_{ux} & = \Dyn_u C'_{xx}\Dyn_x^\top  .
\end{align*}
By assumption $G_{uu} \succ 0$ and by $(H_{t+1})$, $C'_{xx} \succeq 0$, therefore $W_{uu} \succ 0$, minimization in $u$ in \eqref{eq:comput_W} is possible and reads
\begin{align*}
c (x) & =w_x^\top x + \frac{1}{2} x^\top W_{xx} x - \frac{1}{2} (W_{ux} x +w_u)^\top W_{uu}^{-1} (W_{ux} x +w_u),
\end{align*}
with optimal control variable
\[
u^*(x) =  - W_{uu}^{-1} (W_{ux} x +w_u).
\]
Cost-to-go $c$ is then defined by (ignoring the constant terms for the minimization)
\begin{align*}
c_x & = w_x - W_{ux}^\top W_{uu} ^{-1} w_u\\
C_{xx} & = W_{xx} - W_{ux}^\top W_{uu}^{-1} W_{ux}.
\end{align*}
Denote then $C_{xx}^{' 1/2}$ a squared root matrix of $C'_{xx} \succeq 0$ such that $(C_{xx}^{' 1/2})^\top C_{xx}^{' 1/2} = C'_{xx}$. 
Developing the terms in the last equation gives
\begin{align*}
C_{xx} & = W_{xx} - W_{ux}^\top W_{uu}^{-1} W_{ux} \\
& = H_{xx} + \Dyn_x C'_{xx}\Dyn_x^\top - \Dyn_x C'_{xx} \Dyn_u^\top ( G_{uu} + \Dyn_u C'_{xx}\Dyn_u^\top )^{-1} \Dyn_u C'_{xx}\Dyn_x^\top \\
& = H_{xx} + \Dyn_x C_{xx}^{' 1/2 \top}\left( I - C_{xx}^{' 1/2} \Dyn_u^\top ( G_{uu} + \Dyn_u C_{xx}^{'}\Dyn_u^\top )^{-1} \Dyn_u C_{xx}^{' 1/2 \top} \right) C_{xx}^{' 1/2} \Dyn_x^\top \\
& = H_{xx} + \Dyn_x C_{xx}^{' 1/2 \top}\left( I + C_{xx}^{' 1/2}\Dyn_u^\top G_{uu}^{-1}\Dyn_u C_{xx}^{' 1/2 \top}\right)^{-1}C_{xx}^{' 1/2} \Dyn_x^\top\succeq 0,
\end{align*}
where we used Sherman-Morrison-Woodbury's formula in the last equality. This therefore proves $(\mathcal{H}_t)$ and the recurrence.
\end{proof}

\begin{algorithm}[H]
	\caption{Dynamic Programming for Linear Dynamics, Quadratic  Convex Costs
		\eqref{eq:LQ} \label{algo:LQ}}
	\begin{algorithmic}[1]
		\Require{Initial state $\hat x_0$, \\
			Quadratic costs defined by  $h_{t,x}, H_{t,xx}$ with $H_{t,xx} \succeq 0$ for $t=1,\ldots,\horizon$ and  $g_{t,u}, G_{t,uu} $ with $G_{t,uu} \succ 0$,  for $t=0,\ldots,\horizon-1$\\
			Linear dynamics defined by $\Dyn_{t,x}, \Dyn_{t,u}$ for $t=0,\ldots,\horizon-1$}
		\Statex{\textbf{Backward pass}}
		\State{\textbf{Initialize} $C_{\horizon,xx} = H_{\horizon,xx}, c_{\horizon,x} = h_{\horizon,x}$ \label{line:dyn_prog0}}
		\For{$t = \horizon-1,\ldots, 0$ }
		\State{
			Define
			\begin{align*}
			w_{t,x} & = h_{t,x} + \Dyn_{t,x} c_{t+1,x} & W_{t,xx} & = H_{t,xx} + \Dyn_{t,x} C_{t+1,xx}\Dyn_{t,x}^\top \\
			w_{t,u} & = g_{t,u} + \Dyn_{t,u} c_{t+1,u} & W_{t,uu} & = G_{t,uu} + \Dyn_{t,u} C_{t+1,xx}\Dyn_{t,u}^\top \\
			& & W_{t,ux} & = \Dyn_{t,u} C_{t+1,xx}\Dyn_{t,x}^\top 
			\end{align*}
			\label{line:dyn_prog1}
		}
		\State{Compute
			\begin{align*}
			C_{t,xx} = W_{t,xx} - W_{t,ux}^\top W_{t,uu}^{-1} W_{t,ux} \qquad  \qquad
			c_{t,x}  = w_{t,x} - W_{t,ux}^\top W_{t,uu} ^{-1} w_{t,u}
			\end{align*}
			\label{line:dyn_prog2}
		}
		\State{Store \begin{align*}
			K_t = - W_{t,uu}^{-1} W_{t,ux} \qquad \qquad
			k_t = - W_{t,uu}^{-1} w_{t,u}
			\end{align*}
			\label{line:dyn_prog3}
		}
		\EndFor
		\Statex{\textbf{Roll-out pass}}
		\State{\textbf{Initialize} $ x_0 = \hat x_0$}
		\For{$t =0,\ldots,\horizon-1$}
		\begin{align*}
		u_t^* = K_t x_t + k_t \qquad \qquad
		x_{t+1} = \Dyn_{t,x}^\top x_t + \Dyn_{t,u}^\top u_t^*
		\end{align*}
		\EndFor
		\Ensure{Optimal $u_0^*,\ldots,u_{\horizon-1}^*$ for~\eqref{eq:LQ}}
	\end{algorithmic}
\end{algorithm}
\clearpage

\section{Control algorithms detailed}\label{app:compa_ILQR}
We detail the complete implementation of the control algorithms presented in Section~\ref{sec:prox_lin}. We detail the simple implementations of ILQR~\citep{Li04}, ILQG~\citep{Todo05, Li07},  the variant of ILQG~\citep{tassa2012synthesis} and DDP~\citep{Mayn66}. Various line-searches have been proposed as in e.g.~\citep{tassa2012synthesis}. We leave their analysis for future work.

\subsection{ILQR \citep{Li04} and regularized ILQR }
ILQR and regularized ILQR  amount to solve
\begin{align}
\min_{\substack{y_0, \ldots y_\horizon \in \reals^d \\ v_0,\ldots, v_{\horizon-1} \in \reals^p} } &
\sum_{t=1}^{\horizon} q_{h_t}(x_t + y_t; x_t)  + \sum_{t=0}^{\horizon-1} q_{g_t}(u_t+v_t; u_t)+ \textcolor{puorange}{ \frac{1}{ \bm {2\stepsize}}\|\bm {v_t}\|_2^2} \label{eq:ILQR_regularized}\\
\mbox{\textup{subject to}} \quad & y_{t+1} = \Dyn_{t,x}^\top y_t + \Dyn_{t,u}^\top v_t \nonumber \\
& y_0 = 0, \nonumber
\end{align}
where $\Dyn_{t,x} = \nabla_x \dyn_t(x_t, u_t)$, $\Dyn_{t,u} = \nabla_u \dyn_t(x_t, u_t)$, $x_t = \tilde x_t(\bar u)$ and
\begin{align*}
q_{h_t}(x_t + y_t; x_t) & = h_t(x_t) + h_{t,x}^\top y_t  +  \frac{1}{2}y_t^\top H_{t,xx} y_t
\\  q_{g_t}(u_t + v_t; u_t) &  = g_t(u_t) + g_{t,u}^\top v_t  +  \frac{1}{2}v_t ^\top  G_{t,uu} v_t,
\end{align*}
where $h_{t,x} = \nabla h_t(x_t)$, $H_{t,xx} = \nabla^2 h_t(x_t) \succeq 0$, $g_{t,u} = \nabla g_t(u_t)$, $G_{t,uu} = \nabla^2 g_t(u_t) \succeq 0$.

For $\red{\bm \stepsize}=+\infty$, this amounts to the ILQR method. If $h$, $g$ are quadratics this is a Gauss-Newton step, otherwise it amounts to a generalized Gauss-Newton step. 
For $\red{\bm \stepsize}<+\infty$, if $h$, $g$ are quadratics this amounts to a regularized Gauss-Newton step~\eqref{eq:prox_linear_step}, otherwise it amounts to a Levenberg-Marquardt step~\eqref{eq:LM_step}.

The steps~\eqref{eq:ILQR_regularized} are detailed in Algo.~\ref{algo:ILQR_step} based on the derivations made in Section~\ref{app:dyn_prog}. The complete methods are detailed in Algo.~\ref{algo:ILQR} for the ILQR algorithm and in Algo.~\ref{algo:ILQR_reg} for the regularized step. For the classical ILQR method, an Armijo line-search can be used to ensure decrease of the objective as we present in Algorithm~\ref{algo:ILQR}. Different line-searches were proposed, like the one used in Differential dynamic Programming (see below). A line-search for the regularized ILQR is proposed in Algo.~\ref{algo:prox_lin_backtrack} based on the regularized Gauss-Newton method's analysis. Note that theoretically a constant step-size can also be used a shown in Sec.~\ref{sec:prox_lin}. More sophisticated line-searches with proven convergence rates are left for future work. Note that they were experimented in~\citep{Li07} for the ILQG method.

\subsection{ILQG~\citep{Todo05, Li07} and regularized ILQG}
The ILQG method as presented in~\citep{Todo05, Li07} consists in approximating the linearized trajectory by a Gaussian as presented in Sec.~\ref{sec:oracles_ctrl} and solve the corresponding dynamic programming. As for ILQR we add a proximal term to account for the inaccuracy of the model. 
Formally it amounts to solve 
\begin{align}
&\min_{\substack{y_0, \ldots y_\horizon \in \reals^d \\ v_0,\ldots, v_{\horizon-1} \in \reals^p} } \quad  
\sum_{t=1}^{\horizon} \Expect_{\bar w}\left[q_{h_t}(x_t + y_t; x_t)\right]  + \sum_{t=0}^{\horizon-1}  q_{g_t}(u_t+v_t; u_t) + \textcolor{puorange}{ \frac{1}{ \bm {2\stepsize}}\|\bm {v_t}\|_2^2} \label{eq:LQG_step}\\
&\mbox{\textup{s.t.}} \quad  y_{t+1} = \Dyn_{t,x}^\top y_t + \Dyn_{t,u}^\top v_t + \Dyn_{t, w}^\top w_t  + \Ddyn_{t,u, w}[ v_t, w_t, \cdot], \nonumber\\
&\phantom{\mbox{\textup{s.t}} \quad \quad \: } y_0 = 0,  \nonumber
\end{align}
where $\Dyn_{t,x} = \nabla_x \dyn_t(x_t, u_t, 0)$, $\Dyn_{t,u} = \nabla_u \dyn_t(x_t, u_t, 0)$,  $\Dyn_{t,w} = \nabla_w \dyn_t(x_t, u_t, 0)$,  $\Ddyn_{t, u, w} =\nabla_{u w}^2 \dyn_t(x_t, u_t, 0) $, $x_t =  \tilde x_t(\bar u, 0)$, and
\begin{align*}
q_{h_t}(x_t + y_t; x_t) & = h_t(x_t) + h_{t,x}^\top y_t  +  \frac{1}{2}y_t^\top H_{t,xx} y_t
\\  q_{g_t}(u_t + v_t; u_t) &  = g_t(u_t) + g_{t,u}^\top v_t  +  \frac{1}{2}v_t ^\top  G_{t,uu} v_t,
\end{align*}
where $h_{t,x} = \nabla h_t(x_t)$, $H_{t,xx} = \nabla^2 h_t(x_t) \succeq 0$, $g_{t,u} = \nabla g_t(u_t)$, $G_{t,uu} = \nabla^2 g_t(u_t) \succeq 0$.

The classical ILQG algorithm did not take into account a regularization as presented in Algo.~\ref{algo:ILQG} where we present an Armijo line-search though other line-searches akin to the ones made for DDP are possible. Its regularized version is presented in Algo.~\ref{algo:ILQG_reg}. They are based on solving the above problem at each step as presented in Algo.~\ref{algo:ILQG_step}. It is based on the following resolution of the dynamic programming problem.
\begin{proposition}
	Cost-to-go functions for problem~\eqref{eq:LQG_step} are quadratics of the form, for $t\in {0,\ldots,\horizon}$,
	\begin{equation}
	c_t (x) = \frac{1}{2}x^\top C_{t,xx} x + c_{t,x}^\top x, \quad \mbox{with $C_{t,xx} \succeq 0$}.
	\end{equation}
	The optimal control at time $t$ from a state $x_t$ reads for $t\in \{0,\ldots,\horizon-1\}$,
	\[
	u^*_t(x_t) =  K_t x_t + k_t
	\]
	where $C_{t,xx}, c_{t,x}, K_t, k_t$ are defined recursively in Algo.~\ref{algo:ILQG_step}.
\end{proposition}
\begin{proof}
The classical Bellman equation is replaced for problem~\eqref{eq:LQG_step} by
\begin{equation}\label{eq:cost-to-go_noisy}
\valuefunc_t(\hat y_t) = q_{h_t}(\hat y_t) + \min_{v_t}\left\{ q_{g_t}(v_t) + \frac{1}{2 \stepsize} \|v_t\|_2^2 + \Expect_{w_t}\left[ \valuefunc_{t+1}(\Dyn_{t,x}^\top y_t + \Dyn_{t,u}^\top v_t + \Dyn_{t, w}^\top w_t  + \Ddyn_{t,u, w}[ v_t, w_t, \cdot])\right]\right\}.
\end{equation}
where we denoted shortly $ q_{h_t}(\hat y_t) = q_{h_t}(x_t + y_t; x_t)$ and $q_{g_t}(v_t) = q_{g_t}(u_t+v_t; u_t)$.

The dynamic programming procedure is initialized by $c_{\horizon, x} = h_{\horizon,x}$, $C_{\horizon,xx} = H_{\horizon, xx} \succeq 0$. At iteration $t$, we seek to solve analytically the equation~\eqref{eq:cost-to-go_noisy}. For given $y$ and $v$. denoting with $'$ the quantities at time $t+1$ and omitting the index $t$ otherwise, the expectation in~\eqref{eq:cost-to-go_noisy} reads (recall that we supposed $w_t \sim \mathcal{N}(0,\id_q)$)
\begin{align*}
E = & \Expect_{w}\left[ c'(\Dyn_{x}^\top y + \Dyn_{u}^\top v + \Dyn_{w}^\top w  + \Ddyn_{u, w}[ v, w, \cdot])\right] \\
= & \Expect_{w}\Big[ {c_{x}'}^\top(\Dyn_{x}^\top y + \Dyn_{u}^\top v + \Dyn_{w}^\top w  + \Ddyn_{u, w}[ v, w, \cdot]) \\
& \phantom{\Expect_{w}\Big[} +  \frac{1}{2}(\Dyn_{x}^\top y + \Dyn_{u}^\top v + \Dyn_{w}^\top w  + \Ddyn_{u, w}[ v, w, \cdot])^\top {C'}_{xx}(\Dyn_{x}^\top y + \Dyn_{u}^\top v + \Dyn_{w}^\top w  + \Ddyn_{u, w}[ v, w, \cdot])\Big]  \\
= & {c_{x}'}^\top(\Dyn_{x}^\top y + \Dyn_{u}^\top v) + \frac{1}{2}(\Dyn_{x}^\top y + \Dyn_{u}^\top v)^\top {C'}_{xx}(\Dyn_{x}^\top y + \Dyn_{u}^\top v) \\
& + \Expect_{w}\Big[\frac{1}{2}(\Dyn_{w}^\top w  + \Ddyn_{u, w}[ v, w, \cdot])^\top {C'_{xx}}(\Dyn_{w}^\top w  + \Ddyn_{u, w}[ v, w, \cdot])\Big] 
\end{align*}
Now we have
\begin{align*}
E' = & \Expect_{w}\Big[\frac{1}{2}(\Dyn_{w}^\top w  + \Ddyn_{u, w}[ v, w, \cdot])^\top {C'_{xx}}(\Dyn_{w}^\top w  + \Ddyn_{u, w}[ v, w, \cdot])\Big] \\
= & \Expect_{w}\Big[\frac{1}{2}(\Dyn_{w}^\top w  + \Ddyn_{u, w}^\pi[ \cdot, v, w])^\top{C'_{xx}}(\Dyn_{w}^\top w + \Ddyn_{u, w}^\pi[ \cdot, v, w])\Big] \\
= & \Expect_{w}\Big[\frac{1}{2}w^\top(\Dyn_{w}^\top  + \Ddyn_{u, w}^\pi[ \cdot, v, \cdot])^\top{C'_{xx}}(\Dyn_{w}^\top + \Ddyn_{u, w}^\pi[ \cdot, v, \cdot]) w\Big] \\
= & \frac{1}{2}\Tr((\Dyn_{w}^\top  + \Ddyn_{u, w}^\pi[ \cdot, v, \cdot])^\top{C'_{xx}}(\Dyn_{w}^\top + \Ddyn_{u, w}^\pi[ \cdot, v, \cdot])) 
\end{align*}
where $\Ddyn_{u, w}^\pi = (\Psi_1, \ldots, \Psi_q) \in \reals^{d \times p \times q}$ denotes the shuffled tensor $\Ddyn_{u, w}$  as defined in Appendix~\ref{app:notations}. Therefore we have $\Ddyn_{u, w}^\pi[ \cdot, v, \cdot] = (\Psi_1v, \ldots, \Psi_qv) \in \reals^{d \times q}$. Now denoting $\Dyn_{w}^\top = (\psi_1, \ldots, \psi_q) \in \reals^{d \times q}$, we have 
\begin{align*}
E' = & \frac{1}{2} \Tr\left(C'_{xx}\left(\sum_{i=1}^{q} (\psi_i + \Psi_i v )(\psi_i + \Psi_i v )^\top \right)\right) 
=  \sum_{i=1}^{q} \frac{1}{2}(\psi_i + \Psi_i v)^\top C'_{xx} (\psi_i + \Psi_i v) 
=  \theta_0 + v^\top \theta + \frac{1}{2}v^\top \Theta v
\end{align*}
where 
\begin{align*}
\theta_0  = \frac{1}{2}\Tr(\Dyn_w C'_{xx} \Dyn_w^\top),  \qquad 
\theta  = \sum_{i=1}^q \Psi_i^\top C'_{xx} \psi_i, \qquad
\Theta = \sum_{i=1}^q \Psi_i^\top C'_{xx} \Psi_i \succeq 0.
\end{align*}
The computation of the cost-to-go function~\eqref{eq:cost-to-go_noisy} reads then (ignoring the constant terms such as $\theta_0$),
\begin{equation}
\valuefunc(\hat y) = h_x^\top y + \frac{1}{2} y^\top H_{xx} y+ \min_{v}\left\{ \tilde g_u^\top v + \frac{1}{2} v^\top \tilde G_{uu} v + \valuefunc'(\Dyn_{x}^\top y + \Dyn_{u}^\top v)\right\}.
\end{equation}
where 
\[
\tilde g_u = g_u+ \theta \qquad  \tilde G_{uu} = G_{uu} + \Theta + \stepsize^{-1}\idm_p
\]
The rest of the computations follow form Prop.~\ref{prop:LQR_comput}. The positive semi-definiteness of $C_{xx}$ is ensured since $\Theta \succeq 0$.
\end{proof}

\subsection{ILQG as in \citep{tassa2012synthesis}}
For completeness, we detail the implementation of ILQG steps as presented in~\citep{tassa2012synthesis} in Algo.~\ref{algo:ILQG_tassa_step}. The overall method consists in simply iterating these steps, i.e., starting from $\bar u_0$,
\[
\bar u_{k+1} = \texttt{iLQG\_step}(\bar u _k).
\]
This algorithm is the same as the  Differential Dynamic Programming algorithm (see below) except that second order information of the trajectory is not taken into account. We present a simple line-search as the one used for DDP, although more refined line-searches were proposed in~\citep{tassa2012synthesis}.

\subsection{Differential Dynamic Programming \citep{Tass14}} \label{sec:ddp}
For completeness, we present a detailed implementation of Differential Dynamic Programing (DDP) as presented in \citep{Tass14}. Note that several variants of the algorithms exist, especially in the implementation of line-searches, see e.g. \citep{Panto88}. A single step of the differential dynamic programming approach is described in Algo.~\ref{algo:DDP_step}. The overall algorithm simply consists in iterating that step as, starting from $\bar u_0$,
\[
\bar u_{k+1} = \texttt{DDP\_step}(\bar u _k).
\]

 We present the rationale behind the computations as explained in~\citep{Tass14} and precise the discrepancy between the motivation and the implementation. Formally, at a given command $\bar u$ with associated trajectory $\bar x = \tilde x (\bar u)$, the approach consists in approximating the cost-to-go functions as 
\begin{align}\label{eq:ddp_bellman}
\valuefunc_{t}( y) = & q_{h_t}(x_t+ y; x_t) + \min_{v} \left\{q_{g_t}(u_t +v; u_t) + q_{\valuefunc_{t+1} \circ \dyn_t}(x_t + y,u_t + v; x_t, u_t)\right\},
\end{align}
where for a function $f(y)$, $q_f(x + y; x) = f(x) + \nabla f(x)^\top y + y^\top \nabla^2 f(x) y/2$ denotes its second order approximation around $x$. They will take the following form, ignoring constant terms for the minimization,
\[
c_t(y) = \frac{1}{2}y^\top C_{t,xx} y +  c_{t,x}^\top y.
\]
The initial value function is an approximation of the last cost around the current point, i.e.
\[
c_\horizon(y) = \frac{1}{2}y^\top \nabla^2 h(x_\horizon) y +  \nabla h(x_\horizon)^\top y,
\]
where we identify $C_{\horizon,xx} = \nabla^2 h(x_\horizon)$, $c_{\horizon,x} = \nabla h(x_\horizon)$.
At time $t+1$ given an approximate value function $c_{t+1}$, step~\eqref{eq:ddp_bellman} involves computing a second order approximation around points $(x_t, u_t)$ of
\[
M_{t+1}(x,u) = h_t(x) + g_t(u) + c_{t+1}(\phi_{t}(x, u)).
\]
Denote $W_t(y,v) = q_{M_{t+1}}(x_t + y,u_t +v;x_t, u_t)$. We have, denoting $z = (y; v) \in \reals^{d + p}$,
\begin{align}
	W_t(y,v) = & \:q_{M_{t+1}}(x_t + y,u_t +v;x_t, u_t) \label{eq:ddp_formulation} \\
	= & \: M_{t+1}(x_t, u_t) + \nabla h_t (x_t)^\top y + \frac{1}{2}y^\top \nabla ^2 h_t(x_t) y + \nabla g_t(u_t)^\top v + \frac{1}{2} v^\top \nabla^2 g_t(u_t) v \nonumber\\ 
	& \: + \nabla c_{t+1}(\phi_t(x_t, u_t))^\top \nabla \phi_t(x_t, u_t)^\top z + \frac{1}{2}z^\top \nabla \phi_t(x_t, u_t) \nabla^2 c_{t+1}(\phi(x_t, u_t))\nabla \phi_t(x_t, u_t)^\top z \nonumber\\
	& \: + \frac{1}{2}\nabla^2 \phi(x_t, u_t) [z, z, \nabla c_{t+1}(\phi(x_t, u_t))], \nonumber
\end{align}
where $\nabla \phi_t(x_t, u_t)^\top z =  (\nabla_x \phi_t(x_t, u_t)^\top y+ \nabla_u \phi_t(x_t, u_t)^\top v) $, and 
\begin{align*}
\nabla^2 \phi(x_t, u_t) [z, z, \nabla c_{t+1}(\phi(x_t, u_t))]  = \: & \nabla_{xx}^2 \phi(x_t, u_t)[y, y, \nabla c_{t+1}(\phi(x_t, u_t))] \\
&  + \nabla_{uu}^2 \phi(x_t, u_t)[v, v, \nabla c_{t+1}(\phi(x_t, u_t))] \\
& + 2 \nabla_{ux}^2 \phi(x_t, u_t)[v, y, \nabla c_{t+1}(\phi(x_t, u_t))].
\end{align*}
By parameterizing $W_t(y,v)$ as
\begin{align*}
	W_t(y,v) & = w_{t,0} +  w_{t,x}^\top y +  w_{t,u}^\top v +  \frac{1}{2} y^\top W_{t,xx} y + \frac{1}{2}v^\top W_{t,uu} v + v^\top  W_{t,ux} y,
\end{align*}
we get after identification
\begin{align*}
\begin{array}{l}
w_{t,x} = h_{t,x} + \Dyn_{t,x} \tilde c_{t+1, x} \\
w_{t,u} = g_{t,u} + \Dyn_{t,u} \tilde c_{t+1, x} \nonumber
\end{array}
\qquad
\begin{array}{l}
W_{t,xx} = H_{t,xx} + \Dyn_{t,x} C_{t+1,xx}\Dyn_{t,x}^\top + \upphi_{t,xx} [\cdot ,\cdot ,\tilde c_{t+1, x}] \\
W_{t,ux} = \Dyn_{t,u} C_{t+1,xx}\Dyn_{t,x}^\top + \upphi_{t,ux} [\cdot ,\cdot ,\tilde c_{t+1, x}] \nonumber\\
W_{t,uu} = G_{t,uu} + \Dyn_{t,u} C_{t+1,xx}\Dyn_{t,u}^\top + \upphi_{t,uu}[\cdot ,\cdot ,\tilde c_{t+1, x}],
\end{array}
\end{align*}
where 
\begin{gather*}
h_{t, x} = \nabla h_t(x_t), \qquad H_{t,xx} = \nabla^2 h_t(x_t), \qquad g_{t, u} = \nabla g_t(u_t), \qquad G_{t,uu} = \nabla^2 g_t(u_t),  \\
\Dyn_{t,x}  = \nabla_x \phi_t(x_t, u_t), \qquad 
\Dyn_{t,u}  = \nabla_u \phi_t(x_t, u_t), \qquad
\\
\upphi_{t,xx}  = \nabla^2_{xx} \phi_t(x_t, u_t), \qquad 
\upphi_{t,uu}   = \nabla^2_{uu} \phi_t(x_t, u_t), \qquad 
\upphi_{t,ux}   = \nabla^2_{ux} \phi_t(x_t, u_t), \qquad \\
\tilde c_{t+1, x} = c_{t+1, x} + C_{t+1}\phi_t(x_t, u_t).
\end{gather*}
Here rather than using $\tilde c_{t+1, x}$ as advocated by the idea of approximating the Bellman equation around the current iterate, the implementation uses $c_{t+1}$.
To minimize the resulting function in $v$, one must ensure that $W_{t,uu}$ is invertible. This is done by adding a small regularization $\lambda$ such that $W_{t,uu} := W_{t,uu} + \lambda \idm_p \succ 0$ as presented in e.g. \citep{Panto88} and further explored in \citep{Tass14}.

After minimizing in  $v$, we get the new approximate value function to minimize
\[
c_t(y) = \frac{1}{2}y^\top  C_{t,xx} y +  c_{t,x}^\top y,
\]
with 
\begin{align*}
c_{t,x}  = w_{t,x} - W_{t,ux}^\top W_{t,uu} ^{-1} w_{t,u}, \qquad
C_{t,xx}  = W_{t,xx} - W_{t,ux}^\top W_{t,uu}^{-1} W_{t,ux}.
\end{align*}
Once the cost-to-go functions are computed, the next command is given by the solution of these approximated Bellman equations around the trajectory. Precisely, denote 
\[
v^*(y) = \argmin_{v} \left\{q_{g_t}(u_t +v; u_t) + q_{\valuefunc_{t+1} \circ \dyn_t}(x_t + y,u_t + v; x_t, u_t)\right\} = K_t y + k_t,
\]
where $k_t =-W_{t,uu}^{-1}w_{t,u}$ and $K_t=  -W_{t,uu}^{-1}W_{t,ux}y$.
The roll-out phase starts with $x_0^+ = \hat x_0$ and outputs the next command as
\begin{align}
	u_t^+ = u_t + v^*(x_t^+ -x_t) = u_t + K_t(x_t^+ -x_t) + k_t\qquad x_{t+1}^+ = \dyn_t(x_t^+, u_t^+).  \label{eq:ddp_rollout}
\end{align}
A line-search advocated in e.g.~\citep{Tass14} is to move along the direction given by the fixed gain $k_t$, i.e., the roll-out phase reads 
\begin{align*}
	u_t^+ = u_t + K_t(x_t^+ -x_t) + \alpha k_t\qquad x_{t+1}^+ = \dyn_t(x_t^+, u_t^+),
\end{align*}
where $\alpha$ is chosen such that the next iterate has a lower cost than the previous one, by a decreasing  line-search initialized at $\alpha=1$. More sophisticated line-searches were also proposed~\citep{Mayn66, Liao92}.

\begin{algorithm}[t]
	\caption{$\texttt{ILQR\_step}(\bar u, \stepsize)$. ILQR step \citep{Li04} ($\stepsize=\infty$) or regularized ILQR step with step-size $\stepsize$ [Sec.~\ref{sec:prox_lin}] on a command $\bar u$. \label{algo:ILQR_step}}
	\begin{algorithmic}[1]
		\Statex{\textbf{Inputs:} Command $\bar u = (u_0;\ldots, u_{\horizon-1})$, step-size $\stepsize$, initial state $\hat x_0$ convex twice differentiable costs $h_t$ for $t = 0,\ldots \horizon$ with $h_0=0$, convex twice differentiable  penalties $g_t$ for $t=0,\ldots, \horizon-1$
			differentiable	dynamics $\dyn_t$ for $t=0,\ldots, \horizon-1 $
		}
		\Statex{ \bf Forward pass:}
		\State{Set $x_0 = \hat x_0$}
		\For{$t= 0, \ldots, \horizon-1$}
		\State{Compute and store
		\begin{gather*}
		h_{t,x} = \nabla h_t(x_t), \quad H_{t,xx} = \nabla^2 h_t(x_t), \quad g_{t,u} = \nabla g_t(u_t), \quad G_{t,uu} = \nabla^2 g_t(u_t),  \\
		\Dyn_{t,x} = \nabla_x \dyn_t(x_t, u_t),\quad  \Dyn_{t,u} = \nabla_u \dyn_t(x_t, u_t).
		\end{gather*}
		}
	\State{Go to next state $x_{t+1} = \dyn_t(x_t, u_t) $}
		\EndFor
		\Statex{\bf Backward pass:}
		\State{Initialize $C_{\horizon,xx} = \nabla^2 h_\horizon(x_\horizon), c_{\horizon,x} = \nabla h_\horizon(x_\horizon)$ }
		\For{$t = \horizon-1,\ldots, 0$ }
		\State{
			Define
			\begin{align*}
			w_{t,x} & = h_{t,x} + \Dyn_{t,x} c_{t+1,x} & W_{t,xx} & = H_{t,xx} + \Dyn_{t,x} C_{t+1,xx}\Dyn_{t,x}^\top \\
			w_{t,u} & = g_{t,u} + \Dyn_{t,u} c_{t+1,x} & W_{t,uu} & = G_{t,uu} + \Dyn_{t,u} C_{t+1,xx}\Dyn_{t,u}^\top + \red{\stepsize^{-1}\idm_p} \\
			& & W_{t,ux} & = \Dyn_{t,u} C_{t+1,xx}\Dyn_{t,x}^\top 
			\end{align*}
		}
		\State{Compute
			\begin{align*}
			C_{t,xx} = W_{t,xx} - W_{t,ux}^\top W_{t,uu}^{-1} W_{t,ux} \qquad  \qquad
			c_{t,x}  = w_{t,x} - W_{t,ux}^\top W_{t,uu} ^{-1} w_{t,u}
			\end{align*}
		}
		\State{Store \begin{align*}
			K_t = - W_{t,uu}^{-1} W_{t,ux} \qquad \qquad
			k_t = - W_{t,uu}^{-1} w_{t,u}
			\end{align*}
		}
		\EndFor
		\Statex{\bf Roll-out pass:}
		\State{\textbf{Initialize} $ y_0 = 0$}
		\For{$t =0,\ldots,\horizon-1$}
		\begin{align*}
		v_t^* = K_t y_t + k_t \qquad \qquad
		y_{t+1} = \Dyn_{t,x}^\top y_t + \Dyn_{t,u}^\top v_t^*
		\end{align*}
		\EndFor
		\Ensure{$\texttt{ILQR\_step}(\bar u, \stepsize) = \bar v^*$ where $\bar v^* = (v_0^*;\ldots;v_{\horizon-1}^*)$ is optimal for~\eqref{eq:ILQR_regularized}}
	\end{algorithmic}
\end{algorithm}

\begin{algorithm}[t]
	\caption{$\texttt{ILQG\_step}(\bar u, \stepsize)$. ILQG step \citep{Todo05, Li07} ($\stepsize=\infty$) or regularized ILQG step with stepsize $\stepsize$ [Sec.~\ref{sec:prox_lin}] on a command $\bar u$. \label{algo:ILQG_step}}
	\begin{algorithmic}[1]
		\Statex{\textbf{Inputs:} Command $\bar u = (u_0;\ldots, u_{\horizon-1})$, step-size $\stepsize$, initial state $\hat x_0$ convex twice differentiable costs $h_t$ for $t = 0,\ldots \horizon$ with $h_0=0$, convex twice differentiable  penalties $g_t$ for $t=0,\ldots, \horizon-1$
			twice differentiable noisy	dynamics $\dyn_t$ for $t=0,\ldots, \horizon-1 $ as in~\eqref{eq:noisydynamic}.
		}
		\Statex{ \bf Forward pass:}
		\State{Set $x_0 = \hat x_0$}
		\For{$t= 0, \ldots, \horizon-1$}
		\State{Compute and store
			\begin{gather*}
			h_{t,x} = \nabla h_t(x_t), \quad H_{t,xx} = \nabla^2 h_t(x_t), \quad g_{t,u} = \nabla g_t(u_t), \quad G_{t,uu} = \nabla^2 g_t(u_t),  \\
			\Dyn_{t,x} = \nabla_x \dyn_t(x_t, u_t, 0),\quad  \Dyn_{t,u} = \nabla_u \dyn_t(x_t, u_t, 0), \quad \Dyn_{t,w} = \nabla_w \dyn_t(x_t, u_t, 0), \quad \Ddyn_{t, u, w} =\nabla_{u w}^2 \dyn_t(x_t, u_t, 0) 
			\end{gather*}
		}
		\State{Go to the exact next state $x_{t+1} = \dyn_t(x_t, u_t, 0) $}
		\EndFor
		\Statex{\bf Backward pass:}
		\State{Initialize $C_{\horizon,xx} = \nabla^2 h_\horizon(x_\horizon), c_{\horizon,x} = \nabla h_\horizon(x_\horizon)$ }
		\For{$t = \horizon-1,\ldots, 0$ }
		\State{
			Denoting $ (\psi_{t,1}, \ldots, \psi_{t,q}) = \Dyn_{t,w}^\top$ and $(\Psi_{t,1}, \ldots, \Psi_{t,q}) = \Ddyn_{t, u, w}^\pi$, define
			\begin{align*}
			w_{t,x} & = h_{t,x} + \Dyn_{t,x} c_{t+1,x} & W_{t,xx} & = H_{t,xx} + \Dyn_{t,x} C_{t+1,xx}\Dyn_{t,x}^\top \\
			w_{t,u} & = g_{t,u} + \Dyn_{t,u} c_{t+1,x} + \sum_{i=1}^q \Psi_{t,i}^\top C_{t+1,xx} \psi_{t,i} & W_{t,uu} & = G_{t,uu} + \Dyn_{t,u} C_{t+1,xx}\Dyn_{t,u}^\top + \sum_{i=1}^q \Psi_{t,i}^\top C_{t+1,xx} \Psi_{t,i} + \red{\stepsize^{-1}\idm_p} \\
			& & W_{t,ux} & = \Dyn_{t,u} C_{t+1,xx}\Dyn_{t,x}^\top 
			\end{align*}
		}
		\State{Compute
			\begin{align*}
			C_{t,xx} = W_{t,xx} - W_{t,ux}^\top W_{t,uu}^{-1} W_{t,ux} \qquad  \qquad
			c_{t,x}  = w_{t,x} - W_{t,ux}^\top W_{t,uu} ^{-1} w_{t,u}
			\end{align*}
		}
		\State{Store \begin{align*}
			K_t = - W_{t,uu}^{-1} W_{t,ux} \qquad \qquad
			k_t = - W_{t,uu}^{-1} w_{t,u}
			\end{align*}
		}
		\EndFor
		\Statex{\bf Roll-out pass:}
		\State{\textbf{Initialize} $ y_0 = 0$}
		\For{$t =0,\ldots,\horizon-1$}
		\begin{align*}
		v_t^* = K_t y_t + k_t \qquad \qquad
		y_{t+1} = \Dyn_{t,x}^\top y_t + \Dyn_{t,u}^\top v_t^*
		\end{align*}
		\EndFor
		\Ensure{$\texttt{ILQG\_step}(\bar u, \stepsize) = \bar v^*$ where $\bar v^* = (v_0^*;\ldots;v_{\horizon-1}^*)$ is optimal for~\eqref{eq:LQG_step}}
	\end{algorithmic}
\end{algorithm}

\begin{algorithm}[t]
	\caption{$\texttt{iLQG\_step}(\bar u)$ iLQG step  as presented in \citep{tassa2012synthesis} on a command $\bar u$ \label{algo:ILQG_tassa_step}}
	\begin{algorithmic}[1]
		\Require{Command $\bar u = (u_0;\ldots, u_{\horizon-1})$, step-size $\stepsize$, initial state $\hat x_0$ convex twice differentiable costs $h_t$ for $t = 0,\ldots \horizon$ with $h_0=0$, convex twice differentiable  penalties $g_t$ for $t=0,\ldots, \horizon-1$
			differentiable	dynamics $\dyn_t$ for $t=0,\ldots, \horizon-1 $, decreasing factor $\rho_-<1$, control objective $f$
		}
		\Statex{ \bf Forward pass:}
		\State{Set $x_0 = \hat x_0$}
		\For{$t= 0, \ldots, \horizon-1$}
		\State{Compute and store
			\begin{gather*}
			h_{t,x} = \nabla h_t(x_t), \quad H_{t,xx} = \nabla^2 h_t(x_t), \quad g_{t,u} = \nabla g_t(u_t), \quad G_{t,uu} = \nabla^2 g_t(u_t),  \\
			\Dyn_{t,x} = \nabla_x \dyn_t(x_t, u_t),\quad  \Dyn_{t,u} = \nabla_u \dyn_t(x_t, u_t).
			\end{gather*}
		}
		\State{Go to next state $x_{t+1} = \dyn_t(x_t, u_t) $}
		\EndFor
		\Statex{\bf Backward pass:}
		\State{Initialize $C_{\horizon,xx} = \nabla^2 h_\horizon(x_\horizon), c_{\horizon,x} = \nabla h_\horizon(x_\horizon)$ }
		\For{$t = \horizon-1,\ldots, 0$ }
		\State{
			Define
			\begin{align*}
			w_{t,x} & = h_{t,x} + \Dyn_{t,x} c_{t+1,x} & W_{t,xx} & = H_{t,xx} + \Dyn_{t,x} C_{t+1,xx}\Dyn_{t,x}^\top \\
			w_{t,u} & = g_{t,u} + \Dyn_{t,u} c_{t+1,x} & W_{t,uu} & = G_{t,uu} + \Dyn_{t,u} C_{t+1,xx}\Dyn_{t,u}^\top \\
			& & W_{t,ux} & = \Dyn_{t,u} C_{t+1,xx}\Dyn_{t,x}^\top 
			\end{align*}
		}
		\State{Compute
			\begin{align*}
			C_{t,xx} = W_{t,xx} - W_{t,ux}^\top W_{t,uu}^{-1} W_{t,ux} \qquad  \qquad
			c_{t,x}  = w_{t,x} - W_{t,ux}^\top W_{t,uu} ^{-1} w_{t,u}
			\end{align*}
		}
		\State{Store \begin{align*}
			K_t = - W_{t,uu}^{-1} W_{t,ux} \qquad \qquad
			k_t = - W_{t,uu}^{-1} w_{t,u}
			\end{align*}
		}
		\EndFor
		\Statex{\bf Roll-out pass:}
		\State{\textbf{Initialize} $ x_0^+ = \hat x_0$, $\alpha=1$}
		\Repeat{}
		\For{$t =0,\ldots,\horizon-1$}
		\begin{align*}
			u_t^+ = u_t + K_t (x_t^+-x_t)+ \alpha k_t \qquad \qquad
			\textcolor{puorange}{\bm{x_{t+1}^+ = \dyn_t(x_t^+, u_t^+)}}
		\end{align*}
		\EndFor
		\State{Update $\alpha = \rho_- \alpha$}
		\Until{$f(\bar u^+) \leq f(\bar u)$}
		\Ensure{$\texttt{iLQG\_step}(\bar u) = \bar u^+$ where $\bar u^+ = (u_0^+;\ldots;u_{\horizon-1}^+)$}
	\end{algorithmic}
\end{algorithm}

\begin{algorithm}[t]
	\caption{$\texttt{DDP\_step}(\bar u)$. Differential dynamic programming step on a command $\bar u$ \citep{Tass14}\label{algo:DDP_step} }
	\begin{algorithmic}[1]
		\Statex{\textbf{Hyper-parameters}: regularization $\lambda_0$, increasing regularization factor $\rho^+>1$, decreasing factor $\rho_-<1$}
		\Statex{\textbf{Inputs: }Command $\bar u = (u_0;\ldots; u_{\horizon-1})$, initial state $\hat x_0$ convex twice differentiable costs $h_t$ for $t = 0,\ldots \horizon$ with $h_0=0$, convex twice differentiable  penalties $g_t$ for $t=0,\ldots, \horizon-1$,
			twice differentiable	dynamics $\dyn_t$ for $t=0,\ldots, \horizon-1 $,
			control objective $f$
		}
		\Statex{ \bf Forward pass:}
		\State{Set $x_0 = \hat x_0$}
		\For{$t= 0, \ldots, \horizon$}
		\State{Compute and store
			\shortdisplay{
				\begin{gather*}
				h_{t, x} = \nabla h_t(x_t), \qquad H_{t,xx} = \nabla^2 h_t(x_t), \qquad g_{t, u} = \nabla g_t(u_t), \qquad G_{t,uu} = \nabla^2 g_t(u_t),  \\
				\Dyn_{t,x}  = \nabla_x \phi_t(x_t, u_t), \qquad 
				\Dyn_{t,u}  = \nabla_u \phi_t(x_t, u_t), \qquad
				\\
				\upphi_{t,xx}  = \nabla^2_{xx} \phi_t(x_t, u_t), \qquad 
				\upphi_{t,uu}   = \nabla^2_{uu} \phi_t(x_t, u_t), \qquad 
				\upphi_{t,ux}   = \nabla^2_{ux} \phi_t(x_t, u_t). \qquad
				\end{gather*}}
			\State{Go to next state $ x_{t+1} = \dyn_t(x_t,u_t)$}
		}
		\EndFor
		\State{Compute and store $h_{\horizon, x} = \nabla h(x_\horizon), \quad H_{\horizon, xx} = \nabla^2 h(x_\horizon) $}
		\Statex{\bf Backward pass:}
		\State{Initialize $C_{\horizon,xx} = H_{\horizon,xx}, c_{\horizon,x} = h_{\horizon,x}$}
		\For{$t = \horizon-1,\ldots, 0$ }
		\State{
			Compute
			\shortdisplay{
				\begin{align}
				\begin{array}{l}
				\\
				w_{t,x} = h_{t,x} + \Dyn_{t,x} c_{t+1,x} \\
				w_{t,u} = g_{t,u} + \Dyn_{t,u} c_{t+1,x} \nonumber
				\end{array}
				\qquad
				\begin{array}{l}
				W_{t,xx} = H_{t,xx} + \Dyn_{t,x} C_{t+1,xx}\Dyn_{t,x}^\top + \upphi_{t,xx} [\cdot ,\cdot ,c_{t+1,x}] \\
				W_{t,ux} = \Dyn_{t,u} C_{t+1,xx}\Dyn_{t,x}^\top + \upphi_{t,ux} [\cdot ,\cdot ,c_{t+1,x}] \nonumber\\
				W_{t,uu} = G_{t,uu} + \Dyn_{t,u} C_{t+1,xx}\Dyn_{t,u}^\top + \upphi_{t,uu}[\cdot ,\cdot ,c_{t+1,x}]
				\end{array}
				\end{align}}
		}
		\State{Set $\lambda =\lambda_0, \quad W_{t,uu}^0 = W_{t,uu}$}
		\While{$W_{t,uu} \nsucc 0$}
		\State{$W_{t,uu}:= W_{t,uu}^0 + \lambda \idm_p, \qquad \lambda = \rho^+ \lambda$}
		\EndWhile
		\State{Compute
			\shortdisplay{
				\begin{align*}
				c_{t,x}  = w_{t,x} - W_{t,ux}^\top W_{t,uu} ^{-1} w_{t,u}, \quad
				C_{t,xx}  = W_{t,xx} - W_{t,ux}^\top W_{t,uu}^{-1} W_{t,ux}.
				\end{align*}}
		}
		\State{Store $
			K_t = - W_{t,uu}^{-1} W_{t,ux} \quad
			k_t = - W_{t,uu}^{-1} w_{t,u}
			$.
		}
		\EndFor
		\Statex{\bf Roll-out pass:}
		\State{\textbf{Initialize} $ x_0^+ = \hat x_0$, $\alpha=1$}
		\Repeat{}
		\For{$t =0,\ldots,\horizon-1$}
		\begin{align*}
			u_t^+ = u_t + K_t (x_t^+-x_t)+ \alpha k_t \qquad \qquad
			\textcolor{puorange}{\bm{x_{t+1}^+ = \dyn_t(x_t^+, u_t^+)}}
		\end{align*}
		\EndFor
		\State{Update $\alpha = \rho_- \alpha$}
		\Until{$f(\bar u^+) \leq f(\bar u)$}	
		\Ensure{$\texttt{DDP\_step}(\bar u) = \bar u^+$} where $\bar u^+ = (u_0^+;\ldots;u_{\horizon-1}^+)$
	\end{algorithmic}
\end{algorithm}

\begin{algorithm}
\caption{ILQR\label{algo:ILQR} \citep{Li04}}
\begin{algorithmic}[1]
	\Require{Initial state $\hat x_0$, 	differentiable	dynamics $\dyn_t$ for $t=0,\ldots, \horizon-1 $,  convex twice differentiable costs $h_t$ for $t = 1,\ldots, \horizon$, convex twice differentiable  penalties $g_t$ for $t=0,\ldots, \horizon-1$,
, total cost $f$ on the trajectory as defined in~\eqref{eq:composite_pb}, initial command $\bar u_0$, number of iterations $K$}
	\For{$k =0,\ldots,K$}
	\State{Using Algo.~\ref{algo:ILQR_step}, compute
	$
	\bar v_k = \texttt{ILQR\_step}(\bar u_k, +\infty)
	$}
	\State{Find $\stepsize_k$ s.t. $ f(\bar u_k + \stepsize_k \bar v_k) < f(\bar u_k)$}
	\State{Set $\bar u_{k+1} = \bar u_k + \stepsize_k  \bar v_k$}
	\EndFor
	\Ensure{$\bar u^* = \bar u_K$}
\end{algorithmic}
\end{algorithm}

\begin{algorithm}
	\caption{Regularized ILQR\label{algo:ILQR_reg} as presented in Sec.~\ref{sec:prox_lin}}
	\begin{algorithmic}[1]
		\Require{Initial state $\hat x_0$, 	differentiable	dynamics $\dyn_t$ for $t=0,\ldots, \horizon-1 $,  convex twice differentiable costs $h_t$ for $t = 1,\ldots, \horizon$, convex twice differentiable  penalties $g_t$ for $t=0,\ldots, \horizon-1$,
			, total cost $f$ on the trajectory as defined in~\eqref{eq:composite_pb}, initial command $\bar u_0$, number of iterations $K$
		}
		\For{$k =0,\ldots,K$}
		\State{Find $\gamma_k$, such that 
			$
			\bar u_{k+1} = \bar u_k + \texttt{ILQR\_step}(\bar u_k, \stepsize_k)
			$ computed by Algo.~\ref{algo:ILQR_step}	satisfies
			\[
			f(\bar u_{k+1}) \leq f(\bar u_k) + \frac{1}{2\gamma_k} \|\bar u_k - \bar u_{k+1}\|_2^2
			\]
		}
		\EndFor
		\Ensure{$\bar u^* = \bar u_K$}
	\end{algorithmic}
\end{algorithm}

\begin{algorithm}
	\caption{ILQG\label{algo:ILQG} \citep{Todo05, Li07}}
	\begin{algorithmic}[1]
		\Require{Initial state $\hat x_0$, noisy twice	differentiable	dynamics $\dyn_t$ for $t=0,\ldots, \horizon-1 $,  convex twice differentiable costs $h_t$ for $t = 1,\ldots, \horizon$, convex twice differentiable  penalties $g_t$ for $t=0,\ldots, \horizon-1$,
			, total cost $f$ on the trajectory as defined in~\eqref{eq:composite_pb}, initial command $\bar u_0$, number of iterations $K$}
		\For{$k =0,\ldots,K$}
		\State{Using Algo.~\ref{algo:ILQG_step}, compute
			$
			\bar v_k = \texttt{ILQG\_step}(\bar u_k, +\infty)
			$}
		\State{Find $\stepsize_k$ s.t. $ f(\bar u_k + \stepsize_k \bar v_k) < f(\bar u_k)$}
		\State{Set $\bar u_{k+1} = \bar u_k + \stepsize_k  \bar v_k$}
		\EndFor
		\Ensure{$\bar u^* = \bar u_K$}
	\end{algorithmic}
\end{algorithm}

\begin{algorithm}
	\caption{Regularized ILQG\label{algo:ILQG_reg} as presented in Sec.~\ref{sec:prox_lin}}
	\begin{algorithmic}[1]
		\Require{Initial state $\hat x_0$, 	differentiable	dynamics $\dyn_t$ for $t=0,\ldots, \horizon-1 $,  convex twice differentiable costs $h_t$ for $t = 1,\ldots, \horizon$, convex twice differentiable  penalties $g_t$ for $t=0,\ldots, \horizon-1$,
			, total cost $f$ on the trajectory as defined in~\eqref{eq:composite_pb}, initial command $\bar u_0$, number of iterations $K$
		}
		\For{$k =0,\ldots,K$}
		\State{Find $\gamma_k$, such that 
			$
			\bar u_{k+1} = \bar u_k + \texttt{ILQG\_step}(\bar u_k, \stepsize_k)
			$ computed by Algo.~\ref{algo:ILQG_step}	satisfies
			\[
			f(\bar u_{k+1}) \leq f(\bar u_k) + \frac{1}{2\gamma_k} \|\bar u_k - \bar u_{k+1}\|_2^2
			\]
		}
		\EndFor
		\Ensure{$\bar u^* = \bar u_K$}
	\end{algorithmic}
\end{algorithm}
\clearpage

\section{Differential Dynamic Programming interpretation}\label{app:approx_dyn_prog}
A characteristic of Differential Dynamic Programming is that the update pass follows the original trajectory. This little difference makes it very different to the classical optimization schemes we presented so far. Though its convergence is often derived as a Newton's method,  it was shown that in practice it outperforms Newton's method~\citep{Liao91, Liao92}. We analyze it as an optimization procedure on the state variables using recursive model-minimization schemes.

\subsection{Approximate dynamic programming}
We consider problems in the last control variable since any optimal control problem can be written in the form~\eqref{eq:last_state_ctrl} by adding a dimension in the states. We write then Problem~\eqref{eq:last_state_ctrl} as a constrained problem of the form
\begin{align}
\begin{split}
\min_{\substack{x_\horizon \in D_\horizon}} \quad & h_\horizon(x_\horizon) 
\end{split}
\end{align}
where the constraint sets $D_t$ are defined recursively as 
\begin{align*}
D_0 & = \{x_0\} \\
D_{t+1} & = \{x_{t+1} : x_{t+1} = \phi_{t}(x_t,u_t), \: x_t \in D_t \:, u_t \in \reals^{p}\}, \qquad \mbox{for $t = 0,\ldots, \horizon-1$}.
\end{align*}

The approximate dynamic approach consists then as a nested sequence of subproblems that attempt to make an approximate step in the space of the last state. Formally, at a given iterate $\hat x_\horizon$ defined by $(\hat u_0, \ldots, \hat u_{\horizon-1})$, it considers a model-minimization step i.e.
\begin{equation}\label{eq:first_sub_pb}
\min_{z\in D_\horizon} V_\horizon(z) := m_{h_\horizon}(z; \hat x_\horizon) + \stepsize^{-1} d(z,\hat x_\horizon),
\end{equation}
where $m_{h_\horizon}(\cdot; \hat x_t)$ is a given model that approximates $h_\horizon$ around $\hat x_\horizon$, $d(\cdot, \hat x_\horizon)$  is a proximal term and $\stepsize$ is the step-size of the procedure.

Then the procedure consists in considering recursively model-minimizations steps of functions $V_t$, for $t=\horizon,\ldots,1$, where each model-minimization step introduces the minimization of a new value function $V_t$ on a simpler constraint space.

Formally, assume that at time $t$ the problem considered is
\begin{equation}
\min_{z\in D_t} V_t(z)
\end{equation}
for a given function $V_t$ and that one is given an initial point $\hat z_t \in D_t$ with associated sub-command $\hat v_0, \ldots, \hat v_{t-1} $ that defines states $\hat z_0, \ldots, \hat z_{t}$ as $\hat z_{s+1} = \phi_s(\hat z_s, \hat v_s)$ for $0 \leq s\leq t-1$ with $\hat z_0 = \hat x_0$. Then developing the constraint set, the problem reads
\begin{equation}
\min_{z\in D_{t-1}, v \in \reals^p} M_t(z,v): = V_t(\phi_{t-1}(z,v))
\end{equation}
a minimization step on this problem around the given initial point is
\begin{equation}
\min_{z\in D_{t-1}, v \in \reals^p} m_{M_t}(z, v ; \hat z_{t-1},\hat v_{t-1}) + \stepsize^{-1} d((z,v), (\hat z_{t-1},\hat v_{t-1})).
\end{equation}
Then this problem simplifies as
\begin{equation}
\min_{z\in D_{t-1}} V_{t-1}(z) := \min_{v\in \reals^p} m_{M_t}(z, v ; \hat z_{t-1}, \hat v_{t-1}) + \stepsize^{-1} d((z,v), (\hat z_{t-1}, \hat v_{t-1})),
\end{equation}
which defines the next problem. The initial point of this subproblem is chosen as $\hat z_{t-1}$ with associated subcommand $\hat v_0, \ldots, \hat v_{t-2}$.

The recursive algorithm is defined in Algo.~\ref{algo:approx_dyn_prog_recursion}. These use sub-trajectories defined by the dynamics and sub-commands. 
The way the stopping criterion and the step-sizes are chosen depend on the implementation just as the choices of the model $m$ and the proximal term $d$. The optimal command is tracked along the recursion to be output at the end.

\begin{algorithm}[t]
	\begin{algorithmic}[1]
		\State{\textbf{Inputs:}}
		
		\State{-Approximate model $m$, proximal term $d$, initial point $\hat x_0$}
		
		\State{-Time $t$, value function $V_t$}
		\State{-Dynamics $\phi_0,\ldots, \phi_{t}$, initial point $z_t^0\in D_t$ with associated subcommand $v_0^0, \ldots, v_{t-1}^0$ defining states $z_0^0, \ldots, z_{t}^0$ as $z_{s+1}^0 = \phi_s(z_s^0, v_s^0)$ for $0 \leq s\leq t-1$ with $z_0^0 = \hat x_0$}
		\State{-Step sizes $(\stepsize_{t}^k)_{k\in \mathbb{N}}$, stopping criterion $\delta_t:\mathbb{N} \times \reals^d \rightarrow \{0,1\}$}
		
		\If{$t=0$}
		\State{Return $z^* =  \hat x_0$}
		\Else{}
		
		\Repeat{ for $k = 1, \ldots$}
		\State{Denoting $M_t(z,v) = V_t(\phi_{t-1}(z,v))$ and $\hat z = z^{k-1}_{t-1}$, $\hat v = v^{k-1}_{t-1}$, define
			\begin{equation}\label{eq:value_function_algo}
			V_{t-1}(z) = \min_{v\in \reals^p} m_{M_t}(z, v ; \hat z, \hat v) + (\stepsize_{t}^k)^{-1} d((z,v), ( \hat z, \hat v))
			\end{equation}	
		}
		\State{	Find $z^k_{t-1}$ and its associated subcommand $v_0^k, \ldots, v_{t-2}^k$ and subtrajectory $z_0^k, \ldots, z_{t-2}^k$ using Algo.~\ref{algo:approx_dyn_prog_recursion} s.t.
			\[
			z^k_{t-1} \approx \argmin_{y\in D_{t-1}} V_{t-1}(y)
			\]
			fed with
			\begin{itemize}
				\item[-] same $m, d, \hat x_0$ 
				\item[-] time $t-1$, value function $V_{t-1}$
				\item[-] dynamics $\phi_0,\ldots, \phi_{t-1}$, initial point $z^{k-1}_{t-1}$ with associated subcommand $v_0^{k-1}, \ldots, v_{t-2}^{k-1}$
				\item[-] a strategy of step sizes $(\stepsize_{t-1}^k)_{k\in \mathbb{N}}$ and a stopping criterion $\delta_{t-1}:\mathbb{N} \times \reals^d \rightarrow \{0,1\}$
			\end{itemize}
		}
		\State{Compute 
			\begin{equation}\label{eq:opt_control_algo}
			v^k_{t-1} = \argmin_{v\in \reals^p} m_{M_t}(z^k, v ; \hat z, \hat v) + (\stepsize_{t}^k)^{-1} d((z^k,v), (\hat z, \hat v))
			\end{equation}
		}
		\State{Compute $z^k_t = \phi_{t-1}(z^k_{t-1}, v^{k}_{t-1})$}
		\Until{stopping criterion $\delta_t(k,z^k_t)$ is  met}
		\State{Return $z^k_t$ with its associated subcommand $v_0^k, \ldots, v_{t-1}^k$}
		\EndIf
		
		\caption{Approximate Dynamic Programming Recursion}
		\label{algo:approx_dyn_prog_recursion}	
	\end{algorithmic}
\end{algorithm}

The whole procedure instantiates iteratively~Algo.~\ref{algo:approx_dyn_prog_recursion} on~\eqref{eq:first_sub_pb} as presented in Algo.~\ref{algo:approx_dyn_prog}. Note that it is of potential interest to have a different model-minimization scheme for the outer loop and the inner recursive loop. 
\begin{algorithm}[t]
	\begin{algorithmic}
		\State{\textbf{Inputs:}}
		
		\State{-Cost function $h_\horizon$, outer approximate model $m$, outer proximal term $d$, inner approximate model $\tilde m$ and inner proximal term $\tilde d$}
		
		\State{-Dynamics $\phi_0,\ldots, \phi_{\horizon-1}$, initial point $x_\horizon^0\in D_\horizon$ with associated command $u_0^0, \ldots, u_{\horizon-1}^0$ defining states $x_0^0, \ldots, x_{\horizon}^0$ as $x_{t+1}^0 = \phi_t(x_t^0, u_t^0)$ for $0\leq t\leq \horizon-1$ with $x_0^0 = \hat x_0$}
		
		\State{-Step sizes $(\stepsize^k)_{k\in \mathbb{N}}$, stopping criterion $\delta:\mathbb{N} \times \reals^d \rightarrow \{0,1\}$}
		
		\Repeat{ for $k = 1, \ldots$}
		\State{
			Find $x_\horizon^k$ with associated command $u_0^k, \ldots, u_{\horizon-1}^k$ using Algo.~\ref{algo:approx_dyn_prog_recursion} to solve
			\begin{equation}
			x_\horizon^k \approx \argmin_{z\in D_\horizon} V_\horizon(z) := m_{h_\horizon}(z;  x_\horizon^{k-1}) + (\stepsize^k)^{-1} d(z, x_\horizon^{k-1})
			\end{equation}
		}
		\State{fed with
			
			\begin{itemize}
				\item[-] Inner approximate model $\tilde m$ and proximal term $\tilde d$, initial point $\hat x_0$
				\item[-] time $\horizon$, value function $V_{\horizon}$
				\item[-] initial point $x_\horizon^{k-1}$ with associated command $u_0^{k-1}, \ldots, u_{\horizon-1}^{k-1}$
				\item[-] a strategy of step sizes $(\stepsize_{\horizon}^k)_{k\in \mathbb{N}}$ and  a stopping criterion $\delta_{\horizon}:\mathbb{N} \times \reals^d \rightarrow \{0,1\}$
			\end{itemize}
		}		
		\Until{Stopping criterion $\delta(k, x^k_\horizon)$ is met}
		\Ensure{$x^k_\horizon$ with associated command $u_0^k, \ldots, u_{\horizon-1}^k$}
		\caption{Approximate Dynamic Programming}
		\label{algo:approx_dyn_prog}	
	\end{algorithmic}
\end{algorithm}

\subsection{Differential dynamic programming}
Differential dynamic programing is an approximate instance of the above algorithm where (i) one considers a second order approximation of the function to define the model $m$, (ii) one does not use a proximal term $d$, (iii) the stopping criterion is simply to stop after one iteration. 

Precisely, for a twice differentiable function $f$, on a point $\hat w$, we use
$
m_f(w;\hat w) = q_f(w; \hat w ) = f(\hat w) + \nabla f(\hat w)^\top (w-\hat w) + \frac{1}{2} (w-\hat w)^\top \nabla^2 f(\hat w) (w-\hat w)
$.
Notice that without additional assumption on the Hessian $\nabla^2 f(\hat w)$, $q_f(\cdot;\hat w)$ may be unbounded below, such that the model-minimization steps may be not well defined. The definition of the models $q_{M_t}$ in Eq.~\eqref{eq:value_function_algo} correspond to the computations in Eq.~\eqref{eq:ddp_formulation} that lead to the formulation of the cost-to-go functions $c_t$. The solutions output by the recursion in Eq.~\eqref{eq:opt_control_algo} correspond to the roll-out presented in Eq.~\eqref{eq:ddp_rollout}. Crucially, as in the classical DDP formulation, the output at the $t$\textsuperscript{th} time step in the roll-out phase (here when the recursion is unrolled line 13 in Algo.~\ref{algo:approx_dyn_prog_recursion}) is given by the true trajectory. 

Recall that the implementation differs from the motivation. The choice of using the un-shifted cost-to-go functions, i.e., choosing $c_{t+1}$ instead of $\tilde c_{t+1}$ as presented in Sec.~\ref{sec:ddp}, is not explained by our theoretical approach. 

The iLQG method as presented in~\cite{tassa2012synthesis} follows the same approach except that they use the quadratic models defined in a Levenberg-Marquardt steps for each model-minimization of the recursion.

\section{Regularized Gauss-Newton analysis}\label{app:prox_lin}
For completeness we recall how equality~\eqref{eq:explicit_GN} is obtained. As $h, g$ are quadratics, we have 
$h(\bar x+ \bar y) = q_h(\bar x+\bar y;\bar x), g(\bar u+\bar v) = q_g(\bar u +\bar v; \bar u)$. Therefore $c_f(\bar u + \bar v; \bar u) = q_f(\bar u + \bar v; \bar u)$ with $c_f$ defined in~\eqref{eq:cvx_model} and $q_f$ defined in \eqref{eq:quad_model}. The regularized Gauss-Newton step reads then, denoting $\bar x_k = \tilde x(\bar u_k)$, $H = \nabla^2h(\bar x_k)$ and $G = \nabla^2 g(\bar u_k)$
\begin{align*}
	\bar u_{k+1} & = \bar u_k + \argmin_{\bar v} q_f(\bar u_k + \bar v; \bar u_k) + \frac{1}{2\stepsize_k}\|\bar v\|_2^2 \\
	& =  \bar u_k 
	+ \argmin_{\bar v} \Big\{ \nabla h(\bar x_k)^\top (\nabla \tilde x(\bar u_k)^\top \bar v) 
	+ \frac{1}{2} (\nabla \tilde x(\bar u_k)^\top \bar v)^\top H (\nabla \tilde x(\bar u_k)^\top \bar v) \\
	& \phantom{= \bar u_k 
		+ \argmin_{\bar v} \Big\{} + \nabla g(\bar u_k)^\top \bar v + \frac{1}{2}\bar v^\top G \bar v
	+ \frac{1}{2\stepsize_k}\|\bar v\|_2^2 \Big\} \\
	&  = \bar u _k - (\nabla \targetfunc({\bar u}_k) H \nabla\targetfunc({\bar u}_k)^\top + G + \stepsize_k^{-1}\idm_{\horizon p})^{-1} (\nabla \tilde x(\bar u_k) \nabla h(\bar x_k) + \nabla g(\bar u_k)) \\
	& = \bar u_k - (\nabla \targetfunc({\bar u}_k) H \nabla\targetfunc({\bar u}_k)^\top + G + \stepsize_k^{-1}\idm_{\horizon p})^{-1} \nabla \optimobj ({\bar u}_k)
\end{align*} 
We prove the overall convergence of the regularized Gauss-Newton method under a sufficient decrease condition.
\proxlinconv*
\begin{proof}
	For $k \geq 0$,
	\begin{align*}
	f(\unknown_{k}) = c_f(\unknown_{k}; \unknown_{k}) &
	\stackrel{(\star)}{\geq} c_f(\unknown_{k+1}; \unknown_k)+ \frac{1}{\stepsize_k}\|\unknown_{k+1}-\unknown_k\|_2^2
	 \stackrel{\eqref{eq:suff_decrease}}{\geq} \optimobj(\unknown_{k+1}) + \frac{1}{2\gamma_k}\|\unknown_{k+1}-\unknown_k\|^2,
	\end{align*}
	where we used in $(\star)$ the definition of $\unknown_{k+1}$ and strong convexity of $\unknown\rightarrow c(\unknown;\unknown_k) + (2\stepsize_k)^{-1}\|\unknown-\unknown_k\|_2^2$. This ensures first that the iterates stay in the initial level set. Then, summing the inequality and taking the minimum gives 
	\[
	\min_{k=0,\ldots,N} \stepsize_k^{-1}\norm{\unknown_{k+1} - \unknown_{k}}^2 \leq \frac{2 (\optimobj(\unknown_0) - \optimobj^*) }{N+1} .
	\]
	Finally using~\eqref{eq:explicit_GN}, we get 
	\[
	\|\nabla \optimobj (\unknown_k)\|_2 \leq (\ell_{\targetfunc,S}^2L_h  + L_g + \stepsize_k^{-1})\|\unknown_{k+1} - \unknown_k\|_2.
	\]
	Plugging this in previous inequality and rearranging the terms give the result.
\end{proof}

Now we show how the model approximates the objective up to a quadratic error for exact dynamics
\suffdecrease*
\begin{proof}
	As $\targetfunc$ has continuous gradients, it is $\ell_{\targetfunc,C}$-Lipschitz continuous and has $L_{\targetfunc,C}$- Lipschitz gradients on $C \subset \reals^{\horizon p}$. Similarly $\ctrlobj$ is $L_\ctrlobj$-smooth and $\ell_{\ctrlobj,C'}$ on any compact set $C' \subset \reals^{\horizon d}$. Now on $C\subset \reals^{\horizon p}$, denote $B \subset \reals^{\horizon p}$ a ball centered at the origin that contains $C$ and $\rho$ its radius. Define $B' \subset \reals^{\horizon d}$ a ball centered at the origin of radius $2\rho \ell_{\targetfunc,C}$ and finally $C' = \targetfunc(C) + B'$ such that for any  $\unknown, \bar v \in C$, $\targetfunc(\unknown) + \nabla \targetfunc(\unknown)^\top (\bar v- \unknown) \in C'$. Then for any  $\unknown, \bar v \in C$, 
	\begin{align*}
	|\optimobj(\bar v) - c_f(\bar v; \unknown) | & = |h(\tilde x(\bar v)) - h\left(\tilde x(\bar u ) + \nabla \tilde x (\bar u)^\top (\bar v - \bar u) \right)| \\
	& \leq \ell_{\ctrlobj,C'} \|\tilde x(\bar v) - \tilde x(\bar u ) - \nabla \tilde x (\bar u)^\top (\bar v - \bar u)\|_2 \\
	& \leq \frac{\ell_{\ctrlobj,C'} L_{\tilde x,C}}{2} \|\bar v- \bar u\|_2^2
	\end{align*}
	where the last line uses $\tilde x(\bar v) = \tilde x(\bar u) + \int_{0}^{1}\nabla \tilde x(\bar u + s(\bar v -\bar u))^\top (\bar v -\bar u) ds $ and the smoothness of $\tilde x$ on $C$.
\end{proof}

Finally we precise a minimal step-size for which the sufficient decrease condition is ensured.
\minstep*
\begin{proof}
	Using \eqref{eq:explicit_GN}, 
	\[
	\|\unknown_{k+1} - \unknown_{k}\|_2 \leq \stepsize_k \|\nabla \optimobj(\unknown_{k})\|_2
	\]
	so for $\stepsize_k \leq \ell_{\optimobj, S}^{-1}$, $\|\unknown_{k+1} - \unknown_{k}\|_2 \leq 1$ and $\unknown_{k+1} \in C$. As $\unknown_k, \unknown_{k+1} \in C$ we have by \eqref{eq:quad_error}, 
	\[
	\optimobj(\unknown_{k+1}) \leq c(\unknown_{k+1};\unknown_{k}) + \frac{M_C}{2} \|\unknown_{k} -\unknown_{k+1}\|_2^2
	\]
	which is the sufficient decrease condition~\eqref{eq:suff_decrease} for $\stepsize_k \leq M_{C}^{-1}$. 
\end{proof}

We rigorously define the back-tracking line-search that supports Cor.~\ref{cor:line_search} in Algo.~\ref{algo:backtrack} and~\ref{algo:prox_lin_backtrack}.

\begin{algorithm}[t]
	\begin{algorithmic}
		\Require{Objective $f$ as in~\eqref{eq:composite_pb}, convex models $c_f$ as in~\eqref{eq:cvx_model}, point $\bar u$, step size $\stepsize>0$, regularized Gauss-Newton oracle $ \operatorname{GN}(\bar u; \stepsize) \triangleq \argmin_{\bar v \in \reals^{\horizon p}}c_f(\bar v;\bar u) +\frac{\stepsize^{-1}}{2} \|\bar v-\bar u\|_2^2$,  decreasing factor $\rho<1$}
		\While{$f(\operatorname{GN}(\bar u; \stepsize))  > c_{f}(\operatorname{GN}(\bar u; \stepsize);\bar u) + \frac{\stepsize}{2}\|\operatorname{GN}(\bar u; \stepsize) - \bar u\|_2^2$}
		\State{$\stepsize := \rho \stepsize$}			
		\EndWhile
		\Ensure{$\bar u_{+}  := \operatorname{GN}(\bar u; \stepsize), \stepsize_+ = \stepsize$}
		\caption{Line-search for regularized Gauss-Newton method $\mathcal{L}(\bar u,\stepsize) = (\bar u_{+}, \stepsize_+)$}
		\label{algo:backtrack}	
	\end{algorithmic}
\end{algorithm}

\begin{algorithm}[t]
	\begin{algorithmic}
		\Require{Objective $f$ as in~\eqref{eq:composite_pb}, convex models $c_f$ as in~\eqref{eq:cvx_model}, initial point $\bar u_0$, initial step size $\stepsize_{-1}$, accuracy $\epsilon$}
		\Repeat{ for $k = 0, \ldots $}
		\State{Compute $\bar u_{k+1}, \stepsize_{k} = \mathcal{L}(\bar u_k,\stepsize_{k-1})$ using Algo.~\ref{algo:backtrack} such that 
			\[
			f(\bar u_{k+1})  \leq c_{f}(\bar u_{k+1};\bar u_{k}) + \frac{\stepsize_k}{2}\|\bar u_{k+1} - \bar u\|_2^2
			\]
		}
		\Until{$\epsilon$-near stationarity, i.e., $\norm{\nabla f(\bar u_{k+1})} \leq \epsilon$  }
		\Ensure{$\bar u_{k+1}$}
		\caption{Regularized Gauss-Newton method with line-search}
		\label{algo:prox_lin_backtrack}	
	\end{algorithmic}
\end{algorithm}

\section{Accelerated Gauss-Newton}\label{app:acc_proxlin}
We detail the proof of convergence of the accelerated Gauss-Newton algorithm.
\acc*
\begin{proof}
	First part of the statement is ensured by taking the best of both steps. For the second part, note first that assumption~\eqref{eq:lower_bound} implies that the objective $f$ is convex as shown in Lemma 8.3 in~\citep{Drus16}. Now, at iteration $k\geq 1$, for any $\bar u$,
	\begin{align*}
		f(\bar u_k) & \stackrel{\eqref{eq:best_of}}{\leq} f(\bar w_k) \\
		 & \stackrel{\eqref{eq:suff_decrease_acc}}{\leq} c_f(\bar w_k; \bar y_k ) + \frac{\accstepsize^{-1}_k}{2}\|\bar w_k - \bar y_k\|_2^2 \\
		& \stackrel{(\star)}{\leq} c_f(\bar u; \bar y_k) + \frac{\accstepsize^{-1}_k}{2}(\|\bar u -\bar y_k\|_2^2 - \|\bar u -\bar w_k\|_2^2) \\
		& \stackrel{\eqref{eq:lower_bound}}{\leq} f(\bar u)+ \frac{\accstepsize^{-1}_k}{2}(\|\bar u -\bar y_k\|_2^2 - \|\bar u -\bar w_k\|_2^2),
	\end{align*}
	where $(\star)$ comes from strong convexity of
	$\bar u\rightarrow c_f(\bar u;\bar y_k) + \accstepsize_k^{-1}\|\bar u-\bar y_k\|_2^2/2$ and the fact that $\bar w_k$ minimizes it.
	Now choosing $\bar u = \alpha_k \bar u^* + (1-\alpha_k) \bar u_{k-1}$, such that $\bar u -\bar y_k = \alpha_k(\bar u^* - \bar z_{k-1})$  and $\bar u - \bar w_k = \alpha_k(\bar u^* - \bar z_k)$, we get by convexity of $f$, 
	\begin{align*}
		f(\bar u_k) \leq & \alpha_k f(\bar u^*) + (1-\alpha_k) f(\bar u_{k-1})  + \frac{\alpha_k^2\accstepsize^{-1}_k}{2}(\|\bar u^* - \bar z_{k-1}\|_2^2 - \|\bar u^* - \bar z_k\|_2^2).
	\end{align*}
	Subtracting $f^*$ on both sides and rearranging the terms, we get
	\begin{align*}
		f(\bar u_k) -f^* \leq &(1-\alpha_k) f(\bar u_{k-1} -f^*)  + \frac{\alpha_k^2\accstepsize^{-1}_k}{2}(\|\bar u^* - \bar z_{k-1}\|_2^2 - \|\bar u^* - \bar z_k\|_2^2).
	\end{align*}
For $k=1$, using that $\alpha_1=1$, we get
\[
\frac{\accstepsize_1}{\alpha_1^2}\left(f(\bar u_1) -f^*\right)   \leq  \frac{1}{2} \left( \|\bar u^* - \bar z_{0}\|_2^2 - \|\bar u^* - \bar z_1\|_2^2 \right).
\]
	For $k\geq 2$, using the definition of $\alpha_k$, i.e., that $(1- \alpha_k)/\alpha_k^2 = 1/\alpha_{k-1}^2$, we get
	\begin{flalign*}
	\frac{\accstepsize_k}{\alpha_k^2}\left(f(\bar u_k) -f^*\right)   \leq & \frac{\accstepsize_k}{\alpha_{k-1}^2}\left(f(\bar u_{k-1}) -f^*\right)   + \frac{1}{2} \left( \|\bar u^* - \bar z_{k-1}\|_2^2 - \|\bar u^* - \bar z_k\|_2^2 \right) \\
	\leq & \frac{\accstepsize_{k-1}}{\alpha_{k-1}^2}\left(f(\bar u_{k-1}) -f^*\right)   + \frac{1}{2} \left( \|\bar u^* - \bar z_{k-1}\|_2^2 - \|\bar u^* - \bar z_k\|_2^2 \right).
	\end{flalign*}
	Developing the recursion, we obtain
	\begin{align*}
	f(\bar u_k) -f^* \leq  \frac{\alpha_k^2 \delta_k^{-1}}{2} \|\bar u^* - \bar z_0\|_2^2  \leq  \frac{4\accstepsize^{-1}}{(k+1)^2} \|\bar u^* - \bar u_0\|_2^2,
	\end{align*}
	where $\accstepsize = \min_{k \in \{1,\ldots N\}} \accstepsize_k$ and we used the estimate on $\alpha_k$ provided in Lemma B.1 in~\citep{Paqu17}.
\end{proof}

\end{document}